\documentclass[11pt,reqno]{amsart}
\usepackage{amsmath,amsfonts,amssymb,amsthm}
\allowdisplaybreaks[4]
\numberwithin{equation}{section}
\usepackage[margin=1in]{geometry}
\usepackage{newtxtext,newtxmath}
\usepackage{microtype}
\usepackage[numbers,sort&compress]{natbib}
\usepackage[hidelinks]{hyperref}
\hypersetup{
  pdftitle={Segregated Bubbling Solutions for a Critical Schr\"odinger System of Brezis--Nirenberg Type with Sublinear Competitive Coupling},
  pdfauthor={Qing Guo and Chengxiang Zhang}
}

\newtheorem{Thm}{Theorem}[section]
\newtheorem{Lem}[Thm]{Lemma}
\newtheorem{Cor}[Thm]{Corollary}
\newtheorem{Prop}[Thm]{Proposition}
\newtheorem{Def}[Thm]{Definition}
\newtheorem{Rem}[Thm]{Remark}

\title[Segregated bubbling solutions for a critical Schr\"odinger system]
{Segregated Bubbling Solutions for a Critical Schr\"odinger System
of Brezis--Nirenberg Type with Sublinear Competitive Coupling}
\author{Qing Guo}
\thanks{Corresponding author.}
\address[]{College of Science, Minzu University of China, Beijing 100081, China}
\email{guoqing0117@163.com}
\author{Chengxiang Zhang}
\address[]{Laboratory of Mathematics and Complex Systems (Ministry of Education),
    School of Mathematical Sciences,
    Beijing Normal University, Beijing 100875, China}
\email{zcx@bnu.edu.cn}

\begin{document}
\begin{abstract}
We construct segregated bubbling solutions for a two-component critical
Schr\"odinger system of Brezis--Nirenberg type in a smooth bounded domain
$\Omega\subset\mathbb R^N$, $N\geq5$, with any fixed competitive coupling
$\beta<0$. Suppose that the Robin function has two distinct prescribed
critical points, each satisfying a local degree condition. For every
sufficiently small $\epsilon>0$, the system admits a nonnegative weak solution
with both components nontrivial. Each component has a single-bubble profile
and concentrates at one of the prescribed points. Each component also
vanishes identically in a ball centered at the other concentration point;
after rescaling by the natural bubble length, the radius of this ball tends
to infinity.

\par\smallskip
The main obstruction is that $p=N/(N-2)\in(1,2)$, so the gradient of the
interaction potential $(s,t)\mapsto |s|^p|t|^p$ is not differentiable when
one component vanishes and the other is nonzero. Hence the usual global
Lyapunov--Schmidt reduction cannot be applied directly. We first solve a
nonlinear exterior problem variationally. The resulting dead cores remove the
cross-component coupling from the inner localization regions, where a
projected critical reduction can then be carried out. The remaining scale and
center equations are solved by Brouwer degree theory.
\end{abstract}
\maketitle
\section{Introduction}
\subsection{Background and main contributions}

We study the following two-component critical nonlinear Schr\"odinger system:
\begin{equation}\label{eq0}
\begin{cases}
-\Delta u_{1} = \epsilon u_1 + \lvert u_{1}\rvert^{2^{*}-2} u_{1} + \beta \lvert u_{2}\rvert^{\frac{2^{*}}{2}} \lvert u_{1}\rvert^{\frac{2^{*}}{2}-2} u_{1}, & \text{in } \Omega, \\[3mm]
-\Delta u_{2} = \epsilon u_2 + \lvert u_{2}\rvert^{2^{*}-2} u_{2} + \beta \lvert u_{1}\rvert^{\frac{2^{*}}{2}} \lvert u_{2}\rvert^{\frac{2^{*}}{2}-2} u_{2}, & \text{in } \Omega, \\[3mm]
u_{1} = u_{2} = 0, & \text{on } \partial\Omega, \\[3mm]
u_{1}, u_{2} \geq 0, \quad u_1,u_2\not\equiv0, \quad u_1, u_2 \in H_0^1(\Omega).&
\end{cases}
\end{equation}
Here $\Omega$ is a smooth bounded domain in $\mathbb{R}^N$, $N\geq5$, and
$2^*=2N/(N-2)$. Throughout the paper we set
\[
p:=\frac{2^*}{2}=\frac{N}{N-2}\in(1,2)
\]
and use the convention $|t|^{p-2}t=0$ at $t=0$. Solutions are understood in
the weak $H_0^1(\Omega)\times H_0^1(\Omega)$ sense.

Systems of the form \eqref{eq0} arise from standing-wave reductions
\[
\phi_k(t,x)=e^{-i\lambda_k t}u_k(x)
\]
of multicomponent Gross--Pitaevskii equations. The sign of $\beta$ determines
the interaction between the components. Cooperative interactions ($\beta>0$)
occur, for example, in models from nonlinear optics, such as the propagation
of self-trapped mutually incoherent wave packets \cite{Akhmediev99}, whereas
competitive interactions ($\beta<0$) arise in models of multicomponent
Bose--Einstein condensates \cite{Timmermans98}. Mathematically,
\eqref{eq0} is a weakly coupled gradient system: the principal part is
diagonal, while the components interact through lower-order nonlinear terms.
It is also a critical Schr\"odinger system of Brezis--Nirenberg type.

\medskip
We recall the scalar Brezis--Nirenberg problem
\begin{equation}\label{eq:BN}
-\Delta u=u^{\frac{N+2}{N-2}}+\lambda u\quad\text{in }\Omega,
\qquad u=0\quad\text{on }\partial\Omega.
\end{equation}
Brezis and Nirenberg \cite{Brezis83} established the fundamental existence
result for \eqref{eq:BN}. The blow-up analyses of Han \cite{Han91} and Rey
\cite{Rey89} showed that the location of a concentrating family is governed by
the Robin function of the domain. Conversely, Rey \cite{Rey90} constructed
single-bubble solutions near nondegenerate critical points of the Robin
function when $N\geq5$, and Musso and Pistoia \cite{Musso02} subsequently
constructed multi-bubble solutions. These works initiated an extensive study
of bubbling phenomena for critical elliptic problems; see, for example,
\cite{Ayed06,Bartsch06,Castro03,Iacopetti15,Micheletti04,Pistoia13} and the
references therein.

For comparison, consider the more general weakly coupled system
\begin{equation}\label{eqgeneral}
-\Delta u_i=\mu_i|u_i|^{2p-2}u_i
+\beta |u_i|^{p-2}u_i\sum_{j\ne i}|u_j|^p+\nu_i u_i,
\qquad i=1,\ldots,m,
\end{equation}
with $2p\leq2^*$. In the Sobolev subcritical case $2p<2^*$, positive and
least-energy solutions have been widely studied, beginning with
\cite{Ambrosetti07,LinWei05,Maia06}; see also
\cite{Correia16,Mandel15,OliveiraTavares2016,Soave,ST}. In the critical case,
Chen and Zou \cite{ChenZou12} treated the four-dimensional cubic system, and
subsequent works addressed least-energy solutions, phase separation,
sign-changing solutions, and concentration phenomena; see
\cite{ChenLin14,ChenLin15,ChenZou15,GuoLiWei14}. Pistoia and Tavares
\cite{Pistoia2017} constructed, in dimension four, families whose components
concentrate at different points. Related blow-up constructions for other
critical elliptic systems analogous to the Brezis--Nirenberg problem were
obtained by Kim and Pistoia \cite{KimPistoia21}.

Following the terminology in \cite{GuoPistoiaWen}, we call a bubbling family
\emph{segregated} when the limiting concentration sets of different components
are disjoint. In the present two-component, one-bubble setting, this means
\[
\xi_{1,\epsilon}\longrightarrow\bar\xi_1,
\qquad
\xi_{2,\epsilon}\longrightarrow\bar\xi_2,
\qquad
\bar\xi_1\ne\bar\xi_2.
\]
The term refers to the concentration centers; it does not assert that the
supports of the two components are disjoint for each fixed $\epsilon$.

The main contribution of this paper is the construction, for every fixed
$\beta<0$ and every $N\geq5$, of a family of nonnegative segregated bubbling
solutions with both components nontrivial. The two components concentrate,
respectively, at two distinct prescribed critical points of the Robin function,
each satisfying the local degree condition introduced below. In
addition, we establish a quantitative dead-core estimate: on the natural
bubbling scale, each component vanishes identically in a ball centered at the
other component's concentration center, and the radius of this ball tends to
infinity as $\epsilon\to0$.

The key analytical obstruction is the lack of differentiability of the
coupling along the coordinate axes. Indeed, since $N\geq5$,
\[
p=\frac{N}{N-2}\in(1,2).
\]
The interaction potential $(s,t)\mapsto |s|^p|t|^p$ is of class $C^1$, but
its gradient is not differentiable at points $(s,0)$ with $s\ne0$ or $(0,t)$
with $t\ne0$. The segregated profiles constructed here meet exactly this
configuration: near either bubble, on the bubbling scale, one component is
active while the other vanishes on an open neighborhood. Consequently, the
coupling does not admit the bounded linearization required by a direct global
Lyapunov--Schmidt reduction around two projected bubbles. For synchronized
profiles, by contrast, the two components remain comparable and vanish at the
same rate, so the potentially singular factor is compensated in the
linearized coefficients. Related critical constructions in pierced domains
were obtained by Pistoia and Soave \cite{PistoiaCoron} in dimensions three and
four, where $p\geq2$ and this particular differentiability obstruction does
not occur.

Our strategy is an inner--outer decomposition adapted to this loss of
differentiability along the coordinate axes. Set
\[
s_\epsilon:=\epsilon^{\frac1{N-4}},
\qquad
\Omega_\epsilon=s_\epsilon^{-1}\Omega.
\]
In the rescaled domain the bubbles have unit scale, whereas the mutual distance
of the two centers and their distance from the boundary are of order
$s_\epsilon^{-1}$. Around each center we use a smaller inner ball, on which the local data
are prescribed, and a larger localization ball. The larger balls contain the
cut-off transition regions and the supports of the projected inner equations.

The exterior analysis produces, for every sufficiently small $\varrho>0$, a
dead-core ball of radius
\[
R_{\mathrm{dc},\epsilon}(\varrho)
:=\epsilon^{-\frac{1-\varrho}{N-2}}
\]
around the opposite concentration center. The intermediate radii are chosen so
that the larger localization balls lie well inside these dead-core balls,
which in turn remain much smaller than the center-separation scale
$s_\epsilon^{-1}$. Consequently, the opposite component vanishes throughout
each inner localization region, the coupling drops out of the projected inner
equations, and the remaining scale and location parameters can be selected by
a Brouwer degree argument. The precise radii, the full scale hierarchy, and
the dimension-dependent auxiliary range needed to realize this geometry are
specified at the beginning of Section~2.

An exterior minimization was also used in our previous subcritical work
\cite{GuoZhangCV}, but the critical problem is not obtained by simply replacing
subcritical profiles with critical bubbles. In the Brezis--Nirenberg regime,
the bubble tails are algebraic, the Dirichlet projection produces a boundary
contribution at the same order as the small linear term, and the concentration
points are selected by the Robin function. These features require critical
projection estimates and a finite-dimensional degree calculation in addition
to the exterior analysis.

There is also a more fundamental distinction from classical inner--outer
gluing schemes; see, for example,
\cite{delPinoKowalczykWei07,delPinoKowalczykWei11}. Such schemes begin by
linearizing the equation around a smooth global approximation and then solve
coupled linear correction problems in the inner and outer regions. For the
segregated sublinear profiles considered here, however, the coupling map has
no bounded derivative at the axis configurations that occur where one
component vanishes and the other remains nonzero. The standard linear outer
problem underlying those schemes is therefore unavailable, and the schemes
cannot be applied to \eqref{eq0} in their usual form. The exterior interaction
must first be resolved as a genuinely nonlinear variational problem. Only
after the resulting dead core removes the coupling from the inner regions does
the projected critical reduction become available.

The results below treat one bubble in each component. After the main
statements, we compare the degree hypothesis with related stability conditions
and briefly discuss multi-bubble and common-center extensions.

\medskip
\subsection{Notation and main results}

For $\lambda>0$ and $y\in\mathbb R^N$, set
\[
U_{\lambda,y}(x)
=C_N\lambda^{\frac{N-2}{2}}
\bigl(\lambda^2+|x-y|^2\bigr)^{-\frac{N-2}{2}},
\qquad
C_N=[N(N-2)]^{\frac{N-2}{4}}.
\]
Then
\[
-\Delta U_{\lambda,y}=U_{\lambda,y}^{\frac{N+2}{N-2}}
\quad\text{in }\mathbb R^N.
\]
These are the standard positive bubbles for the critical equation; see
\cite{Aubin76,CaffarelliGidasSpruck89,Talenti76}.

Let $\omega_{N-1}=|\mathbb S^{N-1}|$ and
$\gamma_N=((N-2)\omega_{N-1})^{-1}$. We use the unnormalized Newton kernel
\[
\Gamma_x(y)=|x-y|^{2-N}.
\]
For $x\in\Omega$, let $H_\Omega(x,\cdot)$ be harmonic in $\Omega$ and equal to
$\Gamma_x$ on $\partial\Omega$. The Dirichlet Green function is
\[
G_\Omega(x,y)=\Gamma_x(y)-H_\Omega(x,y),
\]
so that
\[
-\Delta_yG_\Omega(x,y)=(N-2)\omega_{N-1}\delta_x(y),
\qquad
G_\Omega(x,y)=0\quad\text{on }\partial\Omega.
\]
Accordingly, Green's representation contains the factor $\gamma_N$. When no
confusion can arise, we write $G$ and $H$ in place of $G_\Omega$ and
$H_\Omega$.

The Robin function is
\[
\varphi_\Omega(x)=H_\Omega(x,x),\qquad x\in\Omega.
\]
It is positive and of class $C^2$ in $\Omega$, and
$\varphi_\Omega(x)\to+\infty$ as
$\operatorname{dist}(x,\partial\Omega)\to0$.

Set
\[
\Omega_\epsilon=\epsilon^{-\frac1{N-4}}\Omega.
\]
The rescaled problem is
\begin{equation}\label{eq1}
\begin{cases}
-\Delta u_1
=\epsilon^{\frac{N-2}{N-4}}u_1+|u_1|^{2^*-2}u_1
+\beta|u_2|^{\frac{2^*}{2}}|u_1|^{\frac{2^*}{2}-2}u_1,
&\text{in }\Omega_\epsilon,\\[1mm]
-\Delta u_2
=\epsilon^{\frac{N-2}{N-4}}u_2+|u_2|^{2^*-2}u_2
+\beta|u_1|^{\frac{2^*}{2}}|u_2|^{\frac{2^*}{2}-2}u_2,
&\text{in }\Omega_\epsilon,\\[1mm]
u_1,u_2\geq0,\quad u_1,u_2\not\equiv0,
&\text{in }\Omega_\epsilon,\\[1mm]
u_1=u_2=0,
&\text{on }\partial\Omega_\epsilon.
\end{cases}
\end{equation}
The transformation
\[
u_i(y)=\epsilon^{\frac{N-2}{2(N-4)}}
\bar u_i\bigl(\epsilon^{\frac1{N-4}}y\bigr),
\qquad i=1,2,
\]
gives a one-to-one correspondence between solutions of \eqref{eq0} and
solutions of \eqref{eq1}. We carry out the construction in
$\Omega_\epsilon$, where the bubbles have scale $O(1)$ and the domain expands
as $\epsilon\to0$.

For $D=\Omega$ or $D=\Omega_\epsilon$ and
$U\in C^2(D)\cap C(\overline D)$, let $P_DU$ denote the solution of
\[
\begin{cases}
-\Delta(P_DU)=-\Delta U,&\text{in }D,\\
P_DU=0,&\text{on }\partial D.
\end{cases}
\]
We write $P=P_\Omega$ and $P_\epsilon=P_{\Omega_\epsilon}$. In particular,
\[
-\Delta PU_{\lambda,y}=U_{\lambda,y}^{\frac{N+2}{N-2}}
\quad\text{in }\Omega,
\qquad
PU_{\lambda,y}=0\quad\text{on }\partial\Omega,
\]
and the analogous identities hold for $P_\epsilon U_{\lambda,y}$ in
$\Omega_\epsilon$.

The concentration points are selected through the following hypothesis.

\begin{Def}\label{defstable}
Let $x_0\in\Omega$ be a critical point of $\varphi_\Omega$. We say that
$x_0$ satisfies the \emph{local degree condition} if, for every neighborhood
$\mathcal U$ of $x_0$ in $\Omega$, there exists a bounded open set $D$ such
that
\[
x_0\in D\Subset\mathcal U,
\qquad
\nabla\varphi_\Omega\neq0\quad\text{on }\partial D,
\qquad
\deg(\nabla\varphi_\Omega,D,0)\neq0.
\]
\end{Def}

This condition is the geometric hypothesis used in the final
finite-dimensional argument. We now state the main results.

\begin{Thm}\label{thm:main}
Assume that $N\geq5$, $\beta<0$, and that
$\bar\xi_1,\bar\xi_2\in\Omega$ are two distinct critical points of
$\varphi_\Omega$, each satisfying the local degree condition in
Definition~\ref{defstable}. Set
\[
\lambda_i^0
=\left(
\frac{\displaystyle\int_{\mathbb R^N}U_{1,0}^2}
{\displaystyle \frac{N-2}{2}C_N\varphi_\Omega(\bar\xi_i)
\int_{\mathbb R^N}U_{1,0}^{2^*-1}}
\right)^{\frac1{N-4}},
\qquad i=1,2.
\]
Then there exists $\epsilon_0>0$ such that, for every
$\epsilon\in(0,\epsilon_0)$, problem~\eqref{eq0} admits a nonnegative weak
solution $(\bar u_{1,\epsilon},\bar u_{2,\epsilon})$ with both components
nontrivial and
\[
\bar u_{1,\epsilon}
=PU_{\epsilon^{\frac1{N-4}}\lambda_{1,\epsilon},\xi_{1,\epsilon}}
+\bar\phi_\epsilon,
\qquad
\bar u_{2,\epsilon}
=PU_{\epsilon^{\frac1{N-4}}\lambda_{2,\epsilon},\xi_{2,\epsilon}}
+\bar\psi_\epsilon,
\]
where
\[
(\lambda_{1,\epsilon},\xi_{1,\epsilon},
\lambda_{2,\epsilon},\xi_{2,\epsilon})
\longrightarrow
(\lambda_1^0,\bar\xi_1,\lambda_2^0,\bar\xi_2)
\]
and
\[
(\bar\phi_\epsilon,\bar\psi_\epsilon)\longrightarrow(0,0)
\quad\text{in }L^\infty(\Omega)\times L^\infty(\Omega).
\]
In particular, the two components concentrate at $\bar\xi_1$ and
$\bar\xi_2$, respectively, and the family is segregated in the sense defined
above.
\end{Thm}

Theorem~\ref{thm:main} follows by rescaling from the next result.

\begin{Thm}\label{thm:rescaled-main}
Under the hypotheses of Theorem~\ref{thm:main}, there exists $\epsilon_0>0$ such
that, for every $\epsilon\in(0,\epsilon_0)$, problem~\eqref{eq1} admits a
nonnegative weak solution $(u_{1,\epsilon},u_{2,\epsilon})$ with both
components nontrivial and
\[
u_{1,\epsilon}
=P_\epsilon U_{\lambda_{1,\epsilon},\tilde\xi_{1,\epsilon}}
+\phi_\epsilon,
\qquad
u_{2,\epsilon}
=P_\epsilon U_{\lambda_{2,\epsilon},\tilde\xi_{2,\epsilon}}
+\psi_\epsilon,
\]
where
\[
\tilde\xi_{i,\epsilon}
=\epsilon^{-\frac1{N-4}}\xi_{i,\epsilon},
\qquad i=1,2,
\]
\[
(\lambda_{1,\epsilon},\xi_{1,\epsilon},
\lambda_{2,\epsilon},\xi_{2,\epsilon})
\longrightarrow
(\lambda_1^0,\bar\xi_1,\lambda_2^0,\bar\xi_2),
\]
and there exist constants $C>0$ and $\sigma>0$, independent of
$\epsilon$, such that
\begin{equation}\label{eq:strong-correction}
\|\phi_\epsilon\|_{L^\infty(\Omega_\epsilon)}
+\|\psi_\epsilon\|_{L^\infty(\Omega_\epsilon)}
\leq C\epsilon^{\frac{N-2}{2(N-4)}+\sigma}.
\end{equation}
\end{Thm}

\medskip

\begin{Thm}\label{thm:dead-core}
There exists $\varrho_0>0$ such that, for every fixed
$\varrho\in(0,\varrho_0)$, there exists
$\epsilon_0=\epsilon_0(\varrho)>0$ such that the solutions
$(u_{1,\epsilon},u_{2,\epsilon})$ constructed in
Theorem~\ref{thm:rescaled-main} satisfy, for every
$\epsilon\in(0,\epsilon_0)$,
\[
u_{1,\epsilon}=0
\quad\text{in }
B_{\epsilon^{-\frac{1-\varrho}{N-2}}}(\tilde\xi_{2,\epsilon}),
\qquad
u_{2,\epsilon}=0
\quad\text{in }
B_{\epsilon^{-\frac{1-\varrho}{N-2}}}(\tilde\xi_{1,\epsilon}).
\]
\end{Thm}

Here $\varrho$ is fixed before $\epsilon\to0$, and the threshold
$\epsilon_0$ may depend on $\varrho$. After returning to the original
variables, the corresponding dead-core radius is
\[
\epsilon^{\frac1{N-4}-\frac{1-\varrho}{N-2}}\longrightarrow0,
\]
whereas its ratio to the natural bubble scale $\epsilon^{1/(N-4)}$ tends to
infinity. Thus Theorem~\ref{thm:dead-core} gives a genuinely local vanishing
property around the opposite concentration point. This conclusion reflects the
sublinear nature of the competitive coupling and is a distinctive feature of
the present construction compared with classical scalar Brezis--Nirenberg
bubbling.

We conclude this subsection with several comments on the degree hypothesis and
on possible extensions of the concentration pattern.

\begin{Rem}\label{rem:degree-comparison}
Definition~\ref{defstable} is a standard local degree assumption. It implies
the $C^1$-stability of the singleton $\{x_0\}$ in the sense of
Li~\cite{Li97}; if $x_0$ is isolated, it is equivalent to the nonvanishing of
the local Brouwer degree of $\nabla\varphi_\Omega$ at $x_0$. In particular,
it holds at every nondegenerate critical point and at every isolated strict
local extremum. It is closely related to the condition in
\cite[Definition~2.4]{Musso02}, where nonzero degree is required on one fixed
neighborhood together with an additional critical-level assumption. That
assumption keeps the scale root fixed in their degree homotopy, whereas here
the scale root is allowed to depend on the center. The use of degree domains
contained in every neighborhood serves to select concentration centers
converging to the prescribed point; when $x_0$ is isolated, one sufficiently
small isolating neighborhood is enough.
\end{Rem}

\begin{Rem}\label{rem:multiple-concentration-points}
The present results concern one bubble in each component. A natural extension
is to seek solutions with several distinct concentration points, of the form
\[
u_{1,\epsilon}\sim
\sum_{a=1}^{k_1}P_\epsilon
U_{\lambda_{1,a,\epsilon},\tilde\xi_{1,a,\epsilon}},
\qquad
u_{2,\epsilon}\sim
\sum_{b=1}^{k_2}P_\epsilon
U_{\lambda_{2,b,\epsilon},\tilde\xi_{2,b,\epsilon}}.
\]
When the two limiting concentration sets are disjoint, the bubbles belonging
to the same component are expected to interact, at leading order, through the
usual Robin--Green reduced function of Kirchhoff--Routh type: its
self-interaction terms involve $\varphi_\Omega$, whereas its pairwise
interaction terms involve $G_\Omega$; see \cite{Musso02,Pistoia13}.
Accordingly, the hypothesis on the concentration points should be replaced by
a nonzero local degree condition at a prescribed zero of the full limiting
reduced map in all scale and center variables, or equivalently at a prescribed
critical configuration of the associated reduced functional. The exterior construction would also
have to produce dead cores near all centers of the opposite component, so that
the cross-component coupling vanishes in every inner region. Carrying this out
requires uniform multi-center estimates and a multi-bubble version of the
projected reduction, and is beyond the scope of the present paper.
\end{Rem}

\begin{Rem}\label{rem:mixed-patterns}
A different extension concerns components concentrating at a common point.
For the present two-component system, a natural synchronized limiting profile
in $\mathbb R^N$ is
\[
\bigl(a_1U_{\lambda,\xi},a_2U_{\lambda,\xi}\bigr),
\qquad a_1,a_2>0,
\]
and substitution into the limiting system gives
\[
\begin{cases}
1=a_1^{2p-2}+\beta a_1^{p-2}a_2^p,\\
1=a_2^{2p-2}+\beta a_2^{p-2}a_1^p.
\end{cases}
\]
Thus such a branch can be considered only when this algebraic system has a
positive solution. Its construction also requires a suitable nondegeneracy
result for the corresponding vector bubble, after the common translation and
scaling modes are taken into account.

More generally, for an $m$-component system one may choose a partition into
$m_0$ nonempty groups, where $1\le m_0\le m$:
\[
\{1,\ldots,m\}=G_1\cup\cdots\cup G_{m_0},
\qquad
G_h\cap G_\ell=\varnothing\quad\text{if }h\ne\ell,
\]
and seek grouped profiles of the form
\[
u_{i,\epsilon}\sim
a_i^{(h)}P_\epsilon
U_{\lambda_{h,\epsilon},\tilde\xi_{h,\epsilon}},
\qquad i\in G_h,\quad h=1,\ldots,m_0,
\]
with distinct limiting centers for different groups. The positive amplitude
vector in each group must satisfy the corresponding algebraic system. The
components are then synchronized within each group and segregated between
different groups; related patterns were studied in \cite{DovettaPistoia}.
In the present Brezis--Nirenberg setting, such a construction would combine
the amplitude systems and vector-bubble linear theory within the groups with
dead-core estimates between different groups. The finite-dimensional
reduction in the group scales and centers would contain Robin terms whose
coefficients depend on the amplitude vectors. If one also allows several
bubbles in a component, the Robin--Green interactions described in
Remark~\ref{rem:multiple-concentration-points} must be included as well.
Different positive amplitude vectors or different partitions may lead to
different branches of bubbling solutions. We do not pursue these extensions
here.
\end{Rem}

\medskip

\subsection{Structure of the paper}

The paper is organized as follows. Section~2 formalizes the hierarchy of
scales and nested regions used in the inner--outer decomposition, and develops
the projected linear theory. Section~3 constructs the nonlinear exterior
solution and proves the decay, stability, and dead-core estimates. Section~4
uses the inclusion of $P_2,Q_2$ in the dead-core regions to solve the localized
projected problem near the two bubbles. Section~5 analyzes the reduced
finite-dimensional equations and completes the proof by Brouwer degree. Appendix~A collects auxiliary estimates and the simultaneous
choice of the auxiliary exponents, while Appendix~B records the
parameter-dependence results needed in the construction.

\medskip
\section{Preliminaries}

\subsection{Scales, regions, and projections}

We first formalize the geometric decomposition described in the introduction.
Recall the notation
\[
s_\epsilon=\epsilon^{\frac1{N-4}},
\qquad
\Omega_\epsilon=s_\epsilon^{-1}\Omega,
\qquad
R_\Omega:=\operatorname{diam}(\Omega).
\]
For $0<\delta<1$, define the open configuration set
\[
O_\delta:=
\left\{(\lambda_1,\xi_1,\lambda_2,\xi_2)
\in(\mathbb R_+\times\Omega)^2:
\begin{aligned}
&\operatorname{dist}(\xi_i,\partial\Omega)>\delta,
\quad \delta<\lambda_i<\delta^{-1},\quad i=1,2,\\
&|\xi_1-\xi_2|>\delta
\end{aligned}
\right\}.
\]
Fix a configuration vector
$\mathbf t=(\lambda_1,\xi_1,\lambda_2,\xi_2)\in O_\delta$ and put
\[
\tilde\xi_i=s_\epsilon^{-1}\xi_i,
\qquad i=1,2.
\]
Then
\[
|\tilde\xi_1-\tilde\xi_2|>\delta s_\epsilon^{-1},
\qquad
\operatorname{dist}(\tilde\xi_i,\partial\Omega_\epsilon)
>\delta s_\epsilon^{-1}.
\]
Thus the mutual separation of the centers and their distance from the boundary
are both measured on the scale $s_\epsilon^{-1}$.

Throughout the paper we fix $\tau$ in the dimension-dependent range
\begin{equation}\label{tau-range}
\tau\in\mathcal I_N:=
\begin{cases}
(4,+\infty), & N=5,\\[1mm]
(2,+\infty), & N=6,\\[1mm]
\displaystyle\left(\frac{4}{N-4},\frac{N+8}{N-6}\right), & N\ge7,
\end{cases}
\end{equation}
and set
\[
B_\tau:=(2+\tau)(N-4).
\]
We define the intermediate radius by
\begin{equation}\label{repsilon}
r_\epsilon:=\epsilon^{-1/B_\tau}.
\end{equation}
The lower bound in \eqref{tau-range} allows us to choose
$\varrho_\tau>0$ such that, for every $\varrho\in(0,\varrho_\tau)$,
\begin{equation}\label{deadcore-fit}
\frac{2}{(2+\tau)(N-4)}
<\frac{1-\varrho}{N-2}.
\end{equation}
With
\[
R_{\mathrm{dc},\epsilon}(\varrho)
:=\epsilon^{-\frac{1-\varrho}{N-2}},
\]
this gives the scale hierarchy
\begin{equation}\label{scale-hierarchy}
1\ll r_\epsilon\ll r_\epsilon^2
\ll R_{\mathrm{dc},\epsilon}(\varrho)
\ll s_\epsilon^{-1}
\qquad (0<\varrho<\varrho_\tau),
\end{equation}
where $a_\epsilon\ll b_\epsilon$ means
$a_\epsilon/b_\epsilon\to0$ as $\epsilon\to0$. For $N\ge7$, the upper
bound in \eqref{tau-range} is used in the mixed-barrier estimate of
Lemma~\ref{lem8}. The interval $\mathcal I_N$ is nonempty for every $N\ge5$.

Unless stated otherwise, constants denoted by $C$ are independent of
$\epsilon$ and of $\mathbf t\in O_\delta$, and every smallness threshold
$\epsilon_0$ is uniform for $\mathbf t\in O_\delta$. These constants and
thresholds may depend on the fixed data $N,\Omega,\beta,\delta,\tau$ and
on the auxiliary cut-off functions and exponents once these have been
chosen. In statements in which $\varrho>0$ is fixed, the threshold may
also depend on $\varrho$.

We now introduce the two nested regions around each center:
\begin{equation}\label{PQ}
\begin{aligned}
P_1&:=B_{r_\epsilon}(\tilde\xi_1),
&\qquad P_2&:=B_{r_\epsilon^2}(\tilde\xi_1),\\
Q_1&:=B_{r_\epsilon}(\tilde\xi_2),
& Q_2&:=B_{r_\epsilon^2}(\tilde\xi_2).
\end{aligned}
\end{equation}
By \eqref{scale-hierarchy} and the definition of $O_\delta$, for all
sufficiently small $\epsilon$,
\[
\overline{P_2}\cup\overline{Q_2}\Subset\Omega_\epsilon,
\qquad
P_2\cap Q_2=\varnothing.
\]
The smaller balls $P_1,Q_1$ are the trace regions for the exterior variational
problem. The larger balls $P_2,Q_2$ are the localization regions for the
projected linear problems and contain the transition annuli used in Section~4.
The relation $r_\epsilon^2\ll R_{\mathrm{dc},\epsilon}(\varrho)$ is chosen so
that the dead-core estimate proved in Section~3 eliminates the opposite
component on all of $P_2$ and $Q_2$; the resulting inclusion is recorded in
Remark~\ref{remdead}.

Set
\begin{equation}\label{UV}
U_i:=U_{\lambda_i,\tilde\xi_i},
\qquad
V_i:=P_\epsilon U_i,
\qquad i=1,2.
\end{equation}
For $j=1,\ldots,N$ and $i=1,2$, let
\begin{equation}\label{Y}
Y_i^0:=\frac{\partial U_{\lambda_i,\tilde\xi_i}}{\partial\lambda_i},
\qquad
Y_i^j:=\frac{\partial U_{\lambda_i,\tilde\xi_i}}
{\partial\tilde\xi_i^j}
=-\partial_{y_j}U_i.
\end{equation}
We denote by $Z_i^0$ and $Z_i^j$ their $H_0^1(\Omega_\epsilon)$-projections,
that is, the unique solutions of
\[
\Delta Z_i^0=\Delta Y_i^0,
\qquad
\Delta Z_i^j=\Delta Y_i^j
\quad\text{in }\Omega_\epsilon,
\qquad
Z_i^0=Z_i^j=0
\quad\text{on }\partial\Omega_\epsilon.
\]
Equivalently,
\begin{equation}\label{Z}
Z_i^0=P_\epsilon Y_i^0,
\qquad
Z_i^j=P_\epsilon Y_i^j,
\qquad i=1,2,\quad j=1,\ldots,N.
\end{equation}

The projections in $\Omega$ and $\Omega_\epsilon$ are related by
\begin{align}\label{PP}
P U_{s_\epsilon\lambda,s_\epsilon y}(x)
=s_\epsilon^{-\frac{N-2}{2}}
P_\epsilon U_{\lambda,y}\left(\frac{x}{s_\epsilon}\right),
\qquad x\in\Omega.
\end{align}
If $\tilde\xi_i=s_\epsilon^{-1}\xi_i$, then
\[
P_\epsilon Y^0_{\lambda_i,\tilde\xi_i}(y)
=s_\epsilon^{\frac N2}P Y^0_{s_\epsilon\lambda_i,\xi_i}(s_\epsilon y),
\qquad
P_\epsilon Y^j_{\lambda_i,\tilde\xi_i}(y)
=s_\epsilon^{\frac N2}P Y^j_{s_\epsilon\lambda_i,\xi_i}(s_\epsilon y).
\]
The Green function and its regular part satisfy
\[
G_{\Omega_\epsilon}(y,z)
=s_\epsilon^{N-2}G_\Omega(s_\epsilon y,s_\epsilon z),
\qquad
H_{\Omega_\epsilon}(y,z)
=s_\epsilon^{N-2}H_\Omega(s_\epsilon y,s_\epsilon z),
\]
and hence
\[
\varphi_{\Omega_\epsilon}(y)
=s_\epsilon^{N-2}\varphi_\Omega(s_\epsilon y).
\]
Appendix~A records the projection estimates in the original variables; the
corresponding estimates in $\Omega_\epsilon$ follow from these scaling
identities.

\medskip

Fix a sufficiently large $R_0>0$ and choose
$\chi_0\in C_c^\infty(\mathbb R^N)$ such that $\chi_0=1$ in $B_{R_0}(0)$
and $\chi_0=0$ outside $B_{2R_0}(0)$. Set
\[
Y_{i,0}=\chi_0(\cdot-\tilde\xi_i)Y_i^0,
\qquad
Y_{i,j}=\chi_0(\cdot-\tilde\xi_i)Y_i^j,
\quad i=1,2,\quad j=1,\ldots,N.
\]
Since $r_\epsilon\to\infty$, after reducing $\epsilon_0$ we have
\[
\operatorname{supp}Y_{1,j}\Subset P_1,
\qquad
\operatorname{supp}Y_{2,j}\Subset Q_1,
\qquad j=0,\ldots,N.
\]

For \(\epsilon>0\), \(i=1,2\), \(\lambda_i>0\), and
\(\tilde\xi_i\in\Omega_\epsilon\), define the linear moment functionals
\[
\ell_{i,j}(u)
:=
\int_{\Omega_\epsilon}
U_{\lambda_i,\tilde\xi_i}^{2^*-2}Y_{i,j}u,
\qquad j=0,1,\ldots,N.
\]
We set
\[
E_i:=E_\epsilon^{\lambda_i,\tilde\xi_i}
=
\bigcap_{j=0}^N\ker \ell_{i,j}
\subset H_0^1(\Omega_\epsilon).
\]
The space \(E_i\) is a closed finite-codimensional subspace of
\(H_0^1(\Omega_\epsilon)\), defined by the above moment conditions. These
conditions are not meant as orthogonality with respect to the \(H_0^1\)-inner
product. In what follows, projected problem always refers to this moment
projection.

Set
\begin{equation*}
\mathbb E=E_{1}\times E_{2}.
\end{equation*}
We use the notation
\[
C_0(\Omega_\epsilon)
:=\{v\in C(\overline{\Omega_\epsilon}):v=0
\text{ on }\partial\Omega_\epsilon\}.
\]

\medskip
\subsection{Linear problems}
We consider the following two linear problems.
\begin{equation} \label{eq:linear-first}
\begin{split}
\begin{cases}
-\Delta\varphi-\epsilon^{\frac{N-2}{N-4}}\varphi-(2^{*}-1)V_1^{2^{*}-2}\varphi=\mathbf1_{P_2}h_{1}+\sum_{j=0}^{N}c_{j} U^{2^{*}-2}_{1}Y_{1,j}\ &\text{in } \Omega_\epsilon\\[3mm]
\varphi=0\ \ &\text{on } \partial\Omega_\epsilon\\[3mm]
\int_{\Omega_\epsilon} U^{2^{*}-2}_{1}Y_{1,l} \varphi =0,\ l=0,1,\ldots, N,
\end{cases}
\end{split}
\end{equation}
\begin{equation} \label{eq:linear-second}
\begin{split}
\begin{cases}
-\Delta\psi-\epsilon^{\frac{N-2}{N-4}}\psi-(2^{*}-1)V_2^{2^{*}-2}\psi=\mathbf1_{Q_2}h_{2}+\sum_{j=0}^{N}d_{j} U^{2^{*}-2}_{2}Y_{2,j}\ &\text{in } \Omega_\epsilon\\[3mm]
\psi=0\ \ &\text{on } \partial\Omega_\epsilon\\[3mm]
\int_{\Omega_\epsilon} U^{2^{*}-2}_{2}Y_{2,l} \psi=0,\ l=0,1,\ldots, N,
\end{cases}
\end{split}
\end{equation}
for some $c_{j},d_{j}\in\mathbb R$.

\begin{Lem}\label{lem1}
Let \(0<\delta<1\). Suppose that \(\varphi=\varphi_\epsilon\) solves
\eqref{eq:linear-first} with \(h_1\in L^\infty(\Omega_\epsilon)\). If
\(r_\epsilon^4\|h_1\|_{L^\infty(P_2)}\to0\) as \(\epsilon\to0\), then
\(\|\varphi_\epsilon\|_{L^\infty(\Omega_\epsilon)}\to0\), uniformly for
configurations in \(O_\delta\). The analogous statement holds for
\eqref{eq:linear-second} with \(P_2,h_1\) replaced by \(Q_2,h_2\).
\end{Lem}
\begin{proof}
Arguing by contradiction, suppose that there exist a sequence
$\epsilon_m\to0$, functions $h_{1,m}$, configuration vectors
$\mathbf t_m=(\lambda_{1,m},\xi_{1,m},\lambda_{2,m},\xi_{2,m})\in O_\delta$, and
solutions $\varphi_m$ of \eqref{eq:linear-first} such that
$r_{\epsilon_m}^4\|h_{1,m}\|_{L^\infty(P_2)}\to0$ but
$\|\varphi_m\|_{L^\infty(\Omega_{\epsilon_m})}\ge c'>0$. After normalization,
we may assume that $\|\varphi_m\|_{L^\infty(\Omega_{\epsilon_m})}=1$. To
simplify notation, we suppress the index $m$ and write
$\tilde\xi_{i,\epsilon}=s_\epsilon^{-1}\xi_{i,\epsilon}$. Then
\eqref{eq:linear-first} has the representation
\begin{equation}\label{eq8}
\begin{split}
\varphi(y)=&\gamma_N\int_{\Omega_{\epsilon}} G_{\Omega_{\epsilon}}(z,y)
\left((2^{*}-1)V^{2^{*}-2}_{1}(z)+\epsilon^{\frac{N-2}{N-4}}\right)\varphi(z)\,dz \\
&+\gamma_N\int_{\Omega_{\epsilon}} G_{\Omega_{\epsilon}}(z,y)
\Big(\mathbf1_{P_2}(z)h_{1}(z)+\sum_{j=0}^{N}c_{j} U^{2^{*}-2}_{1}Y_{1,j}(z)\Big)\,dz.
\end{split}
\end{equation}

From Lemma~\ref{lemB2}, we have
\begin{equation*}
\begin{split}
&\Big\lvert \int_{\Omega_{\epsilon}} G_{\Omega_{\epsilon}}(z,y)(2^{*}-1)V^{2^{*}-2}_{1}(z)\varphi(z) dz\Big\rvert\\
\leq &C \lVert \varphi\rVert_{L^{\infty}(\Omega_\epsilon)} \int_{\Omega_\epsilon} \frac{1}{\lvert z-y\rvert^{N-2}} \Big( \frac{1}{(1+\lvert z-\tilde\xi_1\rvert)^{N-2}}\Big)^{\frac{4}{N-2}}dz
\leq C \lVert \varphi\rVert_{L^{\infty}(\Omega_\epsilon)} \frac{1}{(1+\lvert y-\tilde\xi_1\rvert)^{2}},
\end{split}
\end{equation*}
 \begin{equation*}
\begin{split}
&\Big \lvert\int_{\Omega_{\epsilon}} G_{\Omega_{\epsilon}}(z,y)\epsilon^{\frac{2}{N-4}+1}\varphi(z)\,dz\Big\rvert
\leq C \epsilon^{\frac{2}{N-4}+1}\lVert \varphi\rVert_{L^{\infty}(\Omega_\epsilon)} \int_{\Omega_\epsilon} \frac{1}{\lvert z-y\rvert^{N-2}} dz
\leq C  \epsilon  \lVert \varphi\rVert_{L^{\infty}(\Omega_\epsilon)},
\end{split}
\end{equation*}

\begin{equation*}
\Big\lvert \int_{\Omega_{\epsilon}}G_{\Omega_{\epsilon}}(z,y)
\mathbf1_{P_2}(z)h_1(z)\,dz\Big\rvert
\le C\int_{P_2}\frac{|h_1(z)|}{|z-y|^{N-2}}\,dz
\le Cr_\epsilon^4\|h_1\|_{L^\infty(P_2)}.
\end{equation*}
and
\begin{equation*}
\begin{split}
\Big\lvert \int_{\Omega_{\epsilon}} G_{\Omega_{\epsilon}}(z,y)U^{2^{*}-2}_{1}(z)Y_{1,j}(z) dz\Big\rvert
\leq &C \int_{\Omega_\epsilon} \frac{1}{\lvert z-y\rvert^{N-2}} \frac{1}{(1+\lvert z-\tilde\xi_1\rvert)^{N+2}}dz
\leq   \frac{C}{(1+\lvert y-\tilde\xi_1\rvert)^{N-2}}.
\end{split}
\end{equation*}

We next estimate $c_l$, $l=0,1,\ldots,N$. Multiplying \eqref{eq:linear-first} by $Z_1^l$ and integrating, we obtain the linear system
\begin{equation}\label{eq11}
\sum_{j=0}^{N}\left(\int_{\Omega_{\epsilon}}  U^{2^{*}-2}_{1}Y_{1,j}Z_{1}^l\right)c_{j}=\int_{\Omega_{\epsilon}}-\Delta \varphi Z_{1}^l -\epsilon^{\frac{N-2}{N-4}}\varphi Z_{1}^l -(2^{*}-1)V^{2^{*}-2}_1\varphi Z_{1}^l - \mathbf1_{P_2}h_{1}Z_{1}^l.
\end{equation}
Lemma~\ref{lemB1} gives
\begin{equation*}
\begin{split}
\left\lvert \int_{\Omega_{\epsilon}}\mathbf1_{P_2}h_{1} Z_{1}^l\right \rvert
\leq&  C \lVert h_{1} \rVert_{L^{\infty}(P_{2})} \int_{B_{r_\epsilon^2}(\tilde\xi_1)} \frac{1}{(1+\lvert z-\tilde\xi_1\rvert)^{N-2}}dz
\leq C r^{4}_{\epsilon}\lVert h_{1} \rVert_{L^{\infty}(P_{2})}.
\end{split}
\end{equation*}
Moreover, we estimate
 \begin{equation*}
\begin{split}
&\int_{\Omega_{\epsilon}}
\Big(-\Delta \varphi-\epsilon^{\frac{N-2}{N-4}}\varphi
-(2^{*}-1)V^{2^{*}-2}_1\varphi\Big)Z_{1}^l \\
={}&\int_{\Omega_{\epsilon}}
\Big(-\Delta Z_{1}^l-(2^{*}-1)U^{2^{*}-2}_1Y_{1}^l\Big)\varphi \\
&+(2^{*}-1)\int_{\Omega_\epsilon}
\Big[U_1^{2^*-2}(Y_1^l-Z_1^l)
+\big(U_1^{2^*-2}-V_1^{2^*-2}\big)Z_1^l\Big]\varphi \\
&-\epsilon^{\frac{N-2}{N-4}}\int_{\Omega_\epsilon}\varphi Z_1^l
=O(\epsilon)\|\varphi\|_{L^\infty(\Omega_\epsilon)}.
\end{split}
\end{equation*}
The last estimate follows from Lemmas~\ref{lemRj0} and~\ref{lemD}, together with
the scaling relation
\(\epsilon^{(N-2)/(N-4)}|\Omega_\epsilon|^{2/N}=O(\epsilon)\).

After translation to the bubble center, the coefficient matrix
\[
\left(\int_{\Omega_\epsilon}U_1^{2^*-2}Y_{1,j}Z_1^l\right)_{0\le j,l\le N}
\]
converges to
 \[
\left(\int_{\mathbb R^N}\chi_0U_{\lambda_1,0}^{2^*-2}Y_jY_l\right)_{0\le j,l\le N}.
\]
The weight is positive on a ball and the functions $Y_0,\ldots,Y_N$ are
linearly independent; hence this Gram matrix is positive definite. The
linear system is therefore invertible for small $\epsilon$, and
\[
c_l=O\Big(\epsilon\|\varphi\|_{L^\infty(\Omega_\epsilon)}
+r_\epsilon^4\|h_1\|_{L^\infty(P_2)}\Big).
\]
Consequently,
\begin{equation}\label{eq12}
|\varphi(y)|\le
\frac{C\|\varphi\|_{L^\infty(\Omega_\epsilon)}}{(1+|y-\tilde\xi_1|)^2}
+C\epsilon\|\varphi\|_{L^\infty(\Omega_\epsilon)}
+Cr_\epsilon^4\|h_1\|_{L^\infty(P_2)}.
\end{equation}
Since $\|\varphi\|_\infty=1$, there are $R,c_0>0$ such that
\begin{equation}\label{eq13}
\|\varphi\|_{L^\infty(B_R(\tilde\xi_1))}\ge c_0.
\end{equation}
After passing to a subsequence,
$\widehat\varphi_\epsilon(z):=\varphi_\epsilon(z+\tilde\xi_{1,\epsilon})$
converges in $C^1_{\mathrm{loc}}(\mathbb R^N)$ to a bounded solution $u$ of
\begin{equation}\label{eq14}
-\Delta u-(2^*-1)U_{\lambda,0}^{2^*-2}u=0
\quad\text{in }\mathbb R^N
\end{equation}
 for some $\lambda>0$. Because the limit is bounded, the bounded-solution
nondegeneracy theorem for the linearized bubble equation
\cite{BartschWethWillem03}, after scaling, applies and gives
 \[
u=\sum_{j=0}^N a_jY_j.
\]
Passing to the limit in the moment conditions gives
\[
\sum_{j=0}^Na_j\int_{\mathbb R^N}\chi_0U_{\lambda,0}^{2^*-2}Y_jY_l=0,
\qquad l=0,\ldots,N.
\]
The matrix in this system is positive definite, since $\chi_0\ge0$ is
positive on a ball and $Y_0,\ldots,Y_N$ are linearly independent. Hence
$a_0=\cdots=a_N=0$, so $u=0$, contradicting \eqref{eq13}.

\end{proof}

Lemma~\ref{lem1}, together with the Fredholm argument below, yields the
following solvability result. The estimates are uniform for configurations in
$O_\delta$.

\begin{Prop}\label{prop:linear-first}
There exist constants $\epsilon_0,C>0$, independent of the configuration
in $O_\delta$, such that, for every $\epsilon\in(0,\epsilon_0)$ and
$h_1\in L^\infty(\Omega_\epsilon)$, problem \eqref{eq:linear-first} has a
unique solution $\varphi=L_1h_1\in E_1\cap C_0(\Omega_\epsilon)$. Moreover,
\[
\|L_1h_1\|_{L^\infty(\Omega_\epsilon)}
+\max_{0\le j\le N}|c_j|
\le Cr_\epsilon^4\|h_1\|_{L^\infty(P_2)}.
\]
\end{Prop}

The second linear problem has the analogous solvability property.

\begin{Prop}\label{prop:linear-second}
There exist constants $\epsilon_0,C>0$, independent of the configuration
in $O_\delta$, such that, for every $\epsilon\in(0,\epsilon_0)$ and
$h_2\in L^\infty(\Omega_\epsilon)$, problem \eqref{eq:linear-second} has a
unique solution $\psi=L_2h_2\in E_2\cap C_0(\Omega_\epsilon)$. Moreover,
\[
\|L_2h_2\|_{L^\infty(\Omega_\epsilon)}
+\max_{0\le j\le N}|d_j|
\le Cr_\epsilon^4\|h_2\|_{L^\infty(Q_2)}.
\]
\end{Prop}

\begin{proof}[Proof of Propositions~\ref{prop:linear-first} and~\ref{prop:linear-second}]
We prove the first statement; the second is identical. Set
\[
W_{1,j}:=U_1^{2^*-2}Y_{1,j},
\qquad j=0,\ldots,N.
\]
The matrix
\[
\left(
\int_{\Omega_\epsilon}W_{1,j}Z_1^l
\right)_{0\le j,l\le N}
\]
converges uniformly with respect to the admissible parameters to the family of
Gram matrices appearing in the proof of Lemma~\ref{lem1}. This family depends
continuously on $\lambda_1\in[\delta,\delta^{-1}]$ and is positive definite;
hence its least eigenvalue has a positive uniform lower bound. It follows that
the moment functionals $\ell_{1,0},\ldots,\ell_{1,N}$ are linearly independent
for all sufficiently small $\epsilon$.

On $E_1$ consider the bilinear form
\[
\mathcal Q_1(u,v)
=
\int_{\Omega_\epsilon}\nabla u\cdot\nabla v
-\epsilon^{\frac{N-2}{N-4}}\int_{\Omega_\epsilon}uv
-(2^*-1)\int_{\Omega_\epsilon}V_1^{2^*-2}uv.
\]
For each fixed $\epsilon$, the operator on $E_1$ associated with this
bilinear form is a compact perturbation of the Dirichlet isomorphism and
is therefore Fredholm of index zero. Its kernel is trivial for all
sufficiently small $\epsilon$. Otherwise, there would exist
$\epsilon_m\to0$, configuration vectors $\mathbf t_m\in O_\delta$, and nonzero
functions $\varphi_m$ in the corresponding moment spaces $E_{1,m}$ such that
\[
\mathcal Q_{1,m}(\varphi_m,v)=0
\qquad\text{for every }v\in E_{1,m},
\]
where $\mathcal Q_{1,m}$ denotes the preceding bilinear form with the $m$-th
parameters.
Since the annihilator of $E_{1,m}$ is generated by the corresponding
moment functionals, there exist coefficients $c_{j,m}$ such that
$(\varphi_m,(c_{j,m})_{j=0}^N)$ solves \eqref{eq:linear-first} with $h_1=0$.
Elliptic regularity gives $\varphi_m\in L^\infty(\Omega_{\epsilon_m})$.
After normalization by its nonzero $L^\infty$ norm, Lemma~\ref{lem1} gives a
contradiction. Thus the associated operator is invertible for all sufficiently small
$\epsilon$.

Given $h_1\in L^\infty(\Omega_\epsilon)$, there is a unique
$\varphi\in E_1$ such that
\[
\mathcal Q_1(\varphi,v)
=
\int_{P_2}h_1v
\qquad\text{for every }v\in E_1.
\]
The residual
\[
-\Delta\varphi
-\epsilon^{\frac{N-2}{N-4}}\varphi
-(2^*-1)V_1^{2^*-2}\varphi
-\mathbf1_{P_2}h_1
\]
annihilates $E_1$ and hence has a unique representation
$\sum_{j=0}^Nc_jW_{1,j}$. This proves existence and uniqueness for
\eqref{eq:linear-first}.

If the asserted $L^\infty$ estimate failed, a contradicting sequence could be
divided by the $L^\infty$ norm of its solution. The normalized solutions
would have unit norm, whereas the normalized right-hand sides would satisfy
$r_\epsilon^4\|h_1\|_{L^\infty(P_2)}\to0$, contradicting
Lemma~\ref{lem1}. Therefore
\[
\|\varphi\|_{L^\infty(\Omega_\epsilon)}
\le Cr_\epsilon^4\|h_1\|_{L^\infty(P_2)}.
\]
Finally, the uniformly invertible system \eqref{eq11} and the estimates used
in its derivation give
\[
\max_{0\le j\le N}|c_j|
\le C\left(
\epsilon\|\varphi\|_{L^\infty(\Omega_\epsilon)}
+r_\epsilon^4\|h_1\|_{L^\infty(P_2)}
\right),
\]
and hence the required multiplier bound. The proof of
Proposition~\ref{prop:linear-second} is the same, with
$P_2,h_1,c_j,\varphi$ replaced by $Q_2,h_2,d_j,\psi$.
\end{proof}

For later use, we also record the fixed-$\epsilon$ regularity behind the
preceding continuity assertion. For every $1<q<\infty$,
\[
L_i h\in E_i\cap W^{2,q}(\Omega_\epsilon)\cap W_0^{1,q}(\Omega_\epsilon),
\qquad i=1,2,
\]
whenever $h\in L^\infty(\Omega_\epsilon)$. This follows from the equations and
standard Dirichlet regularity. No $W^{2,q}$ bound uniform as
$\epsilon\to0$ is asserted or needed.

\medskip
\subsection{Comparison principles}
Let $\mathbf C_s>0$ denote the optimal Sobolev constant, so that
\[
\mathbf C_s\|u\|_{L^{2^*}(\mathbb R^N)}^2
\le \|\nabla u\|_{L^2(\mathbb R^N)}^2
\qquad\text{for every }u\in D^{1,2}(\mathbb R^N).
\]
We use the following comparison principle; see, for example,
\cite{Pucciserrin}.

\begin{Lem}\label{lem0}
Let $D\subset\mathbb R^N$ be an open set with $C^1$ boundary. Suppose that
$\Psi\in L^{N/2}(D)$ and
$\|\Psi\|_{L^{N/2}(D)}<\mathbf C_s$. Let
$v_1,v_2\in D^{1,2}_{\mathrm{loc}}(D)\cap C(\overline D)$ satisfy
$(v_2-v_1)_+\in D_0^{1,2}(D)$. If
\[
-\Delta v_1-\Psi v_1\ge -\Delta v_2-\Psi v_2
\quad\text{weakly in }D,
\]
and $v_1\ge v_2$ on $\partial D$, then $v_1\ge v_2$ in $D$.
\end{Lem}
\begin{proof}
Set $\omega=v_2-v_1$. Testing
$-\Delta\omega\le\Psi\omega$ with $\omega_+$ and extending $\omega_+$ by
zero outside $D$, we obtain
\[
\begin{aligned}
\mathbf C_s\|\omega_+\|_{L^{2^*}(D)}^2
&\le \int_D|\nabla\omega_+|^2
 \le \int_D\Psi\omega_+^2\\
&\le \|\Psi\|_{L^{N/2}(D)}
      \|\omega_+\|_{L^{2^*}(D)}^2.
\end{aligned}
\]
The strict inequality for $\|\Psi\|_{L^{N/2}(D)}$ forces $\omega_+=0$.
\end{proof}

\section{The nonlinear exterior problem}\label{sec:3}

\subsection{Truncated variational formulation}

We now implement the nonlinear exterior part of the decomposition. For the
first component, the inner trace is prescribed on $P_1$ and the equation is
solved in $\Omega_\epsilon\setminus P_1$; for the second component the
corresponding regions are $Q_1$ and $\Omega_\epsilon\setminus Q_1$. The larger
balls $P_2,Q_2$ are not boundary regions for this minimization. They provide
the overlap on which the dead-core estimate will eliminate the opposite
component and on which the inner projected problem will later be localized.

Fix $\delta_0\in\bigl(0,(N-2)/4\bigr)$ and define
\[
\Gamma_i(y):=\frac{1}{(1+|y-\tilde\xi_i|)^{N-2-\delta_0}},
\qquad i=1,2.
\]
Since $r_\epsilon\to\infty$,
\[
\|\Gamma_1\|_{L^{2^*}(\mathbb R^N\setminus P_1)}
+\|\Gamma_2\|_{L^{2^*}(\mathbb R^N\setminus Q_1)}
\longrightarrow0
\qquad\text{as }\epsilon\to0.
\]
Let $\zeta\in C^1(\mathbb R)$ be even, satisfy $0\le\zeta\le1$, and obey
\[
\zeta(t)=1\quad\text{for }|t|\le1,
\qquad
\zeta(t)=0\quad\text{for }|t|\ge2.
\]
With $\mathbf1_D$ denoting the indicator of $D$, set
\[
g_1(y,s):=\mathbf1_{P_1}(y)s
+\big(1-\mathbf1_{P_1}(y)\big)
\int_0^s\zeta\left(\frac{t}{\Gamma_1(y)}\right)dt
\]
and
\[
g_2(y,s):=\mathbf1_{Q_1}(y)s
+\big(1-\mathbf1_{Q_1}(y)\big)
\int_0^s\zeta\left(\frac{t}{\Gamma_2(y)}\right)dt.
\]

In the outer regions we also use the notation
\[
g_{1,s}(y,s):=\frac{\partial g_1}{\partial s}(y,s)=\zeta\left(\frac{s}{\Gamma_1(y)}\right),\qquad
y\in\Omega_\epsilon\setminus P_1,
\]
and
\[
g_{2,s}(y,s):=\frac{\partial g_2}{\partial s}(y,s)=\zeta\left(\frac{s}{\Gamma_2(y)}\right),\qquad
y\in\Omega_\epsilon\setminus Q_1.
\]
Thus \(0\le g_{i,s}\le1\). Once the estimate \(|u_i|\le\Gamma_i\) is proved in
the corresponding outer region, one has \(g_i(y,u_i)=u_i\) and
\(g_{i,s}(y,u_i)=1\) there. The truncation is used only to obtain a variationally stable outer problem: it leaves the original nonlinearities unchanged under the desired a priori bounds, while providing fixed \(L^{2^*}\)-dominating functions in the exterior regions.

In the corresponding outer region, we have
\[
\int_{0}^{|u_{i}|} \zeta\Big(\frac{s}{\Gamma_{i}}\Big)ds
=\Gamma_i\int_{0}^{\frac{|u_{i}|}{\Gamma_i} } \zeta(t)dt
\leq\min\{2\Gamma_i,|u_i|\}.
\]
For notational convenience, set
\[
\mathcal C_1:=P_1,
\qquad
\mathcal C_2:=Q_1.
\]
Thus, for \(y\in\Omega_\epsilon\setminus\mathcal C_i\),
\[
\operatorname{sign}(u_i)g_i(y,u_i)=|g_i(y,u_i)|\leq\min\{2\Gamma_i,|u_i|\}.
\]

The perturbation radius below is matched to the two-center barriers used in
Lemma~\ref{lem8}. After $N$ and $\tau$ have been fixed, choose first
\[
0<\vartheta<\min\left\{1,\frac N2-2\right\}
\]
sufficiently small. Then choose
\[
0<\theta'<\vartheta,
\qquad
0<\tilde\theta<\frac2{N-4},
\qquad
0<\bar\theta<N-2
\]
sufficiently small and, for $i=1,2$, set
\begin{align}\label{eq:auxiliary-exponents}
\alpha_i&:=\frac{N-2}{2(N-4)}+\theta',
&\qquad
\tilde\alpha_i&:=\frac2{N-4}-\tilde\theta,\nonumber\\
\theta_i&:=\frac N2-\vartheta,
&
\tilde\theta_i&:=N-2-\bar\theta.
\end{align}
Define
\[
\kappa_0:=\alpha_1+\frac{\theta_1}{B_\tau},
\qquad
\gamma_1:=\frac{N-2}{N-4}.
\]
Finally, choose $\delta_*>0$ sufficiently small and set
\[
\gamma_0:=\kappa_0+\delta_*.
\]
These exponents depend only on $(N,\tau)$ and are fixed throughout the
argument. They can be chosen simultaneously so that
\begin{equation}\label{gamma-range}
\gamma_0<\gamma_1<2\gamma_0,
\end{equation}
and so that the further strict inequalities used below hold. Each such
inequality is stated where it first enters the proof. Their simultaneous
compatibility, including the fixed-point margins, is verified in
Lemma~\ref{lem:exponent-choice} in Appendix~A. Since
$r_\epsilon=\epsilon^{-1/B_\tau}$, we have
\[
\epsilon^{\gamma_0}
=\epsilon^{\alpha_1+\delta_*}r_\epsilon^{-\theta_1}.
\]

Set
\[
\begin{split}
\Lambda_\epsilon:=\big\{(\varphi,\psi)\in C_0(\Omega_\epsilon)^2:\;&
\ell_{1,j}(\varphi)=0,\quad \ell_{2,j}(\psi)=0,
\quad j=0,\ldots,N,\\
&\|\varphi\|_{L^\infty(\Omega_\epsilon)}\le\epsilon^{\gamma_0},\quad
\|\psi\|_{L^\infty(\Omega_\epsilon)}\le\epsilon^{\gamma_0}\big\}.
\end{split}
\]
Thus no energy condition is included in the definition of
$\Lambda_\epsilon$. Since the moment functionals are continuous for the
uniform norm, $\Lambda_\epsilon$ is a closed subset of
$C_0(\Omega_\epsilon)^2$ and is therefore complete in the product uniform
norm. The radius $\epsilon^{\gamma_0}$ is matched to the decay at
$\partial P_1$ and $\partial Q_1$ through the identity above, while
\eqref{gamma-range} provides the strict margins needed in the nonlinear
fixed-point estimates.

The exterior minimization problem is first solved for the energy-compatible
subclass $\Lambda_\epsilon\cap H_0^1(\Omega_\epsilon)^2$. After the exterior
estimates have been established, the solution operator is extended by uniform
approximation to all of $\Lambda_\epsilon$.

\medskip

Let $(\varphi_0,\psi_0)\in
\Lambda_\epsilon\cap H_0^1(\Omega_\epsilon)^2$ and set
\[
u_{1,0}:=V_1+\varphi_0,
\qquad
u_{2,0}:=V_2+\psi_0.
\]
Thus the moment conditions are imposed on the perturbation
$(u_{1,0},u_{2,0})-(V_1,V_2)=(\varphi_0,\psi_0)$ rather than on the full
inner data. We seek a solution of the associated exterior problem below.

\begin{equation}\label{eqout}
\begin{cases}
-\Delta u_{1} -\epsilon^{\frac{N-2}{N-4}} u_1= \lvert u_{1}\rvert^{2^{*}-2}u_{1}+\beta \lvert u_{2}\rvert^{\frac{2^{*}}{2}}\lvert u_{1}\rvert^{\frac{2^{*}}{2}-2}u_{1},  &~y\in \Omega_\epsilon\setminus P_1,\\[3mm]
-\Delta u_{2} -\epsilon^{\frac{N-2}{N-4}} u_2=\lvert u_{2}\rvert^{2^{*}-2}u_{2}+\beta \lvert u_{1}\rvert^{\frac{2^{*}}{2}}\lvert u_{2}\rvert^{\frac{2^{*}}{2}-2}u_{2},  &~y\in\Omega_\epsilon\setminus Q_{1},\\[3mm]
u_{1}=u_{1,0}\quad\text{in }P_{1},\qquad u_{2}=u_{2,0}\quad\text{in }Q_{1}.
\end{cases}
\end{equation}

For prescribed perturbative inner data \((\varphi_0,\psi_0)\in\Lambda_\epsilon\cap H_0^1(\Omega_\epsilon)^2\), we consider the minimization problem in
\[
\begin{split}
\mathcal M(u_{1,0},u_{2,0}) :=
\{(u_1,u_2)&\in H_0^1(\Omega_\epsilon)\times H_0^1(\Omega_\epsilon):
 u_1=u_{1,0}\ \text{in }P_1,\ u_2=u_{2,0}\ \text{in }Q_1,\vphantom{\int}\\
&\qquad\qquad (u_1-V_1,u_2-V_2)\in\mathbb E\}.
\end{split}
\]
Equivalently,
\[
\mathcal M(u_{1,0},u_{2,0})=(V_1,V_2)+(\varphi_0,\psi_0)
+\Big\{(\phi_1,\phi_2):\phi_1=0\text{ in }P_1,\ \phi_2=0\text{ in }Q_1,\ (\phi_1,\phi_2)\in\mathbb E\Big\}.
\]
This convention is used throughout the paper: the finite-dimensional moment
conditions are imposed on the perturbations around the projected bubbles. Since
the moment functions are supported in the corresponding inner regions, outer
variations preserve these constraints.

Define $I:\mathcal M(u_{1,0},u_{2,0})\to\mathbb R$ by
\begin{equation*}
\begin{split}
I(u_{1},u_{2}):=I_\epsilon(u_{1},u_{2})=&\frac{1}{2}\int_{\Omega_\epsilon} \big(\lvert\nabla u_{1}\rvert^{2}+\lvert\nabla u_{2}\rvert^{2}\big)
-\frac{1}{2}\epsilon^{\frac{N-2}{N-4}}\int_{\Omega_\epsilon} \big(\lvert g_{1}(y,u_{1})\rvert^{2}+\lvert g_{2}(y,u_{2})\rvert^{2}\big)\\
&-\frac{1}{2^{*}}\int_{\Omega_\epsilon}\Big(\lvert g_{1}(y,u_{1})\rvert^{2^{*}}+\lvert g_{2}(y,u_{2})\rvert^{2^{*}}+2\beta\lvert g_{1}(y,u_{1})\rvert^{\frac{2^{*}}{2}}\lvert g_{2}(y,u_{2})\rvert^{\frac{2^{*}}{2}}\Big).
\end{split}
\end{equation*}

We shall solve the minimization problem
\begin{equation}\label{min}
\min\{I(u_1,u_2):(u_1,u_2)\in \mathcal M(u_{1,0},u_{2,0})\}.
\end{equation}
\medskip

\begin{Rem}
A critical point of the truncated functional satisfies the Euler--Lagrange
system \eqref{eqg} below. If the a priori estimate \(|u_i|\le\Gamma_i\) holds
in the corresponding outer region, then \(g_i(y,u_i)=u_i\) and
\(g_{i,s}(y,u_i)=1\), so \eqref{eqg} reduces to the original outer problem
\eqref{eqout}.
\end{Rem}
\medskip

The next lemma gives existence for the auxiliary minimization problem.

\begin{Lem}\label{lem3}
Fix \(\epsilon > 0\) sufficiently small. Then the functional \(I\) is weakly
lower semicontinuous and coercive on \( \mathcal{M}(u_{1,0}, u_{2,0}) \).
Consequently, the minimization problem \eqref{min} admits a minimizer.
\end{Lem}

\begin{proof}
We first record the weak closedness of the admissible class. The spaces \(E_i\) are finite intersections of kernels of continuous linear functionals;
hence \(\mathbb E=E_1\times E_2\) is a closed linear subspace of
\(H_0^1(\Omega_\epsilon)\times H_0^1(\Omega_\epsilon)\). Moreover, the maps
\[
H_0^1(\Omega_\epsilon)\longrightarrow L^2(P_1),\qquad
H_0^1(\Omega_\epsilon)\longrightarrow L^2(Q_1),
\]
obtained by restricting a Sobolev function to the indicated subdomain are
bounded. Therefore the conditions
\(u_1=u_{1,0}\) in \(P_1\), \(u_2=u_{2,0}\) in \(Q_1\), and
\((u_1-V_1,u_2-V_2)\in\mathbb E\) define closed affine constraints. Thus
\(\mathcal M(u_{1,0},u_{2,0})\) is closed and convex. Since closed convex subsets of a Banach space are weakly closed, \(\mathcal M(u_{1,0},u_{2,0})\) is weakly closed.

We now prove weak lower semicontinuity. Let
\((u_{1,m},u_{2,m})\subset\mathcal M(u_{1,0},u_{2,0})\) and suppose that
\[
(u_{1,m},u_{2,m})\rightharpoonup(u_1,u_2)
\quad\text{in }H_0^1(\Omega_\epsilon)\times H_0^1(\Omega_\epsilon).
\]
If the liminf is finite, we pass to a subsequence, not relabeled, along which
\(I(u_{1,m},u_{2,m})\) converges to the liminf. Since \(\epsilon>0\) is fixed, \(\Omega_\epsilon\) is bounded,
and the compact embedding
\[
H_0^1(\Omega_\epsilon)\hookrightarrow L^q(\Omega_\epsilon),
\qquad 1\le q<2^*,
\]
holds, we may also assume that
\[
u_{i,m}\to u_i\quad\text{a.e.\ in }\Omega_\epsilon
\quad\text{and strongly in }L^q(\Omega_\epsilon),\quad 1\le q<2^*,
\]
for \(i=1,2\). By the weak closedness just proved,
\((u_1,u_2)\in\mathcal M(u_{1,0},u_{2,0})\).

For almost every \(y\), the map
\(s\mapsto g_i(y,s)\) is continuous, and hence
\[
g_i(y,u_{i,m}(y))\to g_i(y,u_i(y))\qquad\text{for a.e. }y\in\Omega_\epsilon.
\]
Furthermore,
\[
|g_i(y,u_{i,m})|\le |u_{i,0}|\mathbf 1_{\mathcal C_i}+2\Gamma_i\mathbf 1_{\Omega_\epsilon\setminus\mathcal C_i}
=:G_i(y),
\]
where \(G_i\in L^{2^*}(\Omega_\epsilon)\cap L^2(\Omega_\epsilon)\). The dominated
convergence theorem gives
\[
\int_{\Omega_\epsilon}|g_i(y,u_{i,m})|^2\to
\int_{\Omega_\epsilon}|g_i(y,u_i)|^2,
\qquad
\int_{\Omega_\epsilon}|g_i(y,u_{i,m})|^{2^*}\to
\int_{\Omega_\epsilon}|g_i(y,u_i)|^{2^*}.
\]
Since $p=2^*/2$ and \(G_1^pG_2^p\in L^1(\Omega_\epsilon)\), the same argument
also yields
\[
\int_{\Omega_\epsilon}|g_1(y,u_{1,m})|^p|g_2(y,u_{2,m})|^p
\to
\int_{\Omega_\epsilon}|g_1(y,u_1)|^p|g_2(y,u_2)|^p.
\]
Thus all lower-order and truncated critical terms are weakly continuous along
this subsequence. The Dirichlet term is weakly lower semicontinuous, and
therefore
\[
I(u_1,u_2)\le \liminf_{m\to\infty}I(u_{1,m},u_{2,m}).
\]
This proves the weak lower semicontinuity of \(I\) on \(\mathcal M(u_{1,0},u_{2,0})\).

It remains to prove coercivity. Since \(\beta<0\), the coupling term contributes
nonnegatively to the lower bound for \(I\). Using this estimate in \(\Omega_\epsilon\setminus\mathcal C_i\), and recalling that
\(g_i(y,u_i)=u_{i,0}\) in \(\mathcal C_i\), we obtain
\[
\int_{\Omega_\epsilon}\big(|g_i(y,u_i)|^2+|g_i(y,u_i)|^{2^*}\big)\le C_\epsilon,
\]
where \(C_\epsilon<\infty\) depends on the prescribed inner data and on the
cut-off functions \(\Gamma_i\), but not on \((u_1,u_2)\in\mathcal M(u_{1,0},u_{2,0})\).
Consequently,
\[
I(u_1,u_2)
\ge
\frac{1}{2}\sum_{i=1}^2\int_{\Omega_\epsilon}|\nabla u_i|^2
-C_\epsilon.
\]
Hence \(I(u_1,u_2)\to+\infty\) whenever
\(\|(u_1,u_2)\|_{H_0^1\times H_0^1}\to+\infty\), proving coercivity.

Finally, \(\mathcal M(u_{1,0},u_{2,0})\) is nonempty, for instance it contains
\((V_1+\varphi_0,V_2+\psi_0)\). Let \((u_{1,m},u_{2,m})\) be a minimizing
sequence. By coercivity it is bounded in the Hilbert space
\(H_0^1(\Omega_\epsilon)\times H_0^1(\Omega_\epsilon)\); hence, after passing to
a subsequence, it converges weakly to some element of \(\mathcal M(u_{1,0},u_{2,0})\).
The weak lower semicontinuity proved above shows that this weak limit attains
\eqref{min}.
\end{proof}

By Lemma~\ref{lem3}, the problem \eqref{min} admits a minimizer. Variations supported in \(\Omega_\epsilon\setminus P_1\) for the first component and in \(\Omega_\epsilon\setminus Q_1\) for the second component do not change the prescribed inner data. They also do not change the moment constraints, because the functions \(Y_{1,j}\) and \(Y_{2,j}\) are supported in \(P_1\) and \(Q_1\), respectively. Hence no additional Lagrange multiplier appears in the outer Euler--Lagrange equations, and the minimizer satisfies
\begin{equation} \label{eqg}
\left\{
\begin{aligned}
-\Delta u_{1}
&= g_{1,s}(y,u_1)\Big[\epsilon^{\frac{N-2}{N-4}} g_{1}(y,u_{1})
+\lvert g_{1}(y,u_{1})\rvert^{2^{*}-2} g_{1}(y,u_{1}) \\
&\hspace{27mm}
+\beta \lvert g_{2}(y,u_{2})\rvert^{\frac{2^{*}}{2}}
\lvert g_{1}(y,u_{1})\rvert^{\frac{2^{*}}{2}-2}g_{1}(y,u_{1})\Big],
&& y\in \Omega_\epsilon\setminus P_1,\\[2mm]
-\Delta u_{2}
&= g_{2,s}(y,u_2)\Big[\epsilon^{\frac{N-2}{N-4}} g_{2}(y,u_{2})
+\lvert g_{2}(y,u_{2})\rvert^{2^{*}-2}g_{2}(y,u_{2}) \\
&\hspace{27mm}
+\beta \lvert g_{1}(y,u_{1})\rvert^{\frac{2^{*}}{2}}
\lvert g_{2}(y,u_{2})\rvert^{\frac{2^{*}}{2}-2}g_{2}(y,u_{2})\Big],
&& y\in\Omega_\epsilon\setminus Q_{1},\\[2mm]
u_{1}&=u_{1,0} \quad \text{in }P_{1},
& u_{2}&=u_{2,0} \quad \text{in }Q_{1}.
\end{aligned}
\right.
\end{equation}

\medskip

The prescribed inner data belong to $H^1\cap C$ on the corresponding
balls. After subtracting their harmonic liftings to the exterior domains, the
remaining functions satisfy zero Dirichlet conditions and equations with
bounded right-hand sides for each fixed $\epsilon$. Standard Dirichlet
regularity therefore gives continuous representatives of the minimizers and,
for every finite $q>1$,
\[
u_1\in W^{2,q}_{\mathrm{loc}}(\Omega_\epsilon\setminus\overline{P_1}),
\qquad
u_2\in W^{2,q}_{\mathrm{loc}}(\Omega_\epsilon\setminus\overline{Q_1}).
\]
The exterior representatives agree continuously with the prescribed data on
$\partial P_1$ and $\partial Q_1$. We henceforth use these representatives;
in particular,
\[
(u_1,u_2)-(V_1,V_2)\in C_0(\Omega_\epsilon)^2.
\]

\medskip

\subsection{Estimates for the exterior minimizer}

We now establish estimates for the minimizer $(u_1,u_2)$ of \eqref{min}
that satisfies \eqref{eqg}.

\begin{Lem}\label{lem7}
There exists a constant $C>0$, independent of $\epsilon$ and of the
admissible inner data, such that
\begin{equation}\label{decay}
|u_i(y)|\le\frac{C}{(1+|y-\tilde\xi_i|)^{N-2}}
\qquad\text{for }y\in\Omega_\epsilon,\quad i=1,2.
\end{equation}

\end{Lem}

\begin{proof}
Since \(\beta<0\) and \(0\le g_{i,s}\le1\), Kato's inequality applied to
\eqref{eqg} gives
\begin{equation}\label{ineqg}
\begin{cases}
\displaystyle -\Delta |u_{1}| \leq \epsilon^{\frac{N-2}{N-4}} |g_{1}(y,u_{1})|+\lvert g_{1}(y,u_{1})\rvert^{2^{*}-1}, & y\in \Omega_\epsilon\setminus P_1,\\[3mm]
\displaystyle -\Delta |u_{2}| \leq \epsilon^{\frac{N-2}{N-4}} | g_{2}(y,u_{2})|+\lvert g_{2}(y,u_{2})\rvert^{2^{*}-1}, & y\in\Omega_\epsilon\setminus Q_{1}.
\end{cases}
\end{equation}
First, since $|g_i(y,u_i)|\leq\min\{2\Gamma_i, |u_i|\}$ for $i=1,2$, the first inequality in \eqref{ineqg} gives, in $\Omega_\epsilon\setminus P_1$,
\begin{align}\label{eqclaim-u1}
-\Delta|u_1|-\Psi(y)|u_1|\leq 0,
\end{align}
where
\begin{align*}
&\Psi(y)=\epsilon^{\frac{N-2}{N-4}}+(2\Gamma_1)^{2^*-2}
\end{align*}
satisfies \[\int_{\Omega_\epsilon\setminus P_1}\Psi^{\frac{N}{2}}\rightarrow0\quad\text{as }\epsilon\rightarrow0.\]
To apply Lemma~\ref{lem0}, let
\[
\varpi(y) = \frac{1}{(1+a_\Omega\lvert y-\tilde\xi_1\rvert)^{N-2}},
\]
where $a_\Omega:=R_\Omega^{-3}$.
We claim that, in $\Omega_\epsilon\setminus P_{1}$,
\begin{align}\label{eqclaim-varpi}-\Delta\varpi-\Psi(y)\varpi\geq0.
\end{align}
Indeed, a direct computation gives
\[
-\Delta \varpi = \varpi \frac{(N-2)(N-1)a_\Omega}{\lvert y-\tilde\xi_1\rvert(1+a_\Omega\lvert y-\tilde\xi_1\rvert)^{2}}.
\]
Let $r = \lvert y-\tilde\xi_1\rvert \ge r_\epsilon$. We need to show that
\[\frac{(N-2)(N-1)a_\Omega}{r(1+a_\Omega r)^2}>\epsilon^{\frac{N-2}{N-4}}+(2\Gamma_1)^{2^*-2}.\]

First, since $\delta_0<(N-2)/4$, we have
\[
(2^*-2)(N-2-\delta_0)
=4-\frac{4\delta_0}{N-2}>3.
\]
Therefore, as $r\geq r_\epsilon\to +\infty$,
\[
\frac{(N-2)(N-1)a_\Omega}{r(1+a_\Omega r)^2} > 2(2\Gamma_1)^{2^*-2}.
\]
To absorb the lower-order linear term, we note that in $\Omega_\epsilon$, the distance $r$ is bounded by $R_\Omega \epsilon^{-\frac{1}{N-4}}$. For $\epsilon$ sufficiently small, we have $a_\Omega r \ge a_\Omega r_\epsilon \to \infty$, which implies $1+a_\Omega r \le 2a_\Omega r$. Thus,
\[ \frac{(N-2)(N-1)a_\Omega}{r(1+a_\Omega r)^2}\geq
\frac{(N-2)(N-1)}{4 a_\Omega r^3} \ge \frac{(N-2)(N-1)}{4} \epsilon^{\frac{3}{N-4}}
>2 \epsilon^{\frac{N-2}{N-4}}.
\]
 For $N>5$ the last inequality follows from the strict difference of the
powers. When $N=5$ the two powers are both $\epsilon^3$, and the retained
coefficient is essential: $(N-2)(N-1)/4=3>2$. Thus
\eqref{eqclaim-varpi} holds in every dimension $N\ge5$.

On $\partial P_1$, where $r=r_\epsilon\to\infty$, there exists
$C_0>0$ such that
\[
|u_1(y)|=|u_{1,0}(y)|\le U_1(y)+\epsilon^{\gamma_0}
\le C_0\varpi(y).
\]
Moreover, $|u_1|=0\le C_0\varpi$ on $\partial\Omega_\epsilon$, and the
positive part required in Lemma~\ref{lem0} belongs to the corresponding
zero-trace energy space. Applying that lemma with $v_1=C_0\varpi$ and
$v_2=|u_1|$ proves the first estimate in \eqref{decay}. The second follows in
the same way.
\end{proof}
 \medskip

By Lemma~\ref{lem7}, $\lvert u_{1}\rvert\leq\Gamma_{1}$ and $\lvert u_{2}\rvert\leq\Gamma_{2}$. Hence the truncations are inactive in the corresponding outer regions, and
the minimizer $(u_{1},u_{2})$ is a solution of the original outer problem \eqref{eqout}.

\begin{Cor}\label{correal}
For sufficiently small \( \epsilon \), the minimum in \eqref{min} is attained at a solution \( (u_{1}, u_{2}) \) of \eqref{eqout}, provided that the perturbative inner data satisfy \( (u_{1,0},u_{2,0})-(V_1,V_2)\in\Lambda_\epsilon\cap H_0^1(\Omega_\epsilon)^2 \).
\end{Cor}
\medskip

\medskip

We also need estimates for the perturbations.
\begin{Lem}\label{lem8}
There exists a constant $C>0$, independent of $\epsilon$ and of the
energy-compatible inner data satisfying
$(u_{1,0},u_{2,0})-(V_1,V_2)\in\Lambda_\epsilon\cap H_0^1(\Omega_\epsilon)^2$, such that
\begin{align*}
\lvert u_{1}(y)-V_{1}(y)\rvert
&\leq \frac{ C\epsilon^{\alpha_1}}{(1+\lvert y-\tilde\xi_1\rvert)^{ \theta_1}}
+\frac{ C\epsilon^{\tilde\alpha_1}}{(1+\lvert y-\tilde\xi_2\rvert)^{\tilde\theta_1}}, \\
\lvert u_{2}(y)-V_{2}(y)\rvert
&\leq \frac{ C\epsilon^{\alpha_2}}{(1+\lvert y-\tilde\xi_2\rvert)^{ \theta_2}}
+\frac{ C\epsilon^{\tilde\alpha_2}}{(1+\lvert y-\tilde\xi_1\rvert)^{\tilde\theta_2}},
\end{align*}
for all $y \in \Omega_\epsilon$, and
\begin{align*}
\lvert\nabla(u_{1}(y)-V_{1}(y))\rvert
&\leq \frac{C\epsilon^{\alpha_1}}{(1+\lvert y-\tilde\xi_1\rvert)^{ \theta_1}}
+\frac{C\epsilon^{\tilde\alpha_1}}{(1+\lvert y-\tilde\xi_2\rvert)^{\tilde\theta_1}}, \quad y \in \Omega_\epsilon\setminus B_{r_\epsilon+1}(\tilde\xi_1), \\
\lvert\nabla(u_{2}(y)-V_{2}(y))\rvert
&\leq \frac{C\epsilon^{\alpha_2}}{(1+\lvert y-\tilde\xi_2\rvert)^{ \theta_2}}
+\frac{C\epsilon^{\tilde\alpha_2}}{(1+\lvert y-\tilde\xi_1\rvert)^{\tilde\theta_2}}, \quad y \in \Omega_\epsilon\setminus B_{r_\epsilon+1}(\tilde\xi_2),
\end{align*}
where the exponents are those fixed in \eqref{eq:auxiliary-exponents}; in particular,
$\theta_i,\tilde\theta_i\in(0,N-2)$ for $i=1,2$.
 \end{Lem}
 \begin{proof}

Let $(\varphi,\psi)=(u_{1}, u_{2})-(V_1,V_2)$. By \eqref{eqout}, we have
\begin{equation}\label{eq17}
\begin{cases}
-\Delta \varphi =\epsilon^{\frac{N-2}{N-4}}\varphi+\epsilon^{\frac{N-2}{N-4}}V_1+\lvert u_{1}\rvert^{2^{*}-2}u_{1}-U^{2^*-1}_{1}+\beta \lvert u_{2}\rvert^{\frac{2^{*}}{2}}\lvert u_{1}\rvert^{\frac{2^{*}}{2}-2}u_{1},  &~\text{in}~\Omega_\epsilon\setminus P_{1},\\[3mm]
-\Delta \psi =\epsilon^{\frac{N-2}{N-4}}\psi+ \epsilon^{\frac{N-2}{N-4}}V_2+\lvert u_{2}\rvert^{2^{*}-2}u_{2}-U^{2^*-1}_{2}+\beta \lvert u_{1}\rvert^{\frac{2^{*}}{2}}\lvert u_{2}\rvert^{\frac{2^{*}}{2}-2}u_{2},  &~\text{in}~\Omega_\epsilon\setminus Q_{1},\\[3mm]
\varphi=\varphi_{0}\quad\text{in }P_{1},\qquad \psi=\psi_{0}\quad\text{in }Q_{1}.
\end{cases}
\end{equation}

To estimate $|u_1-V_1|$, we apply Lemma~\ref{lem0} to the following inequality in $\Omega_\epsilon\setminus P_{1}$:
\begin{align}\label{eqf}
-\Delta |\varphi|-\Psi(y)|\varphi|\leq f_1+f_2+f_3,
\end{align}
where for some $C>0$,
 \begin{align*}
&\Psi(y)=\epsilon^{\frac{N-2}{N-4}}+C(|u_1|^{2^*-2}+|V_1|^{2^*-2})\\
&f_1=\epsilon^{\frac{N-2}{N-4}}V_1\\
&f_2=\big|V_1^{2^*-1}-U_1^{2^*-1}\big|\\
&f_3=|\beta| \lvert u_{2}\rvert^{\frac{2^{*}}{2}}\lvert u_{1}\rvert^{\frac{2^{*}}{2}-1}.
\end{align*}

Take
\[
\omega_1(y)=
 A_{\mathrm b,1}\frac{\epsilon^{\alpha_1}}{(1+|y-\tilde\xi_1|)^{\theta_1}}
+A_{\mathrm b,2}\frac{\epsilon^{\tilde\alpha_1}}{(1+|y-\tilde\xi_2|)^{\tilde\theta_1}},
\]
where \(A_{\mathrm b,1},A_{\mathrm b,2}>0\) will be chosen below. A direct computation gives
\begin{align}\label{eqomega1}
-\Delta\omega_1(y)
=&A_{\mathrm b,1}\frac{\theta_1\epsilon^{\alpha_1}}{(1+|y-\tilde\xi_1|)^{\theta_1+2}}
\left((N-2-\theta_1)+\frac{N-1}{|y-\tilde\xi_1|}\right)\nonumber\\
&+A_{\mathrm b,2}\frac{\tilde\theta_1\epsilon^{\tilde\alpha_1}}{(1+|y-\tilde\xi_2|)^{\tilde\theta_1+2}}
\left((N-2-\tilde\theta_1)+\frac{N-1}{|y-\tilde\xi_2|}\right).
\end{align}

We now prove that
\begin{align}\label{in1}
-\Delta\omega_1(y) - \Psi(y)\omega_1 \geq f_1 + f_2 + f_3,
\end{align}
by showing that the right-hand side of \eqref{eqomega1} dominates each of the four terms \(4\Psi(y)\omega_1\), \(4f_1\), \(4f_2\), and \(4f_3\).

First, for \(i=1,2\), since \(|y-\tilde\xi_i|\le C\epsilon^{-1/(N-4)}\) and
\(N>4\), we have
\[
(1+|y-\tilde\xi_i|)^{-2}\ge C\epsilon^{2/(N-4)}
\gg \epsilon^{(N-2)/(N-4)}.
\]
 For sufficiently small \(\epsilon\), this gives
\begin{align}\label{eq:barrier-linear-term}
-\Delta\omega_1
\ge C\epsilon^{\frac{N-2}{N-4}}\omega_1.
\end{align}
Moreover, using the admissible range \(\tau\in\mathcal I_N\) in
\eqref{tau-range} and the definitions of
\(\alpha_1,\tilde\alpha_1,\theta_1,\tilde\theta_1\), one obtains
\begin{align}\label{eq:barrier-bubble-potential}
-\Delta\omega_1
\ge
\frac{C}{(1+|y-\tilde\xi_1|)^4}\omega_1.
\end{align}
Indeed, the first part of \(\omega_1\) is immediate from
\((1+|y-\tilde\xi_1|)^{-2}\). For the mixed part, we fix a large constant $K_{\mathrm b}>0$ depending on $\bar\theta$ and split
\(\Omega_\epsilon\setminus P_1\) into the two complementary regions:
\[
\mathcal{R}_A = \{y : |y-\tilde\xi_1|^2 \ge K_{\mathrm b}|y-\tilde\xi_2|\} \quad \text{and} \quad \mathcal{R}_B = \{y : |y-\tilde\xi_1|^2 < K_{\mathrm b}|y-\tilde\xi_2|\}.
\]
In $\mathcal{R}_A$, the mixed term is absorbed by the second part of $-\Delta\omega_1$ purely through algebraic decay. The region $\mathcal{R}_B$ is the only one where the parameter restrictions enter, as the mixed term must be absorbed by the first part of $-\Delta\omega_1$ using the smallness of $\epsilon$. Lemma~\ref{lem:exponent-choice} guarantees that the losses can be fixed so that
\begin{equation}\label{eq:mixed-barrier-condition}
\alpha_1-\tilde\alpha_1
<\frac{2\tilde\theta_1+2-\theta_1}{B_\tau}.
\end{equation}
After substituting
\(\theta_1=N/2-\vartheta\), \(\tilde\theta_1=N-2-\bar\theta\),
\(\alpha_1=(N-2)/(2(N-4))+\theta'\), and
\(\tilde\alpha_1=2/(N-4)-\tilde\theta\), this condition, in the limit
\(\tilde\theta,\theta',\bar\theta\to0^+\), becomes
\[
\frac{N-6}{2(N-4)}<
\frac{\frac{3N}{2}-2+\vartheta}{(2+\tau)(N-4)}.
\]
For \(N=5,6\), the limiting left-hand side is nonpositive, and the condition is
therefore automatic after the losses have been chosen sufficiently small.
For \(N\ge7\), it is sufficient to impose
\[
\tau<\frac{N+8}{N-6},
\]
as in \eqref{tau-range}. Thus the mixed-barrier estimate is valid for all
\(\tau\in\mathcal I_N\). This is the only place where the upper restriction on
\(\tau\) enters the outer-region estimate. Combining
\eqref{eq:barrier-linear-term} and \eqref{eq:barrier-bubble-potential}, and using
\(|u_1|^{2^*-2}+|V_1|^{2^*-2}\le C(1+|y-\tilde\xi_1|)^{-4}\), we obtain
\begin{align}\label{Psiomega1}
-\Delta\omega_1\ge 4\Psi\omega_1,
\end{align}
 for sufficiently small \(\epsilon\).

For \(f_1\), the choice \(\theta_1=N/2-\vartheta\) with
\(0<\theta'<\vartheta\) gives the uniform estimate
\[
\frac{A_{\mathrm b,1}\epsilon^{\alpha_1}}{(1+|y-\tilde\xi_1|)^{\theta_1+2}}
\ge
C\frac{\epsilon^{\frac{N-2}{N-4}}}{(1+|y-\tilde\xi_1|)^{N-2}},
\qquad y\in\Omega_\epsilon\setminus P_1.
\]
The small loss \(\vartheta\) in \(\theta_1\) creates a strict margin in the
borderline case \(N=5\). Therefore
\begin{align}\label{f1}
-\Delta\omega_1\ge4f_1.
\end{align}

For \(f_2\), the projection expansion gives, uniformly in the admissible parameters,
\[
V_1(y)=U_1(y)-C_N\lambda_1^{\frac{N-2}{2}}
\epsilon^{\frac{N-2}{N-4}}
H_\Omega\big(\epsilon^{\frac{1}{N-4}}y,\xi_1\big)
+O\left(\epsilon^{\frac{N}{N-4}}\right).
\]
Together with the mean-value theorem, this yields
\[
|V_1^{2^*-1}-U_1^{2^*-1}|
\le
C U_1^{2^*-2}\epsilon^{\frac{N-2}{N-4}}
H_\Omega\big(\epsilon^{\frac{1}{N-4}}y,\xi_1\big)
+C\epsilon^{\frac{N}{N-4}}U_1^{2^*-2}.
\]
The choice made in Lemma~\ref{lem:exponent-choice} gives
\begin{equation}\label{eq:own-center-barrier-condition}
\alpha_1<\frac{N-\theta_1}{N-4}.
\end{equation}
Indeed,
\[
\frac{N-\theta_1}{N-4}
-\frac{N-2}{2(N-4)}
=\frac{1+\vartheta}{N-4}>0.
\]
Using \eqref{eq:own-center-barrier-condition}, the first part of
$-\Delta\omega_1$ dominates the right-hand side above. Consequently,
\begin{align}\label{f2}
-\Delta\omega_1\ge4f_2.
\end{align}
Next, we estimate the coupling term. Set
\[
\rho_i(y):=1+|y-\tilde\xi_i|,\qquad
d_\epsilon=|\tilde\xi_1-\tilde\xi_2|.
\]
 By Lemma~\ref{lem7},
\[
f_3\le C \rho_1^{-2}\rho_2^{-N}.
\]
Moreover, from the definition of \(O_\delta\),
\[
\delta\,\epsilon^{-1/(N-4)}
\le d_\epsilon
\le R_\Omega\epsilon^{-1/(N-4)}.
\]
We split the outer region into \(\rho_1\le d_\epsilon/2\) and
\(\rho_1>d_\epsilon/2\).  In the first region, \(\rho_2\ge c d_\epsilon\), and the
first term on the right-hand side of \eqref{eqomega1} gives
\begin{align*}
A_{\mathrm b,1}\epsilon^{\alpha_1}\rho_1^{-\theta_1-2}
&\ge C d_\epsilon^{-N}\rho_1^{-2}
\ge C \rho_1^{-2}\rho_2^{-N}.
\end{align*}
Here we used
\[
\alpha_1<\frac{N-\theta_1}{N-4},
\]
which is \eqref{eq:own-center-barrier-condition}.  In the second
region, \(\rho_1^{-2}\le C d_\epsilon^{-2}\), and the second term in
\eqref{eqomega1} gives
\begin{align*}
A_{\mathrm b,2}\epsilon^{\tilde\alpha_1}\rho_2^{-\tilde\theta_1-2}
&=A_{\mathrm b,2}\epsilon^{\tilde\alpha_1}\rho_2^{-N+\bar\theta}
\ge C d_\epsilon^{-2}\rho_2^{-N}
\ge C \rho_1^{-2}\rho_2^{-N},
\end{align*}
because \(\tilde\alpha_1<2/(N-4)\)  and \(\rho_2^{\bar\theta}\ge1\).  Taking
\(A_{\mathrm b,1},A_{\mathrm b,2}\) sufficiently large, we therefore obtain
\begin{align}\label{f3}
-\Delta\omega_1\ge4f_3.
\end{align}

From \eqref{Psiomega1}, \eqref{f1}, \eqref{f2} and \eqref{f3}, we obtain \eqref{in1}.

Combining \eqref{eqf} and \eqref{in1} yields
\[
-\Delta|\varphi|-\Psi|\varphi|
\le -\Delta\omega_1-\Psi\omega_1
\quad\text{in }\mathcal O_{1,\epsilon}:=\Omega_\epsilon\setminus\overline{P_1}.
\]
Moreover, Lemma~\ref{lem7} and the bubble-tail estimate imply
\[
\begin{aligned}
\|\Psi\|_{L^{N/2}(\mathcal O_{1,\epsilon})}
&\le C\epsilon^{\frac{N-2}{N-4}}|\Omega_\epsilon|^{2/N}
 +C\|u_1\|_{L^{2^*}(\mathcal O_{1,\epsilon})}^{2^*-2}
 +C\|V_1\|_{L^{2^*}(\mathcal O_{1,\epsilon})}^{2^*-2}\\
&=O(\epsilon)+o(1)<\mathbf C_s
\end{aligned}
\]
for sufficiently small $\epsilon$, uniformly in the admissible parameters.
On $\partial P_1$, the definition of $\Lambda_\epsilon$ gives
\[
|\varphi|\le\epsilon^{\gamma_0}
=\epsilon^{\alpha_1+\delta_*}r_\epsilon^{-\theta_1}
\le \frac{\epsilon^{\alpha_1}}
{(1+|y-\tilde\xi_1|)^{\theta_1}}
+\frac{\epsilon^{\tilde\alpha_1}}
{(1+|y-\tilde\xi_2|)^{\tilde\theta_1}}
\le\omega_1.
\]
On $\partial\Omega_\epsilon$, one has
$|\varphi|=0\le\omega_1$, and
$(|\varphi|-\omega_1)_+\in D_0^{1,2}(\mathcal O_{1,\epsilon})$. Lemma~\ref{lem0} therefore
gives $|\varphi|\le\omega_1$ in $\mathcal O_{1,\epsilon}$. The same argument applies to
$\psi$.

\medskip

We next consider the first equation in \eqref{eq17}. For
\(y\in\Omega_\epsilon\setminus B_{r_\epsilon+1}(\tilde\xi_1)\), set
\(\mathcal B_y:=B_1(y)\cap\Omega_\epsilon\). If \(B_1(y)\subset\Omega_\epsilon\),
we use the interior gradient estimate; otherwise we use the corresponding
boundary estimate for the homogeneous Dirichlet problem after flattening
\(\partial\Omega_\epsilon\). Since \(\Omega\) is smooth, the dilated domains
\(\Omega_\epsilon\) admit uniformly controlled boundary charts, and hence the
constants are independent of \(\epsilon\). Standard elliptic estimates
\cite[Chapters 3 and 8]{GT01} give

 \begin{equation*}
\begin{split}
\displaystyle \lvert \nabla(u_{1}(y)-V_1(y))\rvert\leq& C\Big(\sup_{\mathcal B_y}\lvert u_{1}(x)-V_1(x)\rvert+\epsilon^{\frac{N-2}{N-4}}\sup_{\mathcal B_y}\lvert u_{1}(x)-V_1(x)\rvert\\& \displaystyle +\epsilon^{\frac{N-2}{N-4}}\sup_{\mathcal B_y}|V_1(x)| +\sup_{\mathcal B_y} \Big\lvert\lvert u_{1}\rvert^{2^{*}-2}u_{1}-U^{2^*-1}_{1}+\beta \lvert u_{2}\rvert^{\frac{2^{*}}{2}}\lvert u_{1}\rvert^{\frac{2^{*}}{2}-2}u_{1} \Big\rvert\Big).
\end{split}
\end{equation*}

For \(x\in \mathcal B_y\) we have \(|x-\tilde\xi_1|\ge c r_\epsilon\). Moreover, either \(|x-\tilde\xi_2|\le \frac{1}{2}|\tilde\xi_1-\tilde\xi_2|\), in which case \(|x-\tilde\xi_1|\ge c\epsilon^{-1/(N-4)}\), or \(|x-\tilde\xi_2|>\frac{1}{2}|\tilde\xi_1-\tilde\xi_2|\), in which case \(|x-\tilde\xi_2|\ge c\epsilon^{-1/(N-4)}\). Using Lemma~\ref{lemB1} and the exponent inequalities in \eqref{eq:auxiliary-exponents}, this gives
\begin{equation*}
\begin{split}
&\sup_{\mathcal B_y} \Big\lvert\beta \lvert u_{2}\rvert^{\frac{2^{*}}{2}}\lvert u_{1}\rvert^{\frac{2^{*}}{2}-2}u_{1} \Big\rvert\\
&\qquad\leq \sup_{\mathcal B_y}C \frac{1}{(1+\lvert x-\tilde\xi_1\rvert)^{2}}
\frac{1}{(1+\lvert x-\tilde\xi_2\rvert)^{N}} \\
&\qquad\leq \sup_{\mathcal B_y}C \left(\frac{\epsilon^{\alpha_1}}{(1+\lvert x-\tilde\xi_1\rvert)^{\theta_1}}
+\frac{\epsilon^{\tilde\alpha_1}}{(1+\lvert x-\tilde\xi_2\rvert)^{\tilde\theta_1}}\right).
\end{split}
\end{equation*}
Hence we obtain
 \begin{equation*}
\begin{split}
&\sup_{\mathcal B_y} \Big\lvert \lvert u_{1}\rvert^{2^{*}-2}u_{1}-U^{2^*-1}_{1}+\beta \lvert u_{2}\rvert^{\frac{2^{*}}{2}}\lvert u_{1}\rvert^{\frac{2^{*}}{2}-2}u_{1} \Big\rvert\\
\leq&\sup_{\mathcal B_y} \Big\lvert \lvert u_{1}\rvert^{2^{*}-2}u_{1}- U^{2^*-1}_{1}\Big\rvert+\sup_{\mathcal B_y} \Big\lvert\beta \lvert u_{2}\rvert^{\frac{2^{*}}{2}}\lvert u_{1}\rvert^{\frac{2^{*}}{2}-2}u_{1} \Big\rvert\\
\leq &\sup_{\mathcal B_y}C \Big(\lvert u_{1}\rvert^{2^{*}-2}+U^{2^{*}-2}_{1}\Big)\lvert u_{1}(x)-U_{1}(x)\rvert+\sup_{\mathcal B_y}\Big(\frac{C\epsilon^{\alpha_1}}{(1+\lvert x-\tilde\xi_1\rvert)^{ \theta_1}}
+\frac{C\epsilon^{\tilde\alpha_1}}{(1+\lvert x-\tilde\xi_2\rvert)^{\tilde\theta_1}}\Big).
\end{split}
\end{equation*}
Therefore,
\[
|\nabla(u_1(y)-V_1(y))|
\le \frac{C\epsilon^{\alpha_1}}{(1+|y-\tilde\xi_1|)^{\theta_1}}
+\frac{C\epsilon^{\tilde\alpha_1}}{(1+|y-\tilde\xi_2|)^{\tilde\theta_1}}
\quad\text{for }y\in\Omega_\epsilon\setminus
B_{r_\epsilon+1}(\tilde\xi_1).
\]

The same argument applies to $|\nabla(u_2-V_2)|$.

\end{proof}

\medskip

\subsection{Dead-core formation}
The solutions of \eqref{eqout} obtained in Corollary~\ref{correal} exhibit
dead cores. We use the radial dead-core comparison principle stated in
Lemma~\ref{lemdead}. Its proof is deferred to
Appendix~A. The argument below is written for the
absolute values and therefore does not rely on the nonnegativity of the final
solution.
\begin{Lem}\label{lemdead}
For every $b>0$, there exists $c_b>0$ such that
\[
\begin{cases}
-\Delta w+bw^{\frac{2^*}{2}-1}=0&\text{in }B_{3/2}(0),\\
w=c_b&\text{on }\partial B_{3/2}(0)
\end{cases}
\]
has a unique nonnegative weak solution in the class
\[
w\in H^1(B_{3/2}(0)),\qquad
w-c_b\in H_0^1(B_{3/2}(0)),\qquad w\ge0\ \text{a.e.}
\]
This solution is radial, belongs to $C^1(B_{3/2}(0))$, vanishes identically in
$B_1(0)$, and satisfies $w'(r)>0$ whenever $w(r)>0$.
\end{Lem}

\begin{Lem}\label{deadout}
There exists $\varrho_0>0$ such that, for every fixed
$\varrho\in(0,\varrho_0)$, there exists
$\epsilon_0=\epsilon_0(\varrho)>0$ such that, for every
$\epsilon\in(0,\epsilon_0)$, the solution \((u_1,u_2)\) of
\eqref{eqout} obtained in Corollary~\ref{correal} satisfies
\[
u_{1}=0\quad\text{in }B_{\epsilon^{-\frac{1-\varrho}{N-2}}}(\tilde\xi_2),
\qquad
u_{2}=0\quad\text{in }B_{\epsilon^{-\frac{1-\varrho}{N-2}}}(\tilde\xi_1).
\]
\end{Lem}
\begin{proof}
The fixed losses can be chosen so that, for $i=1,2$,
\begin{equation}\label{eq:deadcore-exponent-margins}
\alpha_i-\frac{N-2-\theta_i}{N-2}>0,
\qquad
\tilde\alpha_i+\frac{\tilde\theta_i}{N-4}>1.
\end{equation}
Indeed, when the losses vanish, the two margins are respectively
\[
\frac{2(N-3)}{(N-4)(N-2)}>0,
\qquad
\frac4{N-4}>0.
\]
Choose $\varrho_0\in(0,1)$ so that, with
$\alpha_\varrho=(1-\varrho)/(N-2)$,
\[
\alpha_2-\alpha_\varrho(N-2-\theta_2)>0,
\qquad
\tilde\alpha_2+\frac{\tilde\theta_2}{N-4}
-\alpha_\varrho(N-2)>0
\]
for every $\varrho\in(0,\varrho_0)$. Fix such a $\varrho$.
We prove the assertion for \(u_1\); the proof for \(u_2\) is identical.
All estimates below are taken in
$B_{\frac{3}{2}\epsilon^{-\alpha_\varrho}}(\tilde\xi_2)$. In this ball,
$|y-\tilde\xi_1|\ge c\epsilon^{-1/(N-4)}$. Hence the bubble-tail estimate gives
\[
U_1(y)+V_1(y)\le C\epsilon^{\frac{N-2}{N-4}},
\]
and, more importantly, Lemma~\ref{lem7} and the separation of the centers give
\begin{equation}\label{eq:deadcore-opposite-upper}
|u_1(y)|\le C\epsilon^{\frac{N-2}{N-4}}.
\end{equation}
Moreover,
\[
U_2(y)\ge c\epsilon^{\alpha_\varrho(N-2)}.
\]
On $Q_1$, the trace estimate and
$\gamma_0>(N-2)/B_\tau$ give
$|u_2-V_2|=o(U_2)$. On the remainder of the ball, Lemma~\ref{lem8} and the
choice of $\varrho_0$ give the same conclusion, because the two ratios are
bounded respectively by
\[
C\epsilon^{\alpha_2-\alpha_\varrho(N-2-\theta_2)}
\quad\text{and}\quad
C\epsilon^{\tilde\alpha_2+\frac{\tilde\theta_2}{N-4}
-\alpha_\varrho(N-2)}.
\]
The projection expansion also gives $V_2=U_2(1+o(1))$ uniformly in this
ball. Therefore, using \eqref{eq:deadcore-opposite-upper},
\[
|u_2(y)|\ge c\epsilon^{\alpha_\varrho(N-2)},
\qquad
|u_1(y)|\le C\epsilon^{\frac{N-2}{N-4}}
=o\big(\epsilon^{\alpha_\varrho(N-2)}\big).
\]
Recall that $p=N/(N-2)$. Kato's inequality gives
\[
-\Delta|u_1|\le
\epsilon^{\frac{N-2}{N-4}}|u_1|+|u_1|^{2^*-1}
-c|u_2|^p|u_1|^{p-1}.
\]
The self-interaction is absorbed because
$|u_1|^p=o(|u_2|^p)$. The mass term is also lower order. Indeed,
\[
\frac{\epsilon^{\frac{N-2}{N-4}}|u_1|}{|u_2|^p|u_1|^{p-1}}
\le C\epsilon^{\gamma_1(3-p)-p(1-\varrho)}=o(1),
\]
since
\[
\gamma_1(3-p)-p
=\frac{N^2-6N+12}{(N-4)(N-2)}>0.
\]
Thus there exists a constant $b_0>0$, independent of $\epsilon$, such
that
\begin{equation}\label{deadcore-diff-ineq}
-\Delta |u_1|+b_0\epsilon^{\alpha_\varrho N}|u_1|^{p-1}\le0
\quad\text{in }B_{\frac{3}{2}\epsilon^{-\alpha_\varrho}}(\tilde\xi_2).
\end{equation}
Let $w$ be the radial solution in Lemma~\ref{lemdead} with $b=b_0$, and define
\[
w_\epsilon(x)=\epsilon^{\frac{(N-2)^2}{N-4}\alpha_\varrho}
 w(\epsilon^{\alpha_\varrho}x).
\]
A direct scaling gives
\[
-\Delta w_\epsilon+b_0\epsilon^{N\alpha_\varrho}w_\epsilon^{\frac{2^*}{2}-1}=0
\quad\text{in }B_{\frac{3}{2}\epsilon^{-\alpha_\varrho}}(0).
\]
On the outer boundary of the comparison ball,
\[
w_\epsilon(\cdot-\tilde\xi_2)
=c_{b_0}\epsilon^{\frac{(N-2)^2}{N-4}\alpha_\varrho}.
\]
Since
\[
\frac{N-2}{N-4}
-\frac{(N-2)^2}{N-4}\alpha_\varrho
=\frac{N-2}{N-4}\varrho>0,
\]
the estimate \(|u_1|\le C\epsilon^{(N-2)/(N-4)}\) implies, for the fixed
$\varrho$ and all sufficiently small $\epsilon$,
\[
|u_1|\le w_\epsilon(\cdot-\tilde\xi_2)
\quad\text{on }\partial B_{\frac{3}{2}\epsilon^{-\alpha_\varrho}}(\tilde\xi_2).
\]
Testing the positive part of
$|u_1|-w_\epsilon(\cdot-\tilde\xi_2)$ and using the monotonicity of
$t\mapsto t^{p-1}$ gives
\[
|u_1|\le w_\epsilon(\cdot-\tilde\xi_2)
\quad\text{in }B_{\frac{3}{2}\epsilon^{-\alpha_\varrho}}(\tilde\xi_2).
\]
Since \(w=0\) in \(B_1(0)\), the scaled function vanishes in
\(B_{\epsilon^{-\alpha_\varrho}}(\tilde\xi_2)\). This proves
\(u_1=0\) in \(B_{\epsilon^{-\alpha_\varrho}}(\tilde\xi_2)\), which is the
stated radius. The argument for \(u_2\) is identical.
\end{proof}
\begin{Rem}\label{remdead}
Let $\varrho_0$ be as in Lemma~\ref{deadout}. Fix
$\tau\in\mathcal I_N$ and choose
$\varrho\in(0,\min\{\varrho_\tau,\varrho_0\})$. Then
\eqref{deadcore-fit}, Lemma~\ref{deadout}, and
\eqref{scale-hierarchy} give, for all sufficiently small $\epsilon$,
\[
P_2\subset
B_{R_{\mathrm{dc},\epsilon}(\varrho)}(\tilde\xi_1),
\qquad
Q_2\subset
B_{R_{\mathrm{dc},\epsilon}(\varrho)}(\tilde\xi_2),
\]
and hence
\[
u_2=0\quad\text{in }P_2,
\qquad
u_1=0\quad\text{in }Q_2.
\]
Moreover,
$r_\epsilon^2+1=o(R_{\mathrm{dc},\epsilon}(\varrho))$, so the unit
neighborhoods of $P_2$ and $Q_2$ that occur in the local gradient estimates
also remain inside the corresponding dead-core balls. This is the overlap
property needed in the reduction of Section~4.
\end{Rem}

\medskip

Let $(\bar u_{1,0},\bar u_{2,0})$ be another pair of energy-compatible inner data satisfying $(\bar u_{1,0},\bar u_{2,0})-(V_1,V_2)\in \Lambda_\epsilon\cap H_0^1(\Omega_\epsilon)^2$, and let $(\bar u_{1},\bar u_{2})$ be the corresponding minimizer of $I$ on $\mathcal M(\bar u_{1,0},\bar u_{2,0})$. Set
\[
(w,v) := (u_{1},u_{2})-(\bar u_{1},\bar u_{2}),\qquad
(w_{0},v_{0}) := (u_{1,0},u_{2,0})-(\bar u_{1,0},\bar u_{2,0}).
\]
We next estimate $(w,v)$.
Fix once and for all
\[
\max\left\{0,\frac{6-N}{2}\right\}<\theta_{\mathrm{st}}<2.
\]
\begin{Lem}\label{lem9}
For all sufficiently small $\epsilon>0$, there exists a constant $C>0$,
independent of $\epsilon$ and of the admissible inner data, such that
\begin{align*}
&\lVert w\rVert_{L^{\infty}(\Omega_\epsilon)}
+
\lVert v\rVert_{L^{\infty}(\Omega_\epsilon)}
\leq
C\big(\lVert w_{0}\rVert_{L^{\infty}(P_{1})}
+
\lVert v_{0}\rVert_{L^{\infty}(Q_{1})}\big),\\
&\lVert \nabla w\rVert_{L^{\infty}(P_2 \setminus B_{r_\epsilon+1}(\tilde\xi_1))}
+
\lVert \nabla v\rVert_{L^{\infty}(Q_2\setminus B_{r_\epsilon+1}(\tilde\xi_2))}
\leq
C\big(\lVert w_{0}\rVert_{L^{\infty}(P_{1})}
+
\lVert v_{0}\rVert_{L^{\infty}(Q_{1})}\big).
\end{align*}
\end{Lem}

\begin{proof}
By \eqref{eqout}, we have
\begin{equation}\label{eq19}
\left\{
\begin{aligned}
-\Delta w={}&\epsilon^{\frac{N-2}{N-4}}w
+\big(\lvert u_{1}\rvert^{2^{*}-2}u_{1}
      -\lvert \bar u_{1}\rvert^{2^{*}-2}\bar u_{1}\big) \\
&+\beta\big(\lvert u_{2}\rvert^{\frac{2^{*}}{2}}
       \lvert u_{1}\rvert^{\frac{2^{*}}{2}-2}u_{1}
      -\lvert \bar u_{2}\rvert^{\frac{2^{*}}{2}}
       \lvert \bar u_{1}\rvert^{\frac{2^{*}}{2}-2}\bar u_{1}\big),
&&\text{in }\Omega_\epsilon\setminus P_{1},\\[1mm]
-\Delta v={}&\epsilon^{\frac{N-2}{N-4}}v
+\big(\lvert u_{2}\rvert^{2^{*}-2}u_{2}
      -\lvert \bar u_{2}\rvert^{2^{*}-2}\bar u_{2}\big) \\
&+\beta\big(\lvert u_{1}\rvert^{\frac{2^{*}}{2}}
       \lvert u_{2}\rvert^{\frac{2^{*}}{2}-2}u_{2}
      -\lvert \bar u_{1}\rvert^{\frac{2^{*}}{2}}
       \lvert \bar u_{2}\rvert^{\frac{2^{*}}{2}-2}\bar u_{2}\big),
&&\text{in }\Omega_\epsilon\setminus Q_{1},\\[1mm]
w={}&w_{0}\quad\text{in }P_{1},
&v={}&v_{0}\quad\text{in }Q_{1}.
\end{aligned}
\right.
\end{equation}

Recall that $p=2^*/2$ and set $\mathcal N_p(t):=|t|^{p-2}t$. For
$y\in\Omega_\epsilon\setminus P_1$,
\begin{equation}\label{eq:self-difference}
\big||u_1|^{2^*-2}u_1-|\bar u_1|^{2^*-2}\bar u_1\big|
\le C\big(|u_1|^{2^*-2}+|\bar u_1|^{2^*-2}\big)|w|.
\end{equation}
For the coupling term we use the identity
\[
|u_2|^p\mathcal N_p(u_1)-|\bar u_2|^p\mathcal N_p(\bar u_1)
=|u_2|^p\big(\mathcal N_p(u_1)-\mathcal N_p(\bar u_1)\big)
+\big(|u_2|^p-|\bar u_2|^p\big)\mathcal N_p(\bar u_1).
\]
Since $\mathcal N_p$ is increasing,
\[
\beta |u_2|^p\big(\mathcal N_p(u_1)-\mathcal N_p(\bar u_1)\big)\operatorname{sign}w\le0.
\]
For the remaining term,
\[
\big||u_2|^p-|\bar u_2|^p\big|
\le C\big(|u_2|^{p-1}+|\bar u_2|^{p-1}\big)|v|.
\]
Using Lemma~\ref{lem7}, we obtain
\begin{equation}\label{eq:coupling-difference}
\left|\big(|u_2|^p-|\bar u_2|^p\big)\mathcal N_p(\bar u_1)\right|
\le \frac{C|v|}{(1+|y-\tilde\xi_1|)^2(1+|y-\tilde\xi_2|)^2}.
\end{equation}
This argument uses only monotonicity and the elementary power estimate above;
it does not differentiate $h$ at zero.

Combining \eqref{eq:self-difference} and \eqref{eq:coupling-difference}, the first equation in \eqref{eq19} shows that $|w|$ satisfies
\begin{align*}
-\Delta |w|
&\le c\bigl(\epsilon^{\frac{N-2}{N-4}}+|u_1|^{2^*-2}
+|\bar u_1|^{2^*-2}\bigr)|w|
+\frac{c|v|}{(1+|y-\tilde\xi_1|)^2(1+|y-\tilde\xi_2|)^2}\\
&=: \Psi_1|w|+f_1|v|
\quad\text{in }\Omega_\epsilon\setminus P_1.
\end{align*}
Similarly,
\begin{align*}
-\Delta |v|
&\le c\bigl(\epsilon^{\frac{N-2}{N-4}}+|u_2|^{2^*-2}
+|\bar u_2|^{2^*-2}\bigr)|v|
+\frac{c|w|}{(1+|y-\tilde\xi_1|)^2(1+|y-\tilde\xi_2|)^2}\\
&=: \Psi_2|v|+f_2|w|
\quad\text{in }\Omega_\epsilon\setminus Q_1.
\end{align*}

Recall
 \[
\rho_i(y)=1+|y-\tilde\xi_i|,\qquad i=1,2.
\]
Choose a fixed constant $a_{\mathrm b}>0$ sufficiently small and put
$b_1(y)=1+a_{\mathrm b}|y-\tilde\xi_1|$. Define
\[
C_1=r_\epsilon^{N-2}\|w_0\|_{L^\infty(P_1)},\qquad
C_2=\epsilon^{\frac{\theta_{\mathrm{st}}}{(2+\tau)(N-4)}-\frac{2-\theta_{\mathrm{st}}}{N-4}}
\|v\|_{L^\infty(\Omega_\epsilon)},
\]
and
\[
\widetilde w=C_1b_1^{-(N-2)}
+C_2\rho_1^{-(2-\theta_{\mathrm{st}})}\rho_2^{-(2-\theta_{\mathrm{st}})}.
\]
A direct computation gives
 \[
-\Delta b_1^{-(N-2)}
=\frac{(N-2)(N-1)a_{\mathrm b}}{|y-\tilde\xi_1|b_1^N}.
\]
The mixed term has the same positive lower bound as in Lemma~\ref{lem8}:
\[
\begin{split}
-\Delta\big(\rho_1^{-(2-\theta_{\mathrm{st}})}\rho_2^{-(2-\theta_{\mathrm{st}})}\big)
\ge{}&(2-\theta_{\mathrm{st}})(N-6+2\theta_{\mathrm{st}})
\rho_1^{-(2-\theta_{\mathrm{st}})}\rho_2^{-(2-\theta_{\mathrm{st}})}\\
&\times\big(\rho_1^{-2}+\rho_2^{-2}\big).
\end{split}
\]
We record the only dimension-sensitive point. Since
$|y-\tilde\xi_1|\le R_\Omega\epsilon^{-1/(N-4)}$ in $\Omega_\epsilon$, the relative
coefficient of the first term is bounded below by
\[
\frac{(N-2)(N-1)a_{\mathrm b}}{|y-\tilde\xi_1|(1+a_{\mathrm b}|y-\tilde\xi_1|)^2}
\ge \frac{c}{a_{\mathrm b}}\epsilon^{\frac{3}{N-4}}
\]
for all sufficiently small $\epsilon$, where $c>0$ is independent of $a_{\mathrm b}$ and
$\epsilon$. For $N>5$ this dominates $\epsilon^{(N-2)/(N-4)}$ by the powers.
For $N=5$ the powers are equal, and we first choose the fixed constant $a_{\mathrm b}>0$
sufficiently small so that $c/a_{\mathrm b}$ dominates the coefficient of the mass term
in $\Psi_1$. The remaining terms in
$\Psi_1$ are absorbed by the decay from Lemma~\ref{lem7}. Thus the first
barrier is valid uniformly for every $N\ge5$.

As in Lemma~\ref{lem8}, we show that there is a constant $C>0$ such that
\begin{align}\label{eq:stability-supersolution}
-\Delta(C\widetilde w)&\geq \Psi_1|C\widetilde w|+f_1|v|
\quad\text{in }\Omega_\epsilon\setminus P_1.
\end{align}

First, we show
\begin{align}\label{eq:stability-mixed-source}
-\Delta(C\widetilde w)&\geq 2f_1|v|
\quad\text{in }\Omega_\epsilon\setminus P_1.
\end{align}
It suffices to show that there exists $C>0$ such that
\[
\frac{CC_2}{(1+|y-\tilde\xi_1|)^{2-\theta_{\mathrm{st}}}
(1+|y-\tilde\xi_2|)^{4-\theta_{\mathrm{st}}}}
\ge
\frac{|v|}{(1+|y-\tilde\xi_1|)^2(1+|y-\tilde\xi_2|)^2}
\quad\text{in }\Omega_\epsilon\setminus P_1.
\]
Equivalently, it is enough to prove
\[
\frac{(1+|y-\tilde\xi_2|)^{2-\theta_{\mathrm{st}}}}
{(1+|y-\tilde\xi_1|)^{\theta_{\mathrm{st}}}}|v|
\le CC_2
\quad\text{in }\Omega_\epsilon\setminus P_1,
\]
which follows from the definition of \(C_2\), from
\(\rho_1\ge c r_\epsilon\) in \(\Omega_\epsilon\setminus P_1\), from
\(\rho_2\le C\epsilon^{-1/(N-4)}\), and from
\(r_\epsilon=\epsilon^{-1/((2+\tau)(N-4))}\).
Hence \eqref{eq:stability-mixed-source} holds.

Next, we show
\begin{align}\label{eq:stability-potential}
-\Delta\widetilde w&\geq\Psi_1|\widetilde w|
\quad\text{in }\Omega_\epsilon\setminus P_1.
\end{align}
It is enough to verify the following two inequalities:
 \begin{align*}
\frac{C_1(N-2)(N-1)a_{\mathrm b}}{|y-\tilde\xi_1|b_1^{N}}\geq c\big(\epsilon^{\frac{N-2}{N-4}}+|u_{1}|^{2^{*}-2}+|\bar u_{1}|^{2^{*}-2}\big)C_1b_1^{-(N-2)}
\end{align*}
and
\begin{align*}
&(2-\theta_{\mathrm{st}})(N-6+2\theta_{\mathrm{st}})C_2
\rho_1^{-(2-\theta_{\mathrm{st}})}\rho_2^{-(2-\theta_{\mathrm{st}})}
\big(\rho_1^{-2}+\rho_2^{-2}\big)\\
&\qquad\ge
c\bigl(\epsilon^{\frac{N-2}{N-4}}+|u_1|^{2^*-2}
+|\bar u_1|^{2^*-2}\bigr)
C_2\rho_1^{-(2-\theta_{\mathrm{st}})}\rho_2^{-(2-\theta_{\mathrm{st}})}.
\end{align*}

Thus, \eqref{eq:stability-mixed-source} and \eqref{eq:stability-potential} imply \eqref{eq:stability-supersolution}.

By the definition of $C_1$ and the boundary condition $w=w_0$ on $P_1$,
there exists $C>0$ such that $|w|\le C\widetilde w$ on $\partial P_1$.
On the exterior domain
$\mathcal O_{1,\epsilon}=\Omega_\epsilon\setminus\overline{P_1}$,
Lemma~\ref{lem7} gives
\[
\begin{aligned}
\|\Psi_1\|_{L^{N/2}(\mathcal O_{1,\epsilon})}
&\le C\epsilon^{\frac{N-2}{N-4}}|\Omega_\epsilon|^{2/N}
 +C\|u_1\|_{L^{2^*}(\mathcal O_{1,\epsilon})}^{2^*-2}
 +C\|\bar u_1\|_{L^{2^*}(\mathcal O_{1,\epsilon})}^{2^*-2}\\
&=O(\epsilon)+o(1)<\mathbf C_s
\end{aligned}
\]
for sufficiently small $\epsilon$. Here the $o(1)$ is uniform because
$r_\epsilon\to\infty$ and the decay in Lemma~\ref{lem7} is uniform for both
outer minimizers. In addition, $w=0$ on $\partial\Omega_\epsilon$ and
$(|w|-C\widetilde w)_+\in D_0^{1,2}(\mathcal O_{1,\epsilon})$. Lemma~\ref{lem0} therefore applies
and yields
\begin{equation}\label{eq22}
\begin{aligned}
|w(y)|
&\le \frac{Cr_\epsilon^{N-2}\|w_0\|_{L^\infty(P_1)}}
{(1+a_{\mathrm b}|y-\tilde\xi_1|)^{N-2}}\\
&\quad+
\frac{C\epsilon^{\frac{\theta_{\mathrm{st}}}{(2+\tau)(N-4)}-\frac{2-\theta_{\mathrm{st}}}{N-4}}
\|v\|_{L^\infty(\Omega_\epsilon)}}
{(1+|y-\tilde\xi_1|)^{2-\theta_{\mathrm{st}}}(1+|y-\tilde\xi_2|)^{2-\theta_{\mathrm{st}}}}
&&\text{in }\Omega_\epsilon\setminus P_1,\\
|v(y)|
&\le \frac{Cr_\epsilon^{N-2}\|v_0\|_{L^\infty(Q_1)}}
{(1+a_{\mathrm b}|y-\tilde\xi_2|)^{N-2}}\\
&\quad+
\frac{C\epsilon^{\frac{\theta_{\mathrm{st}}}{(2+\tau)(N-4)}-\frac{2-\theta_{\mathrm{st}}}{N-4}}
\|w\|_{L^\infty(\Omega_\epsilon)}}
{(1+|y-\tilde\xi_1|)^{2-\theta_{\mathrm{st}}}(1+|y-\tilde\xi_2|)^{2-\theta_{\mathrm{st}}}}
&&\text{in }\Omega_\epsilon\setminus Q_1.
\end{aligned}
\end{equation}
By dividing the domain $\Omega_\epsilon$ into the regions
\[
\{ x \in \Omega_\epsilon : |x - \tilde\xi_2| < \tfrac{1}{2}|\tilde\xi_1 - \tilde\xi_2| \} \quad \text{and} \quad \{ x \in \Omega_\epsilon : |x - \tilde\xi_2| \geq \tfrac{1}{2}|\tilde\xi_1 - \tilde\xi_2| \},
\]
we obtain the following estimate.
\begin{equation*}
\begin{split}
&\displaystyle \lVert w\rVert_{L^{\infty}(\Omega_\epsilon\setminus P_{1})}\leq C(\lVert w_{0}\rVert_{L^{\infty}(P_{1})}+\epsilon^{\frac{\theta_{\mathrm{st}}}{(2+\tau)(N-4)}}\lVert v_{0}\rVert_{L^{\infty}(Q_{1})}+\epsilon^{\frac{\theta_{\mathrm{st}}}{(2+\tau)(N-4)}}\lVert v\rVert_{L^{\infty}(\Omega_\epsilon \setminus Q_{1})}),\\
&\displaystyle \lVert v\rVert_{L^{\infty}(\Omega_\epsilon\setminus Q_{1})}\leq C(\epsilon^{\frac{\theta_{\mathrm{st}}}{(2+\tau)(N-4)}}\lVert w_{0}\rVert_{L^{\infty}(P_{1})}+\epsilon^{\frac{\theta_{\mathrm{st}}}{(2+\tau)(N-4)}}\lVert w\rVert_{L^{\infty}(\Omega_\epsilon \setminus P_{1})}+\lVert v_{0}\rVert_{L^{\infty}(Q_{1})}).
\end{split}
\end{equation*}

Finally,
\begin{equation}\label{eq:outer-stability-sup}
\begin{split}
\lVert w\rVert_{L^{\infty}(\Omega_\epsilon\setminus P_{1})}+\lVert v\rVert_{L^{\infty}(\Omega_\epsilon\setminus Q_{1})} \leq C(\lVert w_{0}\rVert_{L^{\infty}(P_{1})}+\lVert v_{0}\rVert_{L^{\infty}(Q_{1})}),
\end{split}
\end{equation}
concluding the proof for $w$ and $v$.

\medskip
For the first equation in \eqref{eq19}, if $y\in P_{2}\setminus B_{r_{\epsilon}+1}(\tilde\xi_1)$, then \cite[Theorem 3.9]{GT01} gives
\begin{equation*}
\begin{aligned}
\lvert \nabla w(y)\rvert\leq& C\Big(\sup_{B_{1}(y)}\lvert w(x)\rvert+\sup_{B_{1}(y)}\epsilon^{\frac{N-2}{N-4}}\lvert w(x)\rvert+\sup_{B_{1}(y)} \Big\lvert \big(\lvert u_{1}\rvert^{2^{*}-2}u_{1}-\lvert \bar u_{1}\rvert^{2^{*}-2}\bar u_{1}\big)\\
&+\beta\big(\lvert u_{2}\rvert^{\frac{2^{*}}{2}}\lvert u_{1}\rvert^{\frac{2^{*}}{2}-2}u_{1}-\lvert \bar u_{2}\rvert^{\frac{2^{*}}{2}}\lvert \bar u_{1}\rvert^{\frac{2^{*}}{2}-2}\bar u_{1}\big) \Big\rvert\Big).
\end{aligned}
\end{equation*}
By Lemma~\ref{deadout} and Remark~\ref{remdead}, the choice of
\(\tau\in\mathcal I_N\) gives
\(P_2\subset B_{\epsilon^{-\frac{1-\varrho}{N-2}}}(\tilde\xi_1)\) for a
suitable small \(\varrho>0\). Hence
\(u_2=\bar u_2=0\) in \(P_2\).

Because the inequality in Remark~\ref{remdead} is strict,
\[
r_\epsilon^2+1=o\!\left(
\epsilon^{-\frac{1-\varrho}{N-2}}
\right).
\]
Thus, for every $y\in P_2$ and all sufficiently small $\epsilon$,
\[
B_1(y)\subset B_{r_\epsilon^2+1}(\tilde\xi_1)
\subset B_{\epsilon^{-\frac{1-\varrho}{N-2}}}(\tilde\xi_1).
\]
Consequently $u_2=\bar u_2=0$ throughout the whole ball $B_1(y)$, not merely
at its center. The analogous inclusion holds for every $z\in Q_2$, with the
roles of the two components interchanged.

As in \eqref{eq:self-difference}, we obtain
\begin{equation*}
\begin{aligned}
&\sup_{B_{1}(y)} \Big\lvert  \big(\lvert u_{1}\rvert^{2^{*}-2}u_{1}-\lvert \bar u_{1}\rvert^{2^{*}-2}\bar u_{1}\big) \Big\rvert
\leq &C\sup_{B_{1}(y)}\big(\lvert u_{1}\rvert^{2^{*}-2}+\lvert\bar u_{1}\rvert^{2^{*}-2}\big)\lvert w\rvert.
\end{aligned}
\end{equation*}
Thus,
\begin{equation*}
\begin{aligned}
\lvert \nabla w(y)\rvert\leq C\sup_{B_{1}(y)}\lvert w(x)\rvert.
\end{aligned}
\end{equation*}

The second equation of \eqref{eq19} is handled similarly in $Q_{2}\setminus B_{r_{\epsilon}+1}(\tilde\xi_2)$ and gives
\begin{equation*}
\lvert\nabla v(z)\rvert\le C\Big(
\sup_{B_1(z)}|v(s)|+\epsilon^{\frac{N-2}{N-4}}\sup_{B_1(z)}|v(s)|
+\sup_{B_1(z)}\big||u_2|^{2^*-2}u_2-|\bar u_2|^{2^*-2}\bar u_2\big|\Big).
\end{equation*}
Here, again by Lemma~\ref{deadout} and Remark~\ref{remdead},
\(Q_2\subset B_{\epsilon^{-\frac{1-\varrho}{N-2}}}(\tilde\xi_2)\) and hence
\(u_1=\bar u_1=0\) in \(Q_2\).
Hence
 \begin{equation*}
\begin{aligned}
\lvert \nabla v(z)\rvert\leq C\sup_{B_{1}(z)}\lvert v(s)\rvert.
\end{aligned}
\end{equation*}
The conclusion follows from the preceding estimate for
$\|w\|_{L^{\infty}(\Omega_\epsilon)}+\|v\|_{L^{\infty}(\Omega_\epsilon)}$.
\end{proof}

\medskip

The following uniqueness statement is an immediate consequence of the
preceding stability estimate and will be used repeatedly.

\begin{Lem}\label{lem:outer-unique}
For all sufficiently small $\epsilon>0$ and every energy-compatible admissible pair of inner data \((u_{1,0},u_{2,0})\) with
\((u_{1,0},u_{2,0})-(V_1,V_2)\in\Lambda_\epsilon\cap H_0^1(\Omega_\epsilon)^2\), the minimization problem
\eqref{min} has at most one minimizer.
\end{Lem}
\begin{proof}
Let \((u_1,u_2)\) and \((\bar u_1,\bar u_2)\) be two minimizers associated with the same prescribed inner data. In the notation of Lemma~\ref{lem9}, one has \(w_0=0\) in \(P_1\) and \(v_0=0\) in \(Q_1\). Applying Lemma~\ref{lem9} gives
\[
\|u_1-\bar u_1\|_{L^\infty(\Omega_\epsilon)}+
\|u_2-\bar u_2\|_{L^\infty(\Omega_\epsilon)}=0.
\]
Thus the minimizer is unique.
\end{proof}

We now define the exterior solution operator directly on the admissible inner
data, without introducing a separate restriction operator. The subset
\[
\Lambda_\epsilon\cap H_0^1(\Omega_\epsilon)^2
\]
is dense in $\Lambda_\epsilon$ for the uniform norm. To see this, let
$g=(g_1,g_2)\in\Lambda_\epsilon$ and let $m\ge2$ be an integer. First set
\[
\widetilde g^{(m)}:=(1-m^{-1})g.
\]
This scaling creates a margin of size $\epsilon^{\gamma_0}/m$ below the
boundary of the uniform ball. Approximate $\widetilde g^{(m)}$ uniformly by
a pair in $C_c^\infty(\Omega_\epsilon)^2$, and then subtract linear
combinations of $Y_{i,0},\ldots,Y_{i,N}$ to restore the moment conditions.
The coefficient matrices are
\[
\left(\int_{\Omega_\epsilon}
U_i^{2^*-2}Y_{i,j}Y_{i,\ell}\right)_{0\le j,\ell\le N},
\qquad i=1,2,
\]
and are positive definite. Hence the correction coefficients tend to zero
with the uniform approximation error. Choosing that error sufficiently small
compared with $\epsilon^{\gamma_0}/m$, the corrected pair remains in
$\Lambda_\epsilon$. This gives a sequence in
$\Lambda_\epsilon\cap H_0^1(\Omega_\epsilon)^2$ converging uniformly to $g$.
The preliminary factor $1-m^{-1}$ is used only to leave room for the
approximation and the finite-dimensional moment correction.

For later use, set
\[
\mathcal A_{i,\epsilon}:=
\left\{y\in\Omega_\epsilon:
\frac14r_\epsilon^2\le |y-\tilde\xi_i|
\le\frac34r_\epsilon^2\right\},
\qquad i=1,2.
\]

\begin{Def}\label{defS}
For $g=(g_1,g_2)\in
\Lambda_\epsilon\cap H_0^1(\Omega_\epsilon)^2$, denote by
$(u_1^g,u_2^g)$ the unique minimizer of $I$ in
$\mathcal M(V_1+g_1,V_2+g_2)$ and set
\[
S(g):=(u_1^g,u_2^g)-(V_1,V_2).
\]
\end{Def}

If $g,\bar g\in\Lambda_\epsilon\cap H_0^1(\Omega_\epsilon)^2$ have
$g_1=\bar g_1$ in $P_1$ and $g_2=\bar g_2$ in $Q_1$, then their admissible
minimization classes coincide; uniqueness therefore gives $S(g)=S(\bar g)$.
Moreover, Lemma~\ref{lem9} gives a constant $C>0$ such that
\[
\|S(g)-S(\bar g)\|_{L^\infty(\Omega_\epsilon)^2}
\le C\left(
\|g_1-\bar g_1\|_{L^\infty(P_1)}
+\|g_2-\bar g_2\|_{L^\infty(Q_1)}\right)
\]
for all energy-compatible data. The density just proved therefore yields a
unique extension, still denoted by
\[
S:\Lambda_\epsilon\longrightarrow C_0(\Omega_\epsilon)^2.
\]
For all $g,\bar g\in\Lambda_\epsilon$, this extension satisfies
\begin{equation}\label{eq:S-Lipschitz}
\|S(g)-S(\bar g)\|_{L^\infty(\Omega_\epsilon)^2}
\le C\left(
\|g_1-\bar g_1\|_{L^\infty(P_1)}
+\|g_2-\bar g_2\|_{L^\infty(Q_1)}\right).
\end{equation}
Writing $S(g)=(\varphi,\psi)$, one also has
\begin{equation}\label{eq:S-preserves-inner-data}
\varphi=g_1\quad\text{in }P_1,
\qquad
\psi=g_2\quad\text{in }Q_1.
\end{equation}
Indeed, these identities hold for the minimizers and pass to the uniform
limit.

We shall also use the following regularity consequence of the same
approximation. If $g^{(m)},g\in\Lambda_\epsilon$ and
\[
\|g^{(m)}_1-g_1\|_{L^\infty(P_1)}
+\|g^{(m)}_2-g_2\|_{L^\infty(Q_1)}\longrightarrow0,
\]
then $S(g^{(m)})\to S(g)$ uniformly in $\Omega_\epsilon$. The gradient
estimates in Lemma~\ref{lem9} give uniform convergence of the corresponding
gradients on $\mathcal A_{1,\epsilon}$ and $\mathcal A_{2,\epsilon}$.
More generally, the first components converge locally in $C^1$ on
$\Omega_\epsilon\setminus\overline{P_1}$ and the second components converge
locally in $C^1$ on
$\Omega_\epsilon\setminus\overline{Q_1}$, by fixed-$\epsilon$ local
$W^{2,q}$ estimates with $q>N$. The pointwise, dead-core, and stability
estimates proved above pass to the extension by the same approximation.

Set
\begin{equation}\label{eq:X-space}
\mathcal X_\epsilon:=S(\Lambda_\epsilon)
\end{equation}
and equip this space with
\begin{equation}\label{eq:dX}
\begin{aligned}
d_{\mathcal X}\big((\varphi,\psi),(\bar\varphi,\bar\psi)\big)
:=\;&\|\varphi-\bar\varphi\|_{L^\infty(P_1)}\\
&+\|\psi-\bar\psi\|_{L^\infty(Q_1)}.
\end{aligned}
\end{equation}
This is a metric: estimate \eqref{eq:S-Lipschitz} shows that two elements of
$\mathcal X_\epsilon$ with the same prescribed inner data coincide globally.
Moreover, $(\mathcal X_\epsilon,d_{\mathcal X})$ is complete. To see this,
let $(\varphi_m,\psi_m)=S(g^{(m)})$ be a $d_{\mathcal X}$-Cauchy sequence.
By \eqref{eq:S-Lipschitz}, it is Cauchy in the global uniform norm and hence
converges to some $(\varphi,\psi)\in C_0(\Omega_\epsilon)^2$. Its first
components on $P_1$ and second components on $Q_1$ converge uniformly to
continuous functions $h_1$ and $h_2$. Because the moment functions are
supported in $P_1$ and $Q_1$, respectively, these limits satisfy the required
moment conditions and have norm at most $\epsilon^{\gamma_0}$. By the Tietze extension theorem, $h_1$ and $h_2$ extend continuously
to $\overline{\Omega_\epsilon}$, with value zero on
$\partial\Omega_\epsilon$ and without increasing their uniform norms.
The resulting pair $g$ belongs to $\Lambda_\epsilon$. Then \eqref{eq:S-Lipschitz} shows that
$S(g^{(m)})\to S(g)$ uniformly, so $(\varphi,\psi)=S(g)$ belongs to
$\mathcal X_\epsilon$.

\medskip

\section{Localized projected reduction}\label{sec:4}

The nonlinear exterior problem is now encoded by the operator $S$. We close
the construction through a fixed point on the inner data. The localization
takes place on the larger balls $P_2,Q_2$: the dead-core inclusion eliminates
the opposite component there, while cut-off functions at the scale
$r_\epsilon^2$ connect the inner projected corrections to the global exterior
solution. We therefore seek a solution of the following projected problem:
\begin{equation}\label{eqpert}
\begin{cases}
\displaystyle -\Delta \varphi-\epsilon^{\frac{N-2}{N-4}}\varphi-(2^{*}-1)V^{2^{*}-2}_1\varphi=l_{1}+R_{1}(\varphi)+\beta \lvert u_{2}\rvert^{\frac{2^{*}}{2}}\lvert u_{1}\rvert^{\frac{2^{*}}{2}-2}u_{1}+\sum_{j=0}^{N}c_{j}U^{2^{*}-2}_{1}Y_{1,j},\\[3mm]
\displaystyle -\Delta \psi-\epsilon^{\frac{N-2}{N-4}}\psi-(2^{*}-1)V^{2^{*}-2}_{2}\psi=l_{2}+R_{2}(\psi)+\beta \lvert u_{1}\rvert^{\frac{2^{*}}{2}}\lvert u_{2}\rvert^{\frac{2^{*}}{2}-2}u_{2}+\sum_{j=0}^{N}d_{j}U^{2^{*}-2}_{2}Y_{2,j},\\[3mm]
\displaystyle \int_{\Omega_\epsilon}U_1^{2^*-2}Y_{1,j}\varphi=0,
\quad
\int_{\Omega_\epsilon}U_2^{2^*-2}Y_{2,j}\psi=0,
\quad j=0,\ldots,N,
\end{cases}
\end{equation}
where
\begin{equation*}
\begin{split}
(l_{1},l_{2}):=
\displaystyle \Big(\begin{matrix}
\epsilon^{\frac{N-2}{N-4}}V_1+V^{2^{*}-1}_{1}-U^{2^{*}-1}_{1}\\
\displaystyle \epsilon^{\frac{N-2}{N-4}}V_2+V^{2^{*}-1}_{2}- U^{2^{*}-1}_{2}
\end{matrix}\Big)^{\top}
\end{split}
\end{equation*}
and
\begin{equation*}
\begin{split}
(R_{1}(\varphi),R_{2}(\psi)):=
\Big(\begin{matrix}
\displaystyle  \lvert V_1+\varphi \rvert^{2^{*}-2}(V_1+\varphi)-V_1^{2^{*}-1}-(2^{*}-1)V_1^{2^{*}-2} \varphi\\
\displaystyle \lvert V_2+\psi \rvert^{2^{*}-2}(V_2+\psi)-V^{2^{*}-1}_2-(2^{*}-1)V^{2^{*}-2}_2\psi
\end{matrix}\Big)^{\top}.
\end{split}
\end{equation*}

We work on the complete metric space $\mathcal X_\epsilon$ defined in
\eqref{eq:X-space}. Its elements are the exterior corrections, while the
metric \eqref{eq:dX} records only the first component on $P_1$ and the
second component on $Q_1$. This is exactly the inner-data distance appearing
in the stability estimate of Lemma~\ref{lem9}.

\medskip

Choose $\eta\in C^\infty([0,+\infty))$ such that
$0\le\eta\le1$, $\eta=0$ on $[0,1/4]$, and $\eta=1$ on
$[3/4,+\infty)$. For $i=1,2$, set
\[
\chi_i(y):=\eta\left(\frac{|y-\tilde\xi_i|}{r_\epsilon^2}\right).
\]
Then
\begin{equation}\label{jd}
\left\{
\begin{aligned}
&\chi_i=1
&&\text{in }\Omega_\epsilon\setminus
B_{\frac{3}{4}r_\epsilon^2}(\tilde\xi_i),\\
&\chi_i=0
&&\text{in }B_{\frac{1}{4}r_\epsilon^2}(\tilde\xi_i),\\
&|\nabla\chi_i|\le Cr_\epsilon^{-2},\qquad
 |\Delta\chi_i|\le Cr_\epsilon^{-4},
&& i=1,2,
\end{aligned}
\right.
\end{equation}
where $C>0$ is independent of $\epsilon$ and of the admissible
configuration. In particular,
$\operatorname{supp}\nabla\chi_i\cup\operatorname{supp}\Delta\chi_i
\subset\mathcal A_{i,\epsilon}$.

\smallskip

For $(\varphi,\psi)=S(g)$ with
$g\in\Lambda_\epsilon\cap H_0^1(\Omega_\epsilon)^2$, the pair
$(u_1,u_2)=(\varphi,\psi)+(V_1,V_2)$ solves \eqref{eqout}, so the following
identities hold weakly in the full exterior regions. For a general
$(\varphi,\psi)\in\mathcal X_\epsilon$, they hold distributionally on compact
subsets away from the artificial inner boundaries by the local
$C^1$ convergence in the extension construction above. Before the fixed
point is known to come from energy-compatible data, we use them only on the
cut-off annuli. Thus
\begin{align*}
-
\Delta \varphi-
\epsilon^{\frac{N-2}{N-4}}\varphi
-(2^{*}-1)V^{2^{*}-2}_{1}\varphi
&=l_{1}+R_{1}(\varphi) \\
&\quad+
\beta \lvert V_2+\psi\rvert^{\frac{2^{*}}{2}}
\lvert V_1+\varphi\rvert^{\frac{2^{*}}{2}-2}(V_1+\varphi),
&&\text{in }\Omega_\epsilon\setminus P_1,\\[1mm]
-
\Delta \psi-
\epsilon^{\frac{N-2}{N-4}}\psi
-(2^{*}-1)V^{2^{*}-2}_{2}\psi
&=l_{2}+R_{2}(\psi) \\
&\quad+
\beta \lvert V_1+\varphi\rvert^{\frac{2^{*}}{2}}
\lvert V_2+\psi\rvert^{\frac{2^{*}}{2}-2}(V_2+\psi),
&&\text{in }\Omega_\epsilon\setminus Q_1.
\end{align*}
\smallskip

A direct computation gives
\begin{equation}\label{eq:cutoff-first}
\begin{split}
&-\Delta (\chi_{1}\varphi)-\epsilon^{\frac{N-2}{N-4}}\chi_1\varphi-(2^{*}-1)V^{2^{*}-2}_{1} \chi_{1}\varphi\\
=&\displaystyle \chi_{1}\big(l_{1}+R_{1}( \varphi)+\beta \lvert  V_2+\psi\rvert^{\frac{2^{*}}{2}}\lvert V_1+\varphi\rvert^{\frac{2^{*}}{2}-2}(V_1+\varphi)\big)-\Delta\chi_{1} \varphi-2\nabla\chi_{1}\cdot\nabla\varphi
\end{split}
\end{equation}
and
\begin{equation}\label{eq:cutoff-second}
\begin{split}
&\displaystyle -\Delta (\chi_{2}\psi)-\epsilon^{\frac{N-2}{N-4}}\chi_2\psi-(2^{*}-1) V^{2^{*}-2}_{2}\chi_{2}\psi\\
=&\displaystyle \chi_{2}\Big(l_{2}+R_{2}(\psi)+\beta \lvert V_1+\varphi\rvert^{\frac{2^{*}}{2}}\lvert  V_2+\psi\rvert^{\frac{2^{*}}{2}-2}( V_2+\psi)\Big)-\Delta\chi_{2}\psi-2\nabla\chi_{2}\cdot\nabla\psi.
\end{split}
\end{equation}
\smallskip

For a solution of \eqref{eqpert}, subtracting \eqref{eq:cutoff-first} and \eqref{eq:cutoff-second}
from the corresponding projected equations gives
\begin{equation*}
\begin{split}
&\displaystyle -\Delta (\varphi-\chi_{1}\varphi)-\epsilon^{\frac{N-2}{N-4}}(\varphi-\chi_{1}\varphi)-(2^{*}-1)V^{2^{*}-2}_{1} (\varphi-\chi_{1} \varphi)\\
\displaystyle =&(1-\chi_{1})\Big(l_{1}+R_{1}(\varphi)+\beta \lvert  V_2+\psi\rvert^{\frac{2^{*}}{2}}\lvert  V_1+\varphi\rvert^{\frac{2^{*}}{2}-2}( V_1+\varphi)\Big)+\Delta\chi_{1} \varphi +2\nabla\chi_{1}\cdot\nabla\varphi +\sum_{j=0}^{N}c_{j}U^{2^{*}-2}_{1}Y_{1,j}
\end{split}
\end{equation*}
and
\begin{equation*}
\begin{split}
&\displaystyle -\Delta (\psi-\chi_{2} \psi) -\epsilon^{\frac{N-2}{N-4}}(\psi-\chi_{2} \psi )-(2^{*}-1) V^{2^{*}-2}_{2}(\psi -\chi_{2} \psi )\\
\displaystyle =&(1-\chi_{2})\Big(l_{2}+R_{2}( \psi )+\beta \lvert  V_1+\varphi\rvert^{\frac{2^{*}}{2}}\lvert  V_2+\psi\rvert^{\frac{2^{*}}{2}-2}( V_2+\psi)\Big)+\Delta\chi_{2} \psi +2\nabla\chi_{2}\cdot\nabla\psi +\sum_{j=0}^{N}d_{j}U^{2^{*}-2}_{2}Y_{2,j}.
\end{split}
\end{equation*}
\smallskip

\medskip
The preceding inner-data formulation is used because the standard global fixed point
map is not a contraction in the full perturbation space. The outer operator
\(S\) first resolves the sublinear competitive terms away from the bubble cores, and
the remaining fixed point argument is performed only on the inner data.

By Lemma~\ref{deadout}, Remark~\ref{remdead}, and passage to the
uniform limit in the construction of $S$, every
$(\varphi,\psi)\in\mathcal X_\epsilon$ satisfies
\[
V_2+\psi=0\quad\text{in }P_2,
\qquad
V_1+\varphi=0\quad\text{in }Q_2.
\]
Since $\operatorname{supp}(1-\chi_1)\subset P_2$ and
$\operatorname{supp}(1-\chi_2)\subset Q_2$, the coupling terms vanish on the
inner parts where $1-\chi_i\ne0$. This is why the operator below contains only
$l_i+R_i$ in the $(1-\chi_i)$-terms.

We now define the localized correction map. For
$(\varphi,\psi)\in\mathcal X_\epsilon$, set
\begin{equation}\label{eq:B-map}
\mathcal B_\epsilon(\varphi,\psi):=
\left(\begin{matrix}
\displaystyle
L_{1}\big(\varphi \Delta\chi_{1}+2\nabla\chi_{1}\cdot\nabla \varphi
 +(1-\chi_{1})\big(l_{1}+R_{1}(\varphi)\big)\big)\\[2mm]
\displaystyle
L_{2}\big(\psi \Delta\chi_{2}+2\nabla\chi_{2}\cdot\nabla \psi
 +(1-\chi_{2})\big(l_{2}+R_{2}(\psi)\big)\big)
\end{matrix}\right)^{\!\top}.
\end{equation}
The gradients are used only on the cut-off annuli, where they are defined by
the local $C^1$ property of $S$ stated above. The linear regularity following
Propositions~\ref{prop:linear-first} and~\ref{prop:linear-second} gives
$\mathcal B_\epsilon(\varphi,\psi)\in
\mathbb E\cap C_0(\Omega_\epsilon)^2$. Lemma~\ref{lem13} proves that this pair
belongs to $\Lambda_\epsilon$ for small $\epsilon$; we can then define
\begin{equation}\label{eq:T-map}
T_\epsilon(\varphi,\psi):=
S\big(\mathcal B_\epsilon(\varphi,\psi)\big).
\end{equation}

\medskip

To estimate the inner-data map induced by $\mathcal B_\epsilon$, we first establish the following
a priori estimates.
\begin{Lem}\label{lem10}
We have
\begin{equation*}
\begin{split}
\lVert l_{1}\rVert_{L^{\infty}(P_{2})},~\lVert l_{2}\rVert_{L^{\infty}(Q_{2})}\leq C\epsilon^{\gamma_1},
\end{split}
\end{equation*}
where $\gamma_1=(N-2)/(N-4)$.
\end{Lem}
\begin{proof}
It suffices to estimate $\lVert l_{1}\rVert_{L^{\infty}(P_{2})}$, as the bound for $\lVert l_{2}\rVert_{L^{\infty}(Q_{2})}$ follows analogously.

In fact, for \(y\in P_2\), the projection estimate gives
\[
V_1(y)=U_1(y)-C_N\lambda_1^{\frac{N-2}{2}}
\epsilon^{\frac{N-2}{N-4}}
H_\Omega\big(\epsilon^{\frac{1}{N-4}}y,\xi_1\big)
+O\left(\epsilon^{\frac{N}{N-4}}\right),
\]
where \(H_\Omega\) is the regular part of the Dirichlet Green function in \(\Omega\). Hence
\begin{equation*}
\begin{split}
|l_{1}(y)|
&=\left|\epsilon^{\frac{N-2}{N-4}}V_1(y)
+\big(V^{2^{*}-1}_{1}(y)-U^{2^{*}-1}_{1}(y)\big)\right|\\
&\le C\epsilon^{\frac{N-2}{N-4}}U_1(y)
+C U_1^{2^{*}-2}(y)\epsilon^{\frac{N-2}{N-4}}
H_\Omega\big(\epsilon^{\frac{1}{N-4}}y,\xi_1\big)
+C\epsilon^{\frac{N}{N-4}}U_1^{2^*-2}(y)\\
&\le C\epsilon^{\frac{N-2}{N-4}},
\qquad y\in P_2.
\end{split}
\end{equation*}

\end{proof}

Set
\[
m_N:=\min\{2,2^*-1\}.
\]
\begin{Lem}\label{lem11}
Let $(\varphi,\psi)\in\mathcal X_\epsilon$. Then
\begin{align}\label{eq:S-inner-P2}
\|\varphi\|_{L^\infty(P_2)}+
\|\psi\|_{L^\infty(Q_2)}
&\le C\epsilon^{\kappa_0},\\
\|R_1(\varphi)\|_{L^\infty(P_2)}+
\|R_2(\psi)\|_{L^\infty(Q_2)}
&\le C\epsilon^{m_N\kappa_0}.
\end{align}
Moreover, on the annuli supporting $\nabla\chi_i$ and $\Delta\chi_i$,
\begin{align}\label{eq:S-cutoff-annuli}
|\varphi|+|\nabla\varphi|
&\le C\left(
\epsilon^{\alpha_1+\frac{2\theta_1}{B_\tau}}
+\epsilon^{\tilde\alpha_1+\frac{\tilde\theta_1}{N-4}}
\right),\\
|\psi|+|\nabla\psi|
&\le C\left(
\epsilon^{\alpha_2+\frac{2\theta_2}{B_\tau}}
+\epsilon^{\tilde\alpha_2+\frac{\tilde\theta_2}{N-4}}
\right).
\end{align}
\end{Lem}
\begin{proof}
We first take $(\varphi,\psi)=S(g)$ with
$g\in\Lambda_\epsilon\cap H_0^1(\Omega_\epsilon)^2$, so that the pair
is the perturbation of an actual exterior minimizer. We prove the
estimates for $\varphi$. On $P_1$, the prescribed inner data satisfy
$|\varphi|\le\epsilon^{\gamma_0}\le\epsilon^{\kappa_0}$. On
$P_2\setminus P_1$, Lemma~\ref{lem8} gives
\[
|\varphi(y)|
\le C\epsilon^{\alpha_1}r_\epsilon^{-\theta_1}
+C\epsilon^{\tilde\alpha_1}
(1+|y-\tilde\xi_2|)^{-\tilde\theta_1}.
\]
Since $r_\epsilon^2=o(|\tilde\xi_1-\tilde\xi_2|)$, one has
$|y-\tilde\xi_2|\ge c\epsilon^{-1/(N-4)}$ throughout $P_2$. Hence
\[
|\varphi(y)|
\le C\epsilon^{\kappa_0}
+C\epsilon^{\tilde\alpha_1+\frac{\tilde\theta_1}{N-4}}
\le C\epsilon^{\kappa_0}.
\]
The last inequality follows from
\[
(N-4)\tilde\theta+\bar\theta<2,
\qquad\text{equivalently}\qquad
\tilde\alpha_1+\frac{\tilde\theta_1}{N-4}>\gamma_1,
\]
together with $\gamma_1>\kappa_0$, which follows from
\eqref{gamma-range}. This proves the first estimate in
\eqref{eq:S-inner-P2}. On the annulus
$\frac14r_\epsilon^2\le|y-\tilde\xi_1|\le\frac34r_\epsilon^2$,
Lemma~\ref{lem8} and its gradient estimate give
\eqref{eq:S-cutoff-annuli}. The argument for $\psi$ is the same.

For $N=5$, Taylor's formula gives
$|R_1(\varphi)|\le C(V_1^{1/3}|\varphi|^2+|\varphi|^{7/3})
\le C|\varphi|^2$. For $N\ge6$,
$|R_1(\varphi)|\le C|\varphi|^{2^*-1}$. The second estimate in
\eqref{eq:S-inner-P2} follows, and the proof for $R_2$ is identical.
The general case $(\varphi,\psi)\in\mathcal X_\epsilon$ follows by
approximating its prescribed inner data by energy-compatible elements of
$\Lambda_\epsilon$ and using the uniform and local $C^1$ convergence in the
construction of $S$.
\end{proof}

\medskip

\medskip

\begin{Lem}\label{lem13}
For $\epsilon>0$ sufficiently small, the map $T_\epsilon$ in
\eqref{eq:T-map} is well defined and is a contraction on
$(\mathcal X_\epsilon,d_{\mathcal X})$. More precisely, if
$\mathcal B_\epsilon(\varphi,\psi)=(b_1,b_2)$, then
\begin{equation}\label{eq:strict-self-map}
\|b_1\|_{L^\infty(\Omega_\epsilon)},\
\|b_2\|_{L^\infty(\Omega_\epsilon)}
\le\frac12\epsilon^{\gamma_0}
\qquad\text{for every }(\varphi,\psi)\in\mathcal X_\epsilon.
\end{equation}
Consequently, $T_\epsilon$ has a unique fixed point
$(\varphi_\epsilon,\psi_\epsilon)\in\mathcal X_\epsilon$. This fixed point
belongs to $S(\Lambda_\epsilon\cap H_0^1(\Omega_\epsilon)^2)$ and hence is
the perturbation of an actual exterior minimizer.
\end{Lem}

\begin{proof}
(\romannumeral1) The map is well defined. Let
$(\varphi,\psi)\in\mathcal X_\epsilon$. By
Proposition~\ref{prop:linear-first} and Lemmas~\ref{lem10} and~\ref{lem11},
\begin{align*}
&\left\|L_1\left(
\varphi\Delta\chi_1+2\nabla\chi_1\cdot\nabla\varphi
+(1-\chi_1)(l_1+R_1(\varphi))
\right)\right\|_{L^\infty(\Omega_\epsilon)}\\
&\quad\le C\Big(
\epsilon^{\alpha_1+\frac{2\theta_1}{B_\tau}-\frac{2}{B_\tau}}
+\epsilon^{\tilde\alpha_1+\frac{\tilde\theta_1}{N-4}-\frac{2}{B_\tau}}
+\epsilon^{\gamma_1-\frac{4}{B_\tau}}
+\epsilon^{m_N\kappa_0-\frac{4}{B_\tau}}
\Big).
\end{align*}
The first two terms come from \eqref{eq:S-cutoff-annuli}, using
$|\nabla\chi_1|=O(r_\epsilon^{-2})$,
$|\Delta\chi_1|=O(r_\epsilon^{-4})$, and the factor $r_\epsilon^4$ in
Proposition~\ref{prop:linear-first}; the last two terms follow from
Lemmas~\ref{lem10} and~\ref{lem11}. Lemma~\ref{lem:exponent-choice} gives
$\sigma_0>0$ such that the right-hand side is bounded by
$C\epsilon^{\gamma_0+\sigma_0}$. After reducing $\epsilon_0$, this is at most
$\frac12\epsilon^{\gamma_0}$. The second component is identical. Since the
linear solutions belong to $\mathbb E\cap C_0(\Omega_\epsilon)^2$,
\eqref{eq:strict-self-map} gives
$\mathcal B_\epsilon(\varphi,\psi)\in\Lambda_\epsilon$. Thus
$T_\epsilon(\varphi,\psi)$ is defined and belongs to $\mathcal X_\epsilon$.

\medskip

(\romannumeral2) The map is a contraction. Take
$(\varphi,\psi),(\bar\varphi,\bar\psi)\in\mathcal X_\epsilon$, set
\[
(w,v)=(\varphi,\psi)-(\bar\varphi,\bar\psi),
\qquad
(W,V)=\mathcal B_\epsilon(\varphi,\psi)
-\mathcal B_\epsilon(\bar\varphi,\bar\psi),
\]
and use the distance \eqref{eq:dX}. Since $S$ preserves the prescribed
inner data by \eqref{eq:S-preserves-inner-data},
\begin{align*}
d_{\mathcal X}\big(T_\epsilon(\varphi,\psi),
T_\epsilon(\bar\varphi,\bar\psi)\big)
&=\|W\|_{L^\infty(P_1)}+\|V\|_{L^\infty(Q_1)}.
\end{align*}
When $N=5$,
\[
\|R_1(\varphi)-R_1(\bar\varphi)\|_{L^\infty(P_2)}
\le C\epsilon^{\kappa_0}
 d_{\mathcal X}\big((\varphi,\psi),(\bar\varphi,\bar\psi)\big),
\]
whereas for $N\ge6$,
\[
\|R_1(\varphi)-R_1(\bar\varphi)\|_{L^\infty(P_2)}
\le C\epsilon^{\kappa_0(2^*-2)}
 d_{\mathcal X}\big((\varphi,\psi),(\bar\varphi,\bar\psi)\big).
\]
The corresponding estimates hold for $R_2$. On the supports of
$\nabla\chi_i$ and $\Delta\chi_i$, Lemma~\ref{lem9} gives
\[
|w|+|\nabla w|+|v|+|\nabla v|
\le C\big(r_\epsilon^{-(N-2)}+r_\epsilon^{-(4-\theta_{\mathrm{st}})}\big)
 d_{\mathcal X}\big((\varphi,\psi),(\bar\varphi,\bar\psi)\big).
\]
After applying Propositions~\ref{prop:linear-first} and
\ref{prop:linear-second}, we obtain
\begin{align*}
&d_{\mathcal X}\big(T_\epsilon(\varphi,\psi),
T_\epsilon(\bar\varphi,\bar\psi)\big)\\
&\quad\le \kappa_\epsilon
 d_{\mathcal X}\big((\varphi,\psi),(\bar\varphi,\bar\psi)\big),
\end{align*}
where
\[
\kappa_\epsilon\le C\Big(
 r_\epsilon^{-(N-2)}+r_\epsilon^{-(4-\theta_{\mathrm{st}})}
 +r_\epsilon^{-(N-4)}+r_\epsilon^{-(2-\theta_{\mathrm{st}})}
 +\epsilon^{\kappa_0-4/B_\tau}
 +\epsilon^{(2^*-2)\kappa_0-4/B_\tau}
\Big)=o(1).
\]
The two positive $\epsilon$-exponents are verified in
Lemma~\ref{lem:exponent-choice}. Hence $\kappa_\epsilon\le1/2$ for small
$\epsilon$, and $T_\epsilon$ is a contraction. Since
$\mathcal X_\epsilon$ is complete, it has a unique fixed point. Finally, a
fixed point satisfies
\[
(\varphi_\epsilon,\psi_\epsilon)
=T_\epsilon(\varphi_\epsilon,\psi_\epsilon)
=S\big(\mathcal B_\epsilon(\varphi_\epsilon,\psi_\epsilon)\big),
\]
with $\mathcal B_\epsilon(\varphi_\epsilon,\psi_\epsilon)\in
\Lambda_\epsilon\cap H_0^1(\Omega_\epsilon)^2$; therefore it is an actual
exterior minimizer.
\end{proof}

\medskip

\begin{Lem}\label{lem12}
Let $(\varphi,\psi)$ be the fixed point obtained in Lemma~\ref{lem13}. Then
$(\varphi,\psi)$ solves the projected system \eqref{eqpert}.
\end{Lem}
\begin{proof}
Write
\[
\mathcal B_\epsilon(\varphi,\psi)=(b_1,b_2).
\]
By \eqref{eq:strict-self-map},
$(b_1,b_2)\in\Lambda_\epsilon\cap H_0^1(\Omega_\epsilon)^2$, and the
fixed-point identity gives
\[
(\varphi,\psi)=S(b_1,b_2).
\]
By \eqref{eq:S-preserves-inner-data},
\[
b_1=\varphi\quad\text{in }P_1,
\qquad
b_2=\psi\quad\text{in }Q_1.
\]
Define
\[
\varphi_T=b_1+\chi_1\varphi,
\qquad
\psi_T=b_2+\chi_2\psi.
\]
Both functions belong to $H_0^1(\Omega_\epsilon)$. The supports of
$Y_{1,j}$ and $Y_{2,j}$ lie in $P_1$ and $Q_1$, whereas $\chi_1=0$ in $P_1$
and $\chi_2=0$ in $Q_1$ for small $\epsilon$. Hence
\[
\ell_{1,j}(\chi_1\varphi)=0,
\qquad
\ell_{2,j}(\chi_2\psi)=0,
\qquad j=0,\ldots,N.
\]
Since $(b_1,b_2)\in\mathbb E$, it follows that
$(\varphi_T,\psi_T)\in\mathbb E$. The inner-data identities also give
\[
\varphi_T=\varphi\quad\text{in }P_1,
\qquad
\psi_T=\psi\quad\text{in }Q_1.
\]
Using Propositions~\ref{prop:linear-first} and~\ref{prop:linear-second}, the
cut-off identities \eqref{eq:cutoff-first}--\eqref{eq:cutoff-second}, and the
dead-core vanishing of the coupling terms on
$\operatorname{supp}(1-\chi_i)$, we obtain
\begin{align*}
-\Delta \varphi_T-
\epsilon^{\frac{N-2}{N-4}}\varphi_T
-(2^*-1)V_1^{2^*-2}\varphi_T
&=l_1+R_1(\varphi)
+\beta|\psi+V_2|^{\frac{2^*}{2}}
|\varphi+V_1|^{\frac{2^*}{2}-2}(\varphi+V_1)\\
&\quad+\sum_{j=0}^Nc_jU_1^{2^*-2}Y_{1,j},\\
-\Delta \psi_T-
\epsilon^{\frac{N-2}{N-4}}\psi_T
-(2^*-1)V_2^{2^*-2}\psi_T
&=l_2+R_2(\psi)
+\beta|\varphi+V_1|^{\frac{2^*}{2}}
|\psi+V_2|^{\frac{2^*}{2}-2}(\psi+V_2)\\
&\quad+\sum_{j=0}^Nd_jU_2^{2^*-2}Y_{2,j}.
\end{align*}
Set $(e_1,e_2)=(\varphi_T,\psi_T)-(\varphi,\psi)$. The two pairs have the
same inner data, and all four functions belong to the energy space. Therefore,
after zero extension,
\[
(e_1,e_2)\in
D_0^{1,2}(\mathbb R^N\setminus P_1)\times
D_0^{1,2}(\mathbb R^N\setminus Q_1).
\]
Subtracting the exterior equations gives
\[
-\Delta e_i-\epsilon^{\frac{N-2}{N-4}}e_i
-(2^*-1)V_i^{2^*-2}e_i=0
\]
in the corresponding exterior domain. Testing with $e_i$, using Sobolev's
inequality, and recalling
\[
\|V_1\|_{L^{2^*}(\Omega_\epsilon\setminus P_1)}+
\|V_2\|_{L^{2^*}(\Omega_\epsilon\setminus Q_1)}=o(1),
\qquad
\epsilon^{\frac{N-2}{N-4}}|\Omega_\epsilon|^{2/N}=O(\epsilon),
\]
yields $e_1=e_2=0$ for small $\epsilon$. The displayed equations for
$(\varphi_T,\psi_T)$ therefore become exactly \eqref{eqpert} for
$(\varphi,\psi)$.
\end{proof}

For each fixed sufficiently small $\epsilon>0$, the right-hand sides in the
projected system \eqref{eqpert} are bounded. Standard Dirichlet regularity
therefore yields
\[
\varphi,\psi\in W^{2,q}(\Omega_\epsilon)\cap W_0^{1,q}(\Omega_\epsilon)
\qquad\text{for every }1<q<\infty.
\]
In particular, for $q>N$ the correction is in $C^{1,\alpha}$ for some
$\alpha\in(0,1)$. This fixed-$\epsilon$ regularity justifies the integrations
by parts and the local Pohozaev identities in Section~\ref{sec:5}; no estimate
of the $W^{2,q}$ norm uniform as $\epsilon\to0$ is used.

\medskip

The fixed point also depends continuously on the geometric parameters. The
two lemmas in Appendix~B have separate roles. Lemma~\ref{lem:linear-parameter-dependence}
puts the projected linear problems, including their moving moment conditions,
on fixed Banach spaces. Lemma~\ref{lem:outer-parameter-dependence} transports
admissible inner data between the moving balls and proves continuity of the
exterior minimizer. The proof below then compares an exact fixed point with an
approximate fixed point obtained by transporting its inner data. Here
$C^1(\overline{\Omega_\epsilon})$ always refers to regularity in the spatial
variable; no differentiability of the fixed point with respect to the
geometric parameters is asserted or needed.

\begin{Prop}\label{prop:parameter-continuity}
For each fixed sufficiently small $\epsilon>0$, let
$(\varphi(\mathbf t),\psi(\mathbf t))$ be the fixed point constructed in
Lemma~\ref{lem13}, and let $(c_j(\mathbf t))_{j=0}^N$ and
$(d_j(\mathbf t))_{j=0}^N$ be the corresponding multiplier vectors in
\eqref{eqpert}. Then
\[
\mathbf t\longmapsto
\bigl(\varphi(\mathbf t),\psi(\mathbf t)\bigr)
\]
is continuous from $O_\delta$ into
$C^1(\overline{\Omega_\epsilon})^2$, and the multiplier vectors depend
continuously on $\mathbf t$.
\end{Prop}
\begin{proof}
Fix a compact set $K\Subset O_\delta$ and let
$\mathbf t_m\to\mathbf t$ in $K$. Write
\[
(\varphi_m,\psi_m):=
(\varphi(\mathbf t_m),\psi(\mathbf t_m)),
\qquad
(\varphi,\psi):=
(\varphi(\mathbf t),\psi(\mathbf t)).
\]
For each configuration $\mathbf t'\in K$, let $S_{\mathbf t'}$,
$\mathcal B_{\epsilon,\mathbf t'}$, and $T_{\epsilon,\mathbf t'}$ denote the
operators constructed above, and set
\[
\begin{aligned}
d_{\mathbf t'}\big((\phi_1,\phi_2),(\bar\phi_1,\bar\phi_2)\big)
:=\;&\|\phi_1-\bar\phi_1\|_{L^\infty(P_1(\mathbf t'))}\\
&+\|\phi_2-\bar\phi_2\|_{L^\infty(Q_1(\mathbf t'))}.
\end{aligned}
\]
The contraction estimate in Lemma~\ref{lem13} is uniform on $K$: there is
$\kappa_K<1$ such that every $T_{\epsilon,\mathbf t'}$, $\mathbf t'\in K$,
has contraction constant at most $\kappa_K$ in $d_{\mathbf t'}$.

Set
\[
b:=\mathcal B_{\epsilon,\mathbf t}(\varphi,\psi).
\]
By \eqref{eq:strict-self-map}, both components of $b$ have uniform norm at
most $\frac12\epsilon^{\gamma_0}$, and the fixed-point identity is
$(\varphi,\psi)=S_{\mathbf t}(b)$. Choose $1/2<\rho<\rho'<1$. By
Lemma~\ref{lem:outer-parameter-dependence}, for all sufficiently large $m$
there are transported data
\[
b_m^0:=J_{\mathbf t_m}b\in\Lambda_\epsilon^{\rho'}(\mathbf t_m)
\]
such that
\begin{equation}\label{eq:transported-outer-solution}
(\varphi_m^0,\psi_m^0):=S_{\mathbf t_m}(b_m^0)
\longrightarrow(\varphi,\psi)
\quad\text{uniformly in }\Omega_\epsilon,
\end{equation}
and, after identifying the moving annuli by $\Phi_{\mathbf t_m}$, the
convergence is in $C^1$ on the annuli supporting
$\nabla\chi_{i,\mathbf t}$ and $\Delta\chi_{i,\mathbf t}$.

Define
\[
\widehat b_m:=
\mathcal B_{\epsilon,\mathbf t_m}(\varphi_m^0,\psi_m^0).
\]
The smooth cutoffs satisfy
$\chi_{i,\mathbf t_m}\to\chi_{i,\mathbf t}$ in $C^2$ for fixed $\epsilon$.
On the cutoff annuli, pullback by $\Phi_{\mathbf t_m}$ gives convergence of
$(\varphi_m^0,\psi_m^0)$ and their spatial gradients in $C^1$ on fixed
annuli. Since $\Phi_{\mathbf t_m}\to\operatorname{Id}$ in $C^2$ and agrees
there with the corresponding translations, a change of variables gives
convergence in $L^q(\Omega_\epsilon)$ of the terms containing
$\nabla\chi_i$ and $\Delta\chi_i$. The remaining terms in the localized
sources converge in $L^q$ by the global uniform convergence in
\eqref{eq:transported-outer-solution}. Hence the complete source terms in
\eqref{eq:B-map} converge in $L^q(\Omega_\epsilon)$ for every finite $q$.
Lemma~\ref{lem:linear-parameter-dependence} therefore gives
\begin{equation}\label{eq:B-transport-consistency}
\widehat b_m\longrightarrow b
\quad\text{in }C^1(\overline{\Omega_\epsilon})^2.
\end{equation}
The construction in Lemma~\ref{lem:outer-parameter-dependence} also gives
$b_m^0\to b$ uniformly. Since $S_{\mathbf t_m}$ preserves the prescribed
first component on $P_1(\mathbf t_m)$ and the prescribed second component on
$Q_1(\mathbf t_m)$,
\begin{align*}
&d_{\mathbf t_m}\bigl(
T_{\epsilon,\mathbf t_m}(\varphi_m^0,\psi_m^0),
(\varphi_m^0,\psi_m^0)\bigr)\\
&\qquad=
\|\widehat b_{m,1}-b_{m,1}^0\|_{L^\infty(P_1(\mathbf t_m))}
+\|\widehat b_{m,2}-b_{m,2}^0\|_{L^\infty(Q_1(\mathbf t_m))}
\longrightarrow0.
\end{align*}
Thus $(\varphi_m^0,\psi_m^0)$ is an approximate fixed point for the
contraction associated with $\mathbf t_m$. Since
\[
(\varphi_m,\psi_m)
=T_{\epsilon,\mathbf t_m}(\varphi_m,\psi_m),
\]
we obtain
\[
\begin{aligned}
&d_{\mathbf t_m}\bigl((\varphi_m,\psi_m),(\varphi_m^0,\psi_m^0)\bigr)\\
&\quad\le d_{\mathbf t_m}\bigl(
T_{\epsilon,\mathbf t_m}(\varphi_m,\psi_m),
T_{\epsilon,\mathbf t_m}(\varphi_m^0,\psi_m^0)\bigr)\\
&\qquad\quad+d_{\mathbf t_m}\bigl(
T_{\epsilon,\mathbf t_m}(\varphi_m^0,\psi_m^0),
(\varphi_m^0,\psi_m^0)\bigr)\\
&\quad\le\kappa_K
 d_{\mathbf t_m}\bigl((\varphi_m,\psi_m),(\varphi_m^0,\psi_m^0)\bigr)+o(1).
\end{aligned}
\]
Hence
\begin{equation}\label{eq:fixed-point-inner-parameter-continuity}
d_{\mathbf t_m}\bigl((\varphi_m,\psi_m),(\varphi_m^0,\psi_m^0)\bigr)
\longrightarrow0.
\end{equation}
The uniform stability estimate of Lemma~\ref{lem9}, followed by
\eqref{eq:transported-outer-solution}, now yields
\[
(\varphi_m,\psi_m)\longrightarrow(\varphi,\psi)
\quad\text{uniformly in }\Omega_\epsilon.
\]
Its gradient part also gives $C^1$ convergence on the cutoff annuli.

Finally set
\[
b_m:=\mathcal B_{\epsilon,\mathbf t_m}(\varphi_m,\psi_m).
\]
The preceding uniform and annular $C^1$ convergence implies convergence of
the corresponding localized sources in $L^q$ for every finite $q$.
Lemma~\ref{lem:linear-parameter-dependence} gives
$b_m\to b$ in $C^1$ and, at the same time, convergence of the multiplier
vectors appearing in these projected linear problems. These are exactly the
multipliers in \eqref{eqpert}.

Subtracting the two global projected systems corresponding to
$\mathbf t_m$ and $\mathbf t$, the projected bubbles and kernel functions
converge smoothly for fixed $\epsilon$, the multiplier terms converge by the
preceding paragraph, and the nonlinear terms converge in $L^q$ for every
finite $q$. For the competitive terms one uses only the continuity, on
bounded sets, of
\[
(r,s)\longmapsto |s|^{2^*/2}|r|^{2^*/2-2}r;
\]
no differentiation at $r=0$ is involved. Standard Dirichlet $W^{2,q}$
estimates with $q>N$ yield
\[
(\varphi_m,\psi_m)\longrightarrow(\varphi,\psi)
\quad\text{in }C^1(\overline{\Omega_\epsilon})^2.
\]
This proves the proposition.
\end{proof}

At this stage, the first step of the reduction is complete: the fixed-point
theorem yields a solution of system~\eqref{eqpert}. The remaining task is to
choose the geometric parameters so that all multipliers associated with the
approximate kernel vanish.

\section{The finite-dimensional reduced problem}\label{sec:5}

The second step is to choose the parameter quadruple
$(\lambda_1,\xi_1,\lambda_2,\xi_2)$ so that $c_j=d_j=0$ in \eqref{eqpert} for
$j=0,1,\ldots,N$. This is the objective of the present section.

We first prove a sufficient condition for the vanishing of all multipliers \(c_j,d_j\), \(j=0,1,\ldots,N\).

Set $\partial_j u := \frac{\partial u}{\partial x_j}$. Let
 $\chi\in C_c^\infty(\mathbb R^N)$ satisfy
\[
0\le\chi\le1,\qquad \chi=1\ \text{in }B_1(0),\qquad
\chi=0\ \text{in }\mathbb R^N\setminus B_2(0).
\]
 Choose $\alpha_{\mathrm{cut}}$ such that
\[
\max\left\{
\frac{N-2}{(N+4)(N-4)},
\frac{\gamma_1-\alpha_1}{\theta_1}
\right\}<\alpha_{\mathrm{cut}}<\frac{1}{N-4},
\]
and set
\[
\bar\chi(y):=\chi(\epsilon^{\alpha_{\mathrm{cut}}} y),\qquad
\bar Z_i^j=\bar\chi(\cdot-\tilde\xi_i)Y_i^j,
\quad i=1,2,\quad j=0,1,\ldots,N.
\]
This interval is nonempty. Indeed,
\[
\frac1{N-4}-\frac{\gamma_1-\alpha_1}{\theta_1}
=
\frac{1-\vartheta+(N-4)\theta'}{(N-4)\theta_1}>0,
\]
and the other lower endpoint is plainly smaller than $1/(N-4)$.

First, we have the following result.

\medskip

\begin{Lem}\label{lem:multiplier-vanishing}
Suppose that $\lambda_1,\lambda_2,\tilde\xi_1$ and $\tilde\xi_2$ satisfy
\begin{equation}\label{eq:scale-condition-first}
\begin{split}
\int_{\Omega_\epsilon}\Big(-\Delta u_{1,\epsilon}-\epsilon^{\frac{N-2}{N-4}} u_{1,\epsilon} -  \lvert u_{1,\epsilon}\rvert^{2^{*}-2}u_{1,\epsilon}-\beta \lvert u_{2,\epsilon}\rvert^{\frac{2^{*}}{2}}\lvert u_{1,\epsilon}\rvert^{\frac{2^{*}}{2}-2}u_{1,\epsilon}\Big)\bar Z_{1}^0=0,
\end{split}
\end{equation}
\begin{equation}\label{eq:scale-condition-second}
\begin{split}
\int_{\Omega_\epsilon}\Big(-\Delta u_{2,\epsilon}-\epsilon^{\frac{N-2}{N-4}} u_{2,\epsilon} -  \lvert u_{2,\epsilon}\rvert^{2^{*}-2}u_{2,\epsilon}-\beta \lvert u_{1,\epsilon}\rvert^{\frac{2^{*}}{2}}\lvert u_{2,\epsilon}\rvert^{\frac{2^{*}}{2}-2}u_{2,\epsilon}\Big) \bar Z_{2}^0=0,
\end{split}
\end{equation}
and for some $\rho>0$,
\begin{equation}\label{eq:translation-condition-first}
\begin{aligned}
\int_{B(\tilde\xi_1,\epsilon^{-\frac{1}{N-4}}\rho)}
\Big(&-\Delta u_{1,\epsilon}
-\epsilon^{\frac{N-2}{N-4}}u_{1,\epsilon}
-\lvert u_{1,\epsilon}\rvert^{2^{*}-2}u_{1,\epsilon} \\
&-
\beta \lvert u_{2,\epsilon}\rvert^{\frac{2^{*}}{2}}
\lvert u_{1,\epsilon}\rvert^{\frac{2^{*}}{2}-2}u_{1,\epsilon}\Big)
\partial_j u_{1,\epsilon}=0,
\quad j=1,\ldots,N,
\end{aligned}
\end{equation}
\begin{equation}\label{eq:translation-condition-second}
\begin{aligned}
\int_{B(\tilde\xi_2,\epsilon^{-\frac{1}{N-4}}\rho)}
\Big(&-\Delta u_{2,\epsilon}
-\epsilon^{\frac{N-2}{N-4}}u_{2,\epsilon}
-\lvert u_{2,\epsilon}\rvert^{2^{*}-2}u_{2,\epsilon} \\
&-
\beta \lvert u_{1,\epsilon}\rvert^{\frac{2^{*}}{2}}
\lvert u_{2,\epsilon}\rvert^{\frac{2^{*}}{2}-2}u_{2,\epsilon}\Big)
\partial_j u_{2,\epsilon}=0,
\quad j=1,\ldots,N.
\end{aligned}
\end{equation}
Then $c_j=d_j=0$ for $j=0,1,\ldots,N$.
\end{Lem}

\begin{proof}
We only discuss the coefficients \(c_\ell\); the proof for \(d_\ell\) is the
same. Testing the projected equation against \(\bar Z_1^0\) and against
\(\partial_j u_{1,\epsilon}\) in the ball
\(B(\tilde\xi_1,\epsilon^{-1/(N-4)}\rho)\), the assumptions
\eqref{eq:scale-condition-first} and \eqref{eq:translation-condition-first} give a homogeneous linear system for
\((c_0,\ldots,c_N)\).  After translating to $\tilde\xi_1$ and using Lemma~\ref{lem8}, the scale
test converges to $Y_0$, whereas the test $\partial_j u_{1,\epsilon}$
converges to $-Y_j$. Hence the coefficient matrix converges to $D\mathcal G$,
where
\[
D=\operatorname{diag}(1,-1,\ldots,-1),\qquad
\mathcal G_{\ell k}=\int_{\mathbb R^N}
\chi_0U_{\lambda_1,0}^{2^*-2}Y_\ell Y_k,
\quad 0\le \ell,k\le N.
\]
The Gram matrix $\mathcal G$ is positive definite, and therefore
$D\mathcal G$ is nonsingular. Hence the coefficient matrix is nonsingular
for $\epsilon$ small, and the homogeneous system has only the zero solution. Therefore
\(c_0=\cdots=c_N=0\). The same argument yields \(d_0=\cdots=d_N=0\).
\end{proof}

\medskip

The proof of the main result reduces to solving \eqref{eq:scale-condition-first}, \eqref{eq:translation-condition-first}, \eqref{eq:scale-condition-second}, and \eqref{eq:translation-condition-second}. The key step is to identify the leading-order contributions on the left-hand side of each equation.

\begin{Lem}\label{lem:scale-expansion}
Uniformly on compact subsets of $O_\delta$, we have
\begin{equation}\label{eq:scale-expansion-first}\begin{split}
&\int_{\Omega_\epsilon}\Big(-\Delta u_{1,\epsilon} -\epsilon^{\frac{N-2}{N-4}} u_{1,\epsilon}-  \lvert u_{1,\epsilon}\rvert^{2^{*}-2}u_{1,\epsilon}-\beta \lvert u_{2,\epsilon}\rvert^{\frac{2^{*}}{2}}\lvert u_{1,\epsilon}\rvert^{\frac{2^{*}}{2}-2}u_{1,\epsilon}\Big)\bar Z_{1}^0\\
\displaystyle &=\frac{N-2}{2}C_N\lambda_1^{N-3}\epsilon^{\frac{N-2}{N-4}}\varphi_\Omega(\xi_1)\int_{\mathbb R^N} U_{1,0}^{2^*-1}\,dx-\lambda_1\epsilon^{\frac{N-2}{N-4}}\int_{\mathbb R^N} U^2_{1,0}(x)dx+o\Big(\epsilon^{\frac{N-2}{N-4}}\Big)
\end{split}
\end{equation}
and
\begin{equation}\label{eq:scale-expansion-second}
\begin{split}
&\int_{\Omega_\epsilon}\Big(-\Delta u_{2,\epsilon} -\epsilon^{\frac{N-2}{N-4}} u_{2,\epsilon} -  \lvert u_{2,\epsilon}\rvert^{2^{*}-2}u_{2,\epsilon}-\beta \lvert u_{1,\epsilon}\rvert^{\frac{2^{*}}{2}}\lvert u_{2,\epsilon}\rvert^{\frac{2^{*}}{2}-2}u_{2,\epsilon}\Big)\bar Z_{2}^0\\
\displaystyle &=\frac{N-2}{2}C_N\lambda_2^{N-3}\epsilon^{\frac{N-2}{N-4}}\varphi_\Omega(\xi_2)\int_{\mathbb R^N} U_{1,0}^{2^*-1}\,dx-\lambda_2\epsilon^{\frac{N-2}{N-4}}\int_{\mathbb R^N} U^2_{1,0}(x)dx+o\Big(\epsilon^{\frac{N-2}{N-4}}\Big),
\end{split}
\end{equation}
where $C_N=[N(N-2)]^{(N-2)/4}$ is the normalization constant in the
standard positive bubble $U_{\lambda,\xi}$.
\end{Lem}

\begin{proof}
We retain the coupling contribution and estimate it separately. Using
$-\Delta V_1=U_1^{2^*-1}$, we rewrite the left-hand side of \eqref{eq:scale-expansion-first} as
 \begin{align*}
&\int_{\Omega_\epsilon}\Big(-\Delta u_{1,\epsilon}
-\epsilon^{\frac{N-2}{N-4}}u_{1,\epsilon}
-\lvert u_{1,\epsilon}\rvert^{2^{*}-2}u_{1,\epsilon}
-\beta\lvert u_{2,\epsilon}\rvert^{\frac{2^{*}}{2}}
 \lvert u_{1,\epsilon}\rvert^{\frac{2^{*}}{2}-2}u_{1,\epsilon}\Big)\bar Z_{1}^{0}\\
={}&\int_{\Omega_\epsilon}\Big(-\Delta(V_1+\varphi)
-\epsilon^{\frac{N-2}{N-4}}u_{1,\epsilon}
-\lvert u_{1,\epsilon}\rvert^{2^{*}-2}u_{1,\epsilon}
-\beta |u_{2,\epsilon}|^p|u_{1,\epsilon}|^{p-2}u_{1,\epsilon}\Big)\bar Z_{1}^{0}\\
={}&\int_{\Omega_\epsilon}\Big(U_1^{2^*-1}-(V_1+\varphi)^{2^*-1}\Big)\bar Z_1^0
+\int_{\Omega_\epsilon}\varphi(-\Delta\bar Z_1^0)
-\epsilon^{\frac{N-2}{N-4}}\int_{\Omega_\epsilon}u_{1,\epsilon}\bar Z_1^0\\
&-\beta\int_{\Omega_\epsilon}|u_{2,\epsilon}|^p|u_{1,\epsilon}|^{p-2}u_{1,\epsilon}\bar Z_1^0\\
={}& I_1+I_2+I_3+I_4+I_5,
\end{align*}
where
 \begin{align*}
I_1&:=\int_{\Omega_\epsilon}(U_1^{2^*-1}-V_1^{2^*-1})\bar Z_1^0,\\
I_2&:=-\int_{\Omega_\epsilon}\Big((V_1+\varphi)^{2^*-1}-V_1^{2^*-1}
-(2^*-1)U_1^{2^*-2}\varphi\Big)\bar Z_1^0,\\
I_3&:=\int_{\Omega_\epsilon}\varphi\Big(-\Delta\bar Z_1^0
-(2^*-1)U_1^{2^*-2}\bar Z_1^0\Big),\\
I_4&:=-\beta\int_{\Omega_\epsilon}\lvert u_{2,\epsilon}\rvert^{\frac{2^*}{2}}
\lvert u_{1,\epsilon}\rvert^{\frac{2^*}{2}-2}u_{1,\epsilon}\bar Z_1^0,\\
I_5&:=-\epsilon^{\frac{N-2}{N-4}}\int_{\Omega_\epsilon}u_{1,\epsilon}\bar Z_1^0.
\end{align*}

Put $R_{\mathrm{cut},\epsilon}=\epsilon^{-\alpha_{\mathrm{cut}}}$.  The cut-off can be removed from the
leading parts of $I_1$ and $I_5$ at a lower order.  Indeed, the maximum
principle and the projection expansion give $0\le V_1\le U_1$ and
$0\le U_1-V_1\le C\epsilon^{\gamma_1}$.  Since
$|Y_1^0(y)|\le C(1+|y-\tilde\xi_1|)^{2-N}$,
\begin{align}\label{eq:cutoff-tail-I1}
&\left|\int_{|y-\tilde\xi_1|\ge R_{\mathrm{cut},\epsilon}}
(U_1^{2^*-1}-V_1^{2^*-1})Y_1^0\,dy\right|\\
&\qquad\le C\epsilon^{\gamma_1}
\int_{R_{\mathrm{cut},\epsilon}}^{\infty}r^{-4}r^{2-N}r^{N-1}\,dr
\le C\epsilon^{\gamma_1+2\alpha_{\mathrm{cut}}}
=o(\epsilon^{\gamma_1}).
\end{align}
Similarly,
\begin{align}\label{eq:cutoff-tail-I5}
\epsilon^{\gamma_1}
\left|\int_{|y-\tilde\xi_1|\ge R_{\mathrm{cut},\epsilon}}V_1Y_1^0\,dy\right|
&\le C\epsilon^{\gamma_1}
\int_{R_{\mathrm{cut},\epsilon}}^{\infty}r^{4-2N}r^{N-1}\,dr\\
&\le C\epsilon^{\gamma_1+\alpha_{\mathrm{cut}}(N-4)}
=o(\epsilon^{\gamma_1}).
\end{align}
Thus replacing $\bar\chi(\cdot-\tilde\xi_1)$ by $1$ in these two leading
integrals changes them only by $o(\epsilon^{\gamma_1})$.

For $I_1$, the scaling identities for the projections $P$ and $P_\epsilon$ give
\begin{align*}
I_1
&=\int_{\Omega_\epsilon}\big(U_1^{2^*-1}(y)-V_1^{2^*-1}(y)\big)\bar Z_1^0(y)\,dy\\
&=\epsilon^{\frac{1}{N-4}}\int_{\Omega}\Big(
U_{\epsilon^{\frac{1}{N-4}}\lambda_1,\xi_1}^{2^*-1}
-PU_{\epsilon^{\frac{1}{N-4}}\lambda_1,\xi_1}^{2^*-1}\Big)
\bar\chi\big(\epsilon^{-\frac{1}{N-4}}(x-\xi_1)\big)
Y^0_{\epsilon^{\frac{1}{N-4}}\lambda_1,\xi_1}(x)\,dx\\
&=\epsilon^{\frac{1}{N-4}}(2^*-1)
\int_{\Omega}U_{\epsilon^{\frac{1}{N-4}}\lambda_1,\xi_1}^{2^*-2}
R_{\epsilon^{\frac{1}{N-4}}\lambda_1,\xi_1}
Y^0_{\epsilon^{\frac{1}{N-4}}\lambda_1,\xi_1}\,dx
+o\left(\epsilon^{\frac{N-2}{N-4}}\right)\\
&=C_N\lambda_1^{\frac{N-2}{2}}
\epsilon^{\frac{1}{N-4}}\epsilon^{\frac{N-2}{2(N-4)}}(2^*-1)
\varphi_\Omega(\xi_1)
\int_{\Omega}U_{\epsilon^{\frac{1}{N-4}}\lambda_1,\xi_1}^{2^*-2}
Y^0_{\epsilon^{\frac{1}{N-4}}\lambda_1,\xi_1}\,dx
+o\left(\epsilon^{\frac{N-2}{N-4}}\right)\\
&=C_N\lambda_1^{\frac{N-2}{2}}
\epsilon^{\frac{1}{N-4}}\epsilon^{\frac{N-2}{2(N-4)}}
\varphi_\Omega(\xi_1)
\partial_{\epsilon^{\frac{1}{N-4}}\lambda_1}
\left(\int_{\Omega}U_{\epsilon^{\frac{1}{N-4}}\lambda_1,\xi_1}^{2^*-1}(x)\,dx\right)
+o\left(\epsilon^{\frac{N-2}{N-4}}\right)\\
&=\frac{N-2}{2}C_N\lambda_1^{N-3}
\epsilon^{\frac{N-2}{N-4}}\varphi_\Omega(\xi_1)A_0
+o\left(\epsilon^{\frac{N-2}{N-4}}\right).
\end{align*}
where
\[
A_0=\int_{\mathbb R^N}U_{1,0}^{2^*-1}\,dx,\qquad
R_{\epsilon^{\frac{1}{N-4}}\lambda_1,\xi_1}
=U_{\epsilon^{\frac{1}{N-4}}\lambda_1,\xi_1}
-PU_{\epsilon^{\frac{1}{N-4}}\lambda_1,\xi_1},
\]
and
\[
R_{\lambda,\xi}(x)=\lambda^{\frac{N-2}{2}}C_NH_\Omega(x,\xi)
+O\left(\lambda^{\frac{N+2}{2}}\right)
\qquad\text{as }\lambda\to0.
\]
\medskip

 On the support of $\bar\chi(\cdot-\tilde\xi_1)$, Lemma~\ref{lem8}
and the choice of $\alpha_{\mathrm{cut}}$ give
\[
|\varphi|=o(1)U_1,\qquad V_1=U_1(1+o(1))
\]
uniformly. Hence, for small $\epsilon$, one has $|\varphi|\le V_1/2$, and
the segment joining $V_1$ and $V_1+\varphi$ stays between fixed positive
multiples of $U_1$. Taylor's formula for
$F(s)=|s|^{2^*-2}s$ on this positive segment gives
\[
|F(V_1+\varphi)-F(V_1)-(2^*-1)V_1^{2^*-2}\varphi|
\le C U_1^{2^*-3}|\varphi|^2.
\]
Since $V_1$ and $U_1$ are comparable there, the mean value theorem also gives
\[
|V_1^{2^*-2}-U_1^{2^*-2}|
\le C U_1^{2^*-3}|V_1-U_1|.
\]
Consequently,
 \[
\begin{split}
|I_2|
&\le C\int_{|y-\tilde\xi_1|\le2\epsilon^{-\alpha_{\mathrm{cut}}}}
U_1^{2^*-3}|\varphi|^2|\bar Z_1^0|\\
&\quad+C\int_{|y-\tilde\xi_1|\le2\epsilon^{-\alpha_{\mathrm{cut}}}}
U_1^{2^*-3}|U_1-V_1|\,|\varphi|\,|\bar Z_1^0|\\
&=o\left(\epsilon^{\frac{N-2}{N-4}}\right).
\end{split}
\]
 The last estimate follows from Lemma~\ref{lem8}, the projection expansion in
Lemma~\ref{lemRj0}, and the same radial integrations used in the following
estimates.

\medskip

Since
 \[
-\Delta\bar Z_1^0-(2^*-1)U_1^{2^*-2}\bar Z_1^0
=-2\nabla\bar\chi(\cdot-\tilde\xi_1)\cdot\nabla Y_1^0
-Y_1^0\Delta\bar\chi(\cdot-\tilde\xi_1),
\]
 this term is supported in the annulus
\(
\epsilon^{-\alpha_{\mathrm{cut}}}\le |y-\tilde\xi_1|\le2\epsilon^{-\alpha_{\mathrm{cut}}}
\).
On that annulus,
\[
\left|-\Delta\bar Z_1^0-(2^*-1)U_1^{2^*-2}\bar Z_1^0\right|
\le C(1+|y-\tilde\xi_1|)^{-N}.
\]
Since \(\alpha_{\mathrm{cut}}<1/(N-4)\), the annulus remains at distance comparable to
\(\epsilon^{-1/(N-4)}\) from \(\tilde\xi_2\).  Lemma~\ref{lem8} therefore gives
\begin{equation}\label{eq:scale-error-bound}
\begin{split}
|I_3|
&\le C\epsilon^{\alpha_1+\alpha_{\mathrm{cut}}\theta_1}
+C\epsilon^{\tilde\alpha_1+\frac{\tilde\theta_1}{N-4}}\\
&=o\left(\epsilon^{\frac{N-2}{N-4}}\right),
\end{split}
\end{equation}
by the choice of \(\alpha_{\mathrm{cut}}\) and the auxiliary parameters.
\medskip

From Lemmas~\ref{lem7} and~\ref{lemB1}, we also have
 \begin{equation*}
\begin{split}
| I_4|&=\Big|-\beta \int_{\Omega_\epsilon}\lvert u_{2,\epsilon}\rvert^{\frac{2^{*}}{2}}\lvert u_{1,\epsilon}\rvert^{\frac{2^{*}}{2}-2}u_{1,\epsilon} \bar Z_{1}^0 \Big|\\
\displaystyle \leq& C\int_{\Omega_\epsilon}\Big(\frac{1}{(1+\lvert y-\tilde\xi_1\rvert)^{N-2}}\Big)^{\frac{2^*}{2}-1}\Big( \frac{1}{(1+\lvert y-\tilde\xi_2\rvert)^{N-2}}\Big)^{\frac{2^*}{2}}\frac{1}{(1+\lvert y-\tilde\xi_1\rvert)^{N-2}}\\
\displaystyle \leq& C\int_{\Omega_\epsilon}\Big(\frac{1}{(1+\lvert y-\tilde\xi_1\rvert)^{N-2}}\Big)^{\frac{2^*}{2}}\Big( \frac{1}{(1+\lvert y-\tilde\xi_2\rvert)^{N-2}}\Big)^{\frac{2^*}{2}}\\
\displaystyle \leq& C \frac{1}{\lvert \tilde\xi_1-\tilde\xi_2\rvert^{N-\theta_{\mathrm{st}}}}\int_{\Omega_\epsilon}\Big(\frac{1}{(1+\lvert y-\tilde\xi_1\rvert)^{N+\theta_{\mathrm{st}}}}+\frac{1}{(1+\lvert y-\tilde\xi_2\rvert)^{N+\theta_{\mathrm{st}}}}\Big)\\
\displaystyle =&o\Big(\epsilon^{\frac{N-2}{N-4}}\Big),
\qquad 0<\theta_{\mathrm{st}}<2.
\end{split}
\end{equation*}

\medskip
Finally, we estimate
 \begin{equation*}
\begin{split}
\displaystyle I_5&=-\epsilon^{\frac{N-2}{N-4}} \int_{\Omega_\epsilon}u_{1,\epsilon}\bar Z_{1}^0
=-\epsilon^{\frac{N-2}{N-4}} \int_{\Omega_\epsilon}(V_1+\varphi)\bar Z_{1}^0=-\epsilon^{\frac{N-2}{N-4}} \int_{\Omega_\epsilon}V_1\bar Z_{1}^0-\epsilon^{\frac{N-2}{N-4}} \int_{\Omega_\epsilon}\varphi \bar Z_{1}^0\\
\displaystyle &=-\epsilon^{\frac{N-2}{N-4}} \int_{\Omega_\epsilon}\epsilon^{\frac{N-2}{2(N-4)}}P U_{\epsilon^{\frac{1}{N-4}}\lambda_1,\xi_1}(\epsilon^{\frac{1}{N-4}}y)\epsilon^{\frac{N}{2(N-4)}} Y^0_{\epsilon^{\frac{1}{N-4}}\lambda_1,\xi_1}(\epsilon^{\frac{1}{N-4}}y)dy-\epsilon^{\frac{N-2}{N-4}} \int_{\Omega_\epsilon}\varphi \bar Z_{1}^0\\
\displaystyle &=-\epsilon^{\frac{N-3}{N-4}}\int_\Omega P U_{\epsilon^{\frac{1}{N-4}}\lambda_1,\xi_1}(x) Y^0_{\epsilon^{\frac{1}{N-4}}\lambda_1,\xi_1}(x)dx-\epsilon^{\frac{N-2}{N-4}} \int_{\Omega_\epsilon}\varphi\bar  Z_{1}^0\\
\displaystyle &=-\epsilon^{\frac{N-3}{N-4}}\int_\Omega U_{\epsilon^{\frac{1}{N-4}}\lambda_1,\xi_1}(x) Y^0_{\epsilon^{\frac{1}{N-4}}\lambda_1,\xi_1}(x)dx+o(\epsilon^{\frac{N-2}{N-4}})-\epsilon^{\frac{N-2}{N-4}} \int_{\Omega_\epsilon}\varphi \bar Z_{1}^0\\
\displaystyle &=-\lambda_1\epsilon^{\frac{N-3}{N-4}}\epsilon^{\frac{1}{N-4}}\int_{\mathbb R^N} U^2_{1,0}(x)dx+o(\epsilon^{\frac{N-2}{N-4}})-\epsilon^{\frac{N-2}{N-4}} \int_{\Omega_\epsilon}\varphi\bar  Z_{1}^0\\
\displaystyle &=-\lambda_1\epsilon^{\frac{N-2}{N-4}}\int_{\mathbb R^N} U^2_{1,0}(x)dx+o(\epsilon^{\frac{N-2}{N-4}})-\epsilon^{\frac{N-2}{N-4}} \int_{\Omega_\epsilon}\varphi\bar  Z_{1}^0
\end{split}
\end{equation*}

Moreover,
 \begin{equation*}
\begin{split}
\displaystyle &\Big|-\epsilon^{\frac{N-2}{N-4}} \int_{\Omega_\epsilon}\varphi \bar Z_{1}^0\Big|\\
\displaystyle &{\le C\epsilon^{\frac{N-2}{N-4}}\int_{\Omega_\epsilon\cap\{|y-\tilde\xi_1|\le2\epsilon^{-\alpha_{\mathrm{cut}}}\}}
\left(\frac{\epsilon^{\alpha_1}}{(1+|y-\tilde\xi_1|)^{\theta_1}}+\frac{\epsilon^{\tilde\alpha_1}}{(1+|y-\tilde\xi_2|)^{\tilde\theta_1}}\right)
\frac{dy}{(1+|y-\tilde\xi_1|)^{N-2}}}\\
&{=o\Big(\epsilon^{\frac{N-2}{N-4}}\Big).}
\end{split}
\end{equation*}

\medskip

This proves \eqref{eq:scale-expansion-first}; the same argument gives
\eqref{eq:scale-expansion-second}.

\end{proof}

\medskip
For $\mathbf t=(\lambda_1,\xi_1,\lambda_2,\xi_2)\in O_\delta$, set
\[
d_{\mathrm{sep}}(\mathbf t):=\min\left\{
\operatorname{dist}(\xi_1,\partial\Omega),
\operatorname{dist}(\xi_2,\partial\Omega),
\frac{1}{2}|\xi_1-\xi_2|\right\}.
\]
\begin{Lem}\label{lem:translation-expansion}
Let $K\Subset O_\delta$ and fix
$0<\rho<\inf_{\mathbf t\in K}d_{\mathrm{sep}}(\mathbf t)$. For $j=1,\ldots,N$, uniformly for $\mathbf t\in K$,
\begin{equation}\label{eq510}
\begin{split}
&\int_{B(\tilde\xi_1,\epsilon^{-\frac{1}{N-4}}\rho)}
\Big(-\Delta u_{1,\epsilon}-\epsilon^{\frac{N-2}{N-4}}u_{1,\epsilon}
-|u_{1,\epsilon}|^{2^*-2}u_{1,\epsilon}\\
&\hspace{40mm}
-\beta |u_{2,\epsilon}|^p|u_{1,\epsilon}|^{p-2}u_{1,\epsilon}\Big)
\partial_j u_{1,\epsilon}\\
&=-\frac{1}{2} C_NA_0\lambda_1^{N-2}
\epsilon^{\frac{N-1}{N-4}}\partial_j\varphi_\Omega(\xi_1)
+o\left(\epsilon^{\frac{N-1}{N-4}}\right),
\end{split}
\end{equation}
and
\begin{equation}\label{eq511}
\begin{split}
&\int_{B(\tilde\xi_2,\epsilon^{-\frac{1}{N-4}}\rho)}
\Big(-\Delta u_{2,\epsilon}-\epsilon^{\frac{N-2}{N-4}}u_{2,\epsilon}
-|u_{2,\epsilon}|^{2^*-2}u_{2,\epsilon}\\
&\hspace{40mm}
-\beta |u_{1,\epsilon}|^p|u_{2,\epsilon}|^{p-2}u_{2,\epsilon}\Big)
\partial_j u_{2,\epsilon}\\
&=-\frac{1}{2} C_NA_0\lambda_2^{N-2}
\epsilon^{\frac{N-1}{N-4}}\partial_j\varphi_\Omega(\xi_2)
+o\left(\epsilon^{\frac{N-1}{N-4}}\right),
\end{split}
\end{equation}
where
\[
A_0=\int_{\mathbb R^N}U_{1,0}^{2^*-1}
=(N-2)C_N\omega_{N-1}.
\]
\end{Lem}

\begin{proof}
We prove \eqref{eq510}. Put
\[
B_{\rho,\epsilon}:=B(\tilde\xi_1,\rho/s_\epsilon).
\]
Integration by parts gives the local Pohozaev identity
\begin{equation}\label{eq:15}
\begin{split}
&\int_{B_{\rho,\epsilon}}(-\Delta u_{1,\epsilon}-s_\epsilon^{N-2}u_{1,\epsilon}
-|u_{1,\epsilon}|^{2^*-2}u_{1,\epsilon})\partial_j u_{1,\epsilon}\\
&=\int_{\partial B_{\rho,\epsilon}}\left(
-\partial_\nu u_{1,\epsilon}\,\partial_j u_{1,\epsilon}
+\frac{1}{2}|\nabla u_{1,\epsilon}|^2\nu_j
-\frac{s_\epsilon^{N-2}}{2}u_{1,\epsilon}^2\nu_j
-\frac{1}{2^*}|u_{1,\epsilon}|^{2^*}\nu_j\right).
\end{split}
\end{equation}
 On this boundary, the scaling identities and Lemma~\ref{lemRj0} yield
\begin{equation}\label{eq:17}
\begin{split}
V_1(y)&=C_N\lambda_1^{\frac{N-2}{2}}s_\epsilon^{N-2}
G_\Omega(s_\epsilon y,\xi_1)+o(s_\epsilon^{N-2}),\\
\nabla_yV_1(y)&=C_N\lambda_1^{\frac{N-2}{2}}s_\epsilon^{N-1}
\nabla_xG_\Omega(s_\epsilon y,\xi_1)+o(s_\epsilon^{N-1}),
\end{split}
\end{equation}
uniformly in $C^1$ for fixed $\rho$.

 We next justify that the correction does not contribute to the leading
boundary flux.  On a fixed rescaled annulus
\[
\frac{\rho}{2s_\epsilon}\le |y-\tilde\xi_1|\le\frac{3\rho}{2s_\epsilon},
\]
 the two distances $1+|y-\tilde\xi_i|$ are comparable to $s_\epsilon^{-1}$.  Hence
Lemma~\ref{lem8} gives
\[
|\varphi(y)|\le Cs_\epsilon^{N-2+\sigma}
\]
 for some $\sigma>0$ independent of the parameters in compact subsets of the
admissible set.  To estimate the gradient, apply the interior estimate on a
ball of radius $c/s_\epsilon$ contained in this annulus.  The equation for $\varphi$
and Lemmas~\ref{lem7} and~\ref{lem8} give there
\[
|\Delta\varphi|
\le C\big(s_\epsilon^{N-2}|\varphi|+s_\epsilon^{2N-4}+s_\epsilon^{N+2}\big).
\]
After decreasing $\sigma$ if necessary, the scaled interior estimate yields
\[
|\nabla\varphi(y)|\le Cs_\epsilon^{N-1+\sigma}.
\]
Since $|\partial B_{\rho,\epsilon}|=O(s_\epsilon^{-(N-1)})$, it follows that
\begin{equation}\label{eq:16}
\int_{\partial B_{\rho,\epsilon}}\big(|\nabla\varphi|^2+|\varphi|^{2^*}\big)
=o(s_\epsilon^{N-1}).
\end{equation}
 Together with \eqref{eq:17}, this shows that all boundary terms containing
$\varphi$ are $o(s_\epsilon^{N-1})$.  The critical boundary term containing only
$V_1$ is $O(s_\epsilon^{N+1})$, and the corresponding mass term is $O(s_\epsilon^{2N-5})$;
both are $o(s_\epsilon^{N-1})$.
For the unnormalized Green function used here,
\begin{equation}\label{eq:green-flux}
\int_{\partial B_\rho(\xi)}
\left(-\partial_\nu G_\Omega\,\partial_jG_\Omega
+\frac{1}{2}|\nabla G_\Omega|^2\nu_j\right)
=-\frac{(N-2)\omega_{N-1}}{2}\partial_j\varphi_\Omega(\xi).
\end{equation}
 Indeed, write $G_\Omega(x,\xi)=|x-\xi|^{2-N}-H_\Omega(x,\xi)$ and let the
radius tend to zero; the pure singular term has zero integral, while the
cross term gives $-(N-2)\omega_{N-1}\partial_{x_j}H_\Omega(\xi,\xi)$.
Symmetry of $H_\Omega$ gives
$2\partial_{x_j}H_\Omega(\xi,\xi)=\partial_j\varphi_\Omega(\xi)$.
Combining \eqref{eq:15}--\eqref{eq:green-flux} therefore gives
\[
\int_{B_{\rho,\epsilon}}(-\Delta u_{1,\epsilon}-s_\epsilon^{N-2}u_{1,\epsilon}
-|u_{1,\epsilon}|^{2^*-2}u_{1,\epsilon})\partial_j u_{1,\epsilon}
=-\frac{1}{2}C_NA_0\lambda_1^{N-2}s_\epsilon^{N-1}
\partial_j\varphi_\Omega(\xi_1)+o(s_\epsilon^{N-1}).
\]
Finally, Lemma~\ref{lem7}, the gradient estimate in Lemma~\ref{lem8}, and a
split into neighborhoods of the two centers give
\[
\int_{B_{\rho,\epsilon}}|u_{2,\epsilon}|^p|u_{1,\epsilon}|^{p-1}
|\partial_j u_{1,\epsilon}|
\le Cs_\epsilon^N\big(1+|\log s_\epsilon|\big)=o(s_\epsilon^{N-1}).
\]
This proves \eqref{eq510}. The proof of \eqref{eq511} is identical.
\end{proof}

\medskip

\begin{proof}[Proof of Theorems~\ref{thm:rescaled-main} and~\ref{thm:main}]
Assume that the Robin function \(\varphi_\Omega\) has two distinct
critical points \(\bar\xi_1,\bar\xi_2\in\Omega\), each satisfying the local
degree condition in Definition~\ref{defstable}.  Set
\[
\gamma_{\mathrm{tr}}:=\frac{N-1}{N-4}.
\]
Recall that $\gamma_1=(N-2)/(N-4)$, and denote by \(M_i^0\) the left-hand sides of
\eqref{eq:scale-condition-first} and \eqref{eq:scale-condition-second}, and by
\(M_i^j\) the left-hand sides of \eqref{eq:translation-condition-first} and
\eqref{eq:translation-condition-second}. On each compact parameter domain used
below, we fix $\rho>0$ below the uniform lower bound in
Lemma~\ref{lem:translation-expansion}. We introduce the normalized reduced map
 \[
\mathcal F_\epsilon(\lambda_1,\xi_1,\lambda_2,\xi_2)
:=
\begin{pmatrix}
\epsilon^{-\gamma_1}\lambda_1^{-1}M_1^0\\[1mm]
(-\frac{1}{2} C_NA_0\epsilon^{\gamma_{\mathrm{tr}}}\lambda_1^{N-2})^{-1}(M_1^1,\ldots,M_1^N)\\[1mm]
\epsilon^{-\gamma_1}\lambda_2^{-1}M_2^0\\[1mm]
(-\frac{1}{2} C_NA_0\epsilon^{\gamma_{\mathrm{tr}}}\lambda_2^{N-2})^{-1}(M_2^1,\ldots,M_2^N)
\end{pmatrix}.
\]
By Proposition~\ref{prop:parameter-continuity}, the correction and its
first spatial derivatives depend continuously on the parameters for each
fixed sufficiently small $\epsilon$. Hence all components of
$\mathcal F_\epsilon$ are continuous. By
Lemmas~\ref{lem:scale-expansion} and~\ref{lem:translation-expansion},
uniformly on compact subsets of the admissible parameter set, we have the
asymptotic expansion
\[
\mathcal F_\epsilon=\mathcal F_0+o(1),
\]
where the limiting map $\mathcal F_0$ is explicitly given by
\[
\mathcal F_0(\lambda_1,\xi_1,\lambda_2,\xi_2)
:=
\begin{pmatrix}
A\lambda_1^{N-4}\varphi_\Omega(\xi_1) - B \\[1mm]
\nabla\varphi_\Omega(\xi_1) \\[1mm]
A\lambda_2^{N-4}\varphi_\Omega(\xi_2) - B \\[1mm]
\nabla\varphi_\Omega(\xi_2)
\end{pmatrix},
\]
and the constants are defined as
\[
A:=\frac{N-2}{2}C_N\int_{\mathbb R^N}U_{1,0}^{2^*-1},\qquad
B:=\int_{\mathbb R^N}U_{1,0}^2.
\]
Let
\[
\lambda_i^0=\left(\frac{B}{A\varphi_\Omega(\bar\xi_i)}\right)^{1/(N-4)},
\qquad i=1,2.
\]
Choose $0<\delta<1$ and disjoint neighborhoods
$\mathcal U_i\Subset\Omega$ of $\bar\xi_i$ so small that points in their
closures remain separated from each other and from $\partial\Omega$. The
scale intervals below are also chosen so that the resulting parameter
domains have compact closure in $O_\delta$. By the local degree condition,
choose bounded open sets
$D_i$ with
\[
\bar\xi_i\in D_i\Subset\mathcal U_i,\qquad
0\notin\nabla\varphi_\Omega(\partial D_i),
\qquad
\deg(\nabla\varphi_\Omega,D_i,0)\ne0.
\]
For $\xi\in\overline{D_i}$ set
\[
\lambda_i^*(\xi)=
\left(\frac{B}{A\varphi_\Omega(\xi)}\right)^{1/(N-4)}.
\]
Choose a bounded open interval $J_i\Subset(0,+\infty)$ whose interior
contains $\lambda_i^*(\overline{D_i})$, and put
$\mathcal D=J_1\times D_1\times J_2\times D_2$. Then
$0\notin\mathcal F_0(\partial\mathcal D)$. To compute the degree, homotope the
$i$-th scale component through
\[
A\lambda_i^{N-4}
\big((1-t)\varphi_\Omega(\xi_i)+t\varphi_\Omega(\bar\xi_i)\big)-B,
\qquad 0\le t\le1,
\]
while leaving the translation component $\nabla\varphi_\Omega(\xi_i)$
unchanged. If the translation component vanishes, the positive scale root
along the homotopy lies between $\lambda_i^*(\xi_i)$ and
$\lambda_i^*(\bar\xi_i)=\lambda_i^0$, both of which lie in
$\operatorname{int}J_i$. Thus no zero reaches the boundary during the
homotopy. Hence
\[
\begin{split}
\deg(\mathcal F_0,\mathcal D,0)
={}&\deg(F_1,J_1,0)\deg(\nabla\varphi_\Omega,D_1,0)\\
&\times\deg(F_2,J_2,0)\deg(\nabla\varphi_\Omega,D_2,0)\ne0,
\end{split}
\]
where $F_i(\lambda)=A\lambda^{N-4}\varphi_\Omega(\bar\xi_i)-B$ and
$\deg(F_i,J_i,0)=1$.
The uniform convergence $\mathcal F_\epsilon\to\mathcal F_0$ implies
$0\notin\mathcal F_\epsilon(\partial\mathcal D)$ for small $\epsilon$, and
therefore
\[
\deg(\mathcal F_\epsilon,\mathcal D,0)
=
\deg(\mathcal F_0,\mathcal D,0)\ne0.
\]
Thus $\mathcal F_\epsilon$ has a zero in $\mathcal D$.

To obtain convergence to the prescribed points, for each $k$ choose disjoint
neighborhoods $\mathcal U_{1,k}$ and $\mathcal U_{2,k}$ with
$\mathcal U_{i,k}\Subset B_{1/k}(\bar\xi_i)$, together with degree domains
$D_{i,k}\Subset\mathcal U_{i,k}$. Choose bounded open intervals $J_{i,k}$
containing $\lambda_i^*(\overline{D_{i,k}})$ and satisfying
$\operatorname{diam}(J_{i,k})\to0$. Then choose a decreasing sequence
$\epsilon_k\downarrow0$ so that the degree construction is valid on
$J_{1,k}\times D_{1,k}\times J_{2,k}\times D_{2,k}$ whenever
$0<\epsilon<\epsilon_k$. For
$\epsilon_{k+1}\le\epsilon<\epsilon_k$, select a zero in this parameter
domain. It follows that
$\xi_{i,\epsilon}\to\bar\xi_i$ and
$\lambda_{i,\epsilon}\to\lambda_i^0$. By
Lemma~\ref{lem:multiplier-vanishing}, all Lagrange multipliers \(c_j,d_j\) vanish, so the
projected solution obtained in Section~4 is a genuine solution of
\eqref{eq1}. Rescaling gives the solution of the original problem
\eqref{eq0}.

\medskip
\noindent\textbf{Nonnegativity.}
At this point all multipliers vanish, so the constructed pair satisfies the
original system weakly in $\Omega_\epsilon$. We first note that the inner data
are nonnegative. Indeed, uniformly on $P_1$ and $Q_1$,
\[
V_i=U_i+O\!\left(\epsilon^{\frac{N-2}{N-4}}\right),
\qquad
U_i\ge c r_\epsilon^{2-N},
\qquad i=1,2.
\]
Moreover,
\[
\epsilon^{\frac{N-2}{N-4}}=o(r_\epsilon^{2-N}),
\qquad
\epsilon^{\gamma_0}=o(r_\epsilon^{2-N})
\]
by Lemma~\ref{lem:exponent-choice}. Since the fixed-point traces are bounded
by $\epsilon^{\gamma_0}$, it follows that
\[
u_{1,\epsilon}\ge0\quad\text{in }P_1,
\qquad
u_{2,\epsilon}\ge0\quad\text{in }Q_1
\]
for sufficiently small $\epsilon$.

For brevity, write $u_i=u_{i,\epsilon}$ and set
$u_i^-:=\max\{-u_i,0\}$. Then
$u_1^-\in H_0^1(\Omega_\epsilon\setminus P_1)$ and
$u_2^-\in H_0^1(\Omega_\epsilon\setminus Q_1)$ after extension by zero.
Testing the first equation with $-u_1^-$ gives
\[
\begin{aligned}
\int_{\Omega_\epsilon}|\nabla u_1^-|^2
={}&
\epsilon^{\frac{N-2}{N-4}}
\int_{\Omega_\epsilon}(u_1^-)^2
+
\int_{\Omega_\epsilon}(u_1^-)^{2^*}\\
&+\beta\int_{\Omega_\epsilon}|u_2|^p(u_1^-)^p.
\end{aligned}
\]
The last term is nonpositive because $\beta<0$, and therefore
\[
\int_{\Omega_\epsilon}|\nabla u_1^-|^2
\le
\epsilon^{\frac{N-2}{N-4}}
\int_{\Omega_\epsilon}(u_1^-)^2
+
\int_{\Omega_\epsilon}(u_1^-)^{2^*}.
\]
Sobolev's and H\"older's inequalities give
\[
\begin{aligned}
\mathbf C_s\|u_1^-\|_{L^{2^*}}^2
&\le
\left(
\epsilon^{\frac{N-2}{N-4}}|\Omega_\epsilon|^{2/N}
+
\|u_1\|_{L^{2^*}(\Omega_\epsilon\setminus P_1)}^{2^*-2}
\right)
\|u_1^-\|_{L^{2^*}}^2.
\end{aligned}
\]
The first coefficient is $O(\epsilon)$. By Lemma~\ref{lem7},
\[
\int_{\Omega_\epsilon\setminus P_1}|u_1|^{2^*}
\le C\int_{r_\epsilon}^{\infty}r^{-2N}r^{N-1}\,dr
\le Cr_\epsilon^{-N},
\]
and hence
\[
\|u_1\|_{L^{2^*}(\Omega_\epsilon\setminus P_1)}^{2^*-2}
\le Cr_\epsilon^{-2}=o(1).
\]
Thus the coefficient on the right is strictly smaller than
$\mathbf C_s$ for sufficiently small $\epsilon$, and $u_1^-=0$. The same
argument gives $u_2^-=0$. Therefore
\[
u_{1,\epsilon},u_{2,\epsilon}\ge0
\quad\text{in }\Omega_\epsilon.
\]
The argument above gives a nonnegative solution of the rescaled problem. It
remains to justify the stronger estimate \eqref{eq:strong-correction}. Let
$\varrho_0$ be as in Lemma~\ref{deadout}. After all the auxiliary exponents
have been fixed, choose
$0<\varrho<\min\{\varrho_\tau,\varrho_0\}$; it will be decreased once more
below if necessary. Set
\[
\gamma_{\mathrm{bub}}:=\frac{N-2}{2(N-4)}.
\]
We discuss $\phi_\epsilon$; the argument for $\psi_\epsilon$ is identical.
The fixed point has the form $S(b)$ with
$b\in\Lambda_\epsilon\cap H_0^1(\Omega_\epsilon)^2$. Hence, by
\eqref{eq:S-preserves-inner-data},
$|\phi_\epsilon|\le\epsilon^{\gamma_0}$ on $P_1$. On
\[
\Omega_\epsilon\setminus
\left(P_1\cup
B_{R_{\mathrm{dc},\epsilon}(\varrho)}(\tilde\xi_2)\right),
\]
Lemma~\ref{lem8} gives
\[
\begin{aligned}
|\phi_\epsilon|
&\le C\epsilon^{\alpha_1}r_\epsilon^{-\theta_1}
+C\epsilon^{\tilde\alpha_1}
R_{\mathrm{dc},\epsilon}(\varrho)^{-\tilde\theta_1}\\
&=C\epsilon^{\kappa_0}
+C\epsilon^{\tilde\alpha_1+
\frac{1-\varrho}{N-2}\tilde\theta_1}.
\end{aligned}
\]
Finally, Lemma~\ref{deadout} gives $u_{1,\epsilon}=0$ in
$B_{R_{\mathrm{dc},\epsilon}(\varrho)}(\tilde\xi_2)$, so there
$\phi_\epsilon=-V_1$. Since
$R_{\mathrm{dc},\epsilon}(\varrho)=o(|\tilde\xi_1-\tilde\xi_2|)$ and
$0\le V_1\le U_1$,
\[
|\phi_\epsilon|\le C|\tilde\xi_1-\tilde\xi_2|^{2-N}
\le C\epsilon^{\frac{N-2}{N-4}}.
\]
By the definitions of $\kappa_0$ and $\gamma_0$,
\[
\gamma_0>\kappa_0>\gamma_{\mathrm{bub}}.
\]
Moreover, the losses were fixed sufficiently small in
Lemma~\ref{lem:exponent-choice} so that
\[
\tilde\alpha_1+\frac{\tilde\theta_1}{N-2}
>\gamma_{\mathrm{bub}},
\qquad
\frac{N-2}{N-4}=2\gamma_{\mathrm{bub}}.
\]
Decrease the already fixed $\varrho$, if necessary, so that
\[
\tilde\alpha_1+
\frac{1-\varrho}{N-2}\tilde\theta_1>\gamma_{\mathrm{bub}}.
\]
Taking $\sigma$ smaller than the minimum of the resulting four positive gaps
proves \eqref{eq:strong-correction}.

Under the rescaling in \eqref{eq1},
\[
\bar\phi_\epsilon(x)=
\epsilon^{-\gamma_{\mathrm{bub}}}\phi_\epsilon
\left(\epsilon^{-\frac1{N-4}}x\right),
\qquad
\bar\psi_\epsilon(x)=
\epsilon^{-\gamma_{\mathrm{bub}}}\psi_\epsilon
\left(\epsilon^{-\frac1{N-4}}x\right).
\]
Therefore
$\|\bar\phi_\epsilon\|_\infty+
\|\bar\psi_\epsilon\|_\infty\le C\epsilon^\sigma\to0$,
which proves the assertion in the original variables.

This completes the proof of Theorems~\ref{thm:rescaled-main} and~\ref{thm:main}.
\end{proof}

\medskip
\begin{proof}[Proof of Theorem~\ref{thm:dead-core}]
The solution constructed in Theorem~\ref{thm:rescaled-main} is obtained
through the same outer minimization operator \(S\). Hence
Lemma~\ref{deadout} applies to the outer minimizer associated with the
fixed-point inner trace and gives the two dead-core inclusions stated in
Theorem~\ref{thm:dead-core}.
\end{proof}

\appendix
\section{Auxiliary tools}\label{app:auxiliary-tools}

\begin{proof}[Proof of Lemma~\ref{lemdead}]
Set
\[
q_{\mathrm{dc}}:=\frac{2}{N-2}\in(0,1),\qquad
m_{\mathrm{dc}}:=\frac{2}{1-q_{\mathrm{dc}}}
=\frac{2(N-2)}{N-4}.
\]
On the cone of nonnegative functions in $C([1,3/2])$, consider the
order-preserving operator
\[
(\mathcal Kv)(r)
=b\int_1^r t^{1-N}\int_1^t s^{N-1}v(s)^{q_{\mathrm{dc}}}\,ds\,dt.
\]
Since $m_{\mathrm{dc}}q_{\mathrm{dc}}+2=m_{\mathrm{dc}}$ and $1\le s,t\le 3/2$, there exist constants
$0<C_-\le C_+$, depending only on $N$ and $b$, such that
\[
C_-A^{q_{\mathrm{dc}}}(r-1)^{m_{\mathrm{dc}}}
\le \mathcal K\big(A(\,\cdot\,-1)^{m_{\mathrm{dc}}}\big)(r)
\le C_+A^{q_{\mathrm{dc}}}(r-1)^{m_{\mathrm{dc}}},
\qquad 1\le r\le 3/2.
\]
Choose $A_->0$ sufficiently small and $A_+>A_-$ sufficiently large so that
\[
\underline w(r):=A_-(r-1)^{m_{\mathrm{dc}}}\le\mathcal K\underline w(r),
\qquad
\mathcal K\overline w(r)\le\overline w(r):=A_+(r-1)^{m_{\mathrm{dc}}}.
\]
Define $w_0=\underline w$ and $w_{\ell+1}=\mathcal Kw_\ell$ for $\ell\ge0$. Monotonicity gives
\[
\underline w\le w_\ell\le w_{\ell+1}\le\overline w
\qquad\text{on }[1,3/2].
\]
Hence $w_\ell$ converges pointwise to a function $w$ between the two barriers.
Dominated convergence in the integral formula yields $w=\mathcal Kw$. In
particular,
\[
A_-(r-1)^{m_{\mathrm{dc}}}\le w(r)\le A_+(r-1)^{m_{\mathrm{dc}}},
\qquad 1\le r\le 3/2,
\]
so $w(r)>0$ for $r>1$. Differentiating the integral equation gives
\[
w'(r)=br^{1-N}\int_1^r s^{N-1}w(s)^{q_{\mathrm{dc}}}\,ds>0
\qquad (r>1).
\]
Extend $w$ by zero on $[0,1]$. Since $m_{\mathrm{dc}}>1$, the preceding bounds and the
integral formula give $w(1)=w'(1)=0$. Thus the radial extension belongs to
$C^1(B_{3/2}(0))$ and to $H^1(B_{3/2}(0))$, and it satisfies
$-\Delta w+bw^{q_{\mathrm{dc}}}=0$ weakly across the interface $r=1$. Set $c_b:=w(3/2)>0$.
Then $w-c_b\in H_0^1(B_{3/2}(0))$ in the trace sense.

Let $w_1,w_2$ be two nonnegative weak solutions with boundary value $c_b$.
Because $w_1-w_2\in H_0^1(B_{3/2}(0))$, testing the difference of the equations with
$w_1-w_2$ yields
\[
\int_{B_{3/2}(0)}|\nabla(w_1-w_2)|^2
+b\int_{B_{3/2}(0)}(w_1^{q_{\mathrm{dc}}}-w_2^{q_{\mathrm{dc}}})(w_1-w_2)=0.
\]
Both terms are nonnegative, and hence $w_1=w_2$. Rotational invariance and
uniqueness imply that the solution is radial.
\end{proof}

\medskip

The following elementary product estimate is taken from \cite{WY10}.
\begin{Lem}\label{lemB1}
Let $a,b>0$ and let $x_0,y_0\in\mathbb R^N$ be distinct. For every
$0<\sigma\le\min\{a,b\}$, there exists a constant $C>0$ such that
\[
\frac{1}{(1+|z-x_0|)^a(1+|z-y_0|)^b}
\le \frac{C}{|x_0-y_0|^\sigma}
\left(
\frac{1}{(1+|z-x_0|)^{a+b-\sigma}}
+\frac{1}{(1+|z-y_0|)^{a+b-\sigma}}
\right)
\]
for every $z\in\mathbb R^N$.
\end{Lem}
\medskip

An argument analogous to that used in the appendix of \cite{WY10} gives the
following estimate.
\begin{Lem}\label{lemB2}
For every $\sigma>0$ with $\sigma\ne N-2$, there is a constant $C>0$,
independent of $\epsilon$, such that
\[
\int_{\Omega_\epsilon}\frac{1}{|y-z|^{N-2}}
\frac{dz}{(1+|z|)^{2+\sigma}}
\le \frac{C}{(1+|y|)^{\min\{\sigma,N-2\}}}
\]
for every $y\in\mathbb R^N$.
\end{Lem}

\medskip
We next record the standard projection expansions used in
Section~\ref{sec:5}; see \cite{Musso02} for their derivation. Put
$\widehat\lambda_\epsilon:=s_\epsilon\lambda=\epsilon^{1/(N-4)}\lambda$.

\begin{Lem}\label{lemRj0}
As $\epsilon\to0$,
\begin{align*}
PU_{\widehat\lambda_\epsilon,\xi}(x)
&=U_{\widehat\lambda_\epsilon,\xi}(x)
-C_N\widehat\lambda_\epsilon^{\frac{N-2}{2}}H(x,\xi)
+o\left(\widehat\lambda_\epsilon^{\frac{N-2}{2}}\right),\\
PU_{\widehat\lambda_\epsilon,\xi}(x)
&=C_N\widehat\lambda_\epsilon^{\frac{N-2}{2}}G(x,\xi)
+o\left(\widehat\lambda_\epsilon^{\frac{N-2}{2}}\right).
\end{align*}
The expansions are uniform for $\lambda$ in compact subsets of
$(0,+\infty)$, for $\xi$ in compact subsets of $\Omega$, and for $x$ in
compact subsets of $\Omega\setminus\{\xi\}$.
\end{Lem}

A direct computation gives
\begin{align*}
Y_{\lambda,y}^{0}(x)
&=C_N\frac{N-2}{2}\lambda^{\frac{N-4}{2}}
\frac{|x-y|^2-\lambda^2}{(\lambda^2+|x-y|^2)^{N/2}},\\
Y_{\lambda,y}^{j}(x)
&=C_N(N-2)\lambda^{\frac{N-2}{2}}
\frac{x_j-y_j}{(\lambda^2+|x-y|^2)^{N/2}},
\qquad j=1,\ldots,N.
\end{align*}

\begin{Lem}\label{lem:projected-kernel-expansions}
Under the same uniformity assumptions as in Lemma~\ref{lemRj0},
\begin{align*}
P Y_{\widehat\lambda_\epsilon,\xi}^{j}(x)
&=C_N\widehat\lambda_\epsilon^{\frac{N-2}{2}}
\partial_{\xi_j}G(x,\xi)
+o\left(\widehat\lambda_\epsilon^{\frac{N-2}{2}}\right),
&&j=1,\ldots,N,\\
P Y_{\widehat\lambda_\epsilon,\xi}^{0}(x)
&=C_N\frac{N-2}{2}\widehat\lambda_\epsilon^{\frac{N-4}{2}}G(x,\xi)
+o\left(\widehat\lambda_\epsilon^{\frac{N-4}{2}}\right).
\end{align*}
\end{Lem}

Set
\[
D^j_{\widehat\lambda_\epsilon,\xi}:=Y^j_{\widehat\lambda_\epsilon,\xi}
-PY^j_{\widehat\lambda_\epsilon,\xi},
\qquad j=0,1,\ldots,N.
\]

\begin{Lem}\label{lemD}
Under the same uniformity assumptions,
\begin{align*}
D^j_{\widehat\lambda_\epsilon,\xi}(x)
&=C_N\widehat\lambda_\epsilon^{\frac{N-2}{2}}
\partial_{\xi_j}H(x,\xi)
+o\left(\widehat\lambda_\epsilon^{\frac{N-2}{2}}\right),
&&j=1,\ldots,N,\\
D^0_{\widehat\lambda_\epsilon,\xi}(x)
&=C_N\frac{N-2}{2}\widehat\lambda_\epsilon^{\frac{N-4}{2}}H(x,\xi)
+o\left(\widehat\lambda_\epsilon^{\frac{N-4}{2}}\right).
\end{align*}
\end{Lem}

\subsection*{Simultaneous choice of the auxiliary exponents}

The following lemma verifies that the losses in
\eqref{eq:auxiliary-exponents} can be chosen simultaneously.

\begin{Lem}\label{lem:exponent-choice}
Let $N\ge5$ and $\tau\in\mathcal I_N$. The numbers
$\vartheta,\theta',\tilde\theta,\bar\theta$ and $\delta_*$ can be chosen as in
\eqref{eq:auxiliary-exponents} so that, for some $\sigma_0>0$,
\begin{equation}\label{eq:basic-exponent-compatibility}
\gamma_0<\gamma_1<2\gamma_0,
\qquad
\gamma_0>\frac{N-2}{B_\tau},
\qquad
(N-4)\tilde\theta+\bar\theta<2,
\end{equation}
\begin{align}
\alpha_i-\tilde\alpha_i
&<\frac{2\tilde\theta_i+2-\theta_i}{B_\tau},
&
\alpha_i&<\frac{N-\theta_i}{N-4},
\qquad i=1,2,\label{eq:all-barrier-margins}\\
\alpha_i-\frac{N-2-\theta_i}{N-2}&>0,
&
\tilde\alpha_i+\frac{\tilde\theta_i}{N-4}&>1,
\qquad i=1,2,\label{eq:all-deadcore-margins}
\end{align}
and
\begin{equation}\label{eq:self-map-exponent-margins}
\min\left\{
\alpha_1+\frac{2\theta_1-2}{B_\tau},
\ \tilde\alpha_1+\frac{\tilde\theta_1}{N-4}-\frac2{B_\tau},
\ \gamma_1-\frac4{B_\tau},
\ m_N\kappa_0-\frac4{B_\tau}
\right\}
\ge\gamma_0+\sigma_0.
\end{equation}
Moreover,
\begin{equation}\label{eq:contraction-exponent-margins}
\kappa_0-\frac4{B_\tau}>0,
\qquad
(2^*-2)\kappa_0-\frac4{B_\tau}>0,
\end{equation}
and
\begin{equation}\label{eq:final-correction-exponent-margin}
\tilde\alpha_1+\frac{\tilde\theta_1}{N-2}
>\frac{N-2}{2(N-4)}.
\end{equation}
The upper bound on $\tau$ in \eqref{tau-range} is used only in the first
inequality of \eqref{eq:all-barrier-margins}.
\end{Lem}

The two inequalities in \eqref{eq:all-barrier-margins} are used in
Lemma~\ref{lem8}; \eqref{eq:all-deadcore-margins} is used in
Lemma~\ref{deadout}; \eqref{eq:self-map-exponent-margins} and
\eqref{eq:contraction-exponent-margins} are used in Lemma~\ref{lem13}; and
\eqref{eq:final-correction-exponent-margin} is used in the proof of
Theorem~\ref{thm:rescaled-main}.

\begin{proof}
First set the four losses and $\delta_*$ equal to zero. Then, for
$i=1,2$,
\[
\alpha_i=\frac{N-2}{2(N-4)},\qquad
\tilde\alpha_i=\frac2{N-4},\qquad
\theta_i=\frac N2,\qquad
\tilde\theta_i=N-2,
\]
and
\[
\kappa_0=\frac{N-2}{2(N-4)}+\frac{N}{2B_\tau}.
\]
The three nontrivial gaps in \eqref{eq:basic-exponent-compatibility} are
\[
\gamma_1-\kappa_0
=\frac{(N-2)\tau+N-4}{2(N-4)(2+\tau)},\qquad
2\kappa_0-\gamma_1
=\frac{N}{(N-4)(2+\tau)},
\]
\[
\kappa_0-\frac{N-2}{B_\tau}
=\frac{N+\tau(N-2)}{2(N-4)(2+\tau)},
\]
and hence are positive.

After subtracting $\kappa_0$ from the four entries in the minimum in
\eqref{eq:self-map-exponent-margins}, their zero-loss gaps are
\begin{align*}
\alpha_1+\frac{2\theta_1-2}{B_\tau}-\kappa_0
&=\frac1{2(2+\tau)},\\
\tilde\alpha_1+\frac{\tilde\theta_1}{N-4}
-\frac2{B_\tau}-\kappa_0
&=\frac{N\tau+N+2\tau}{2(N-4)(2+\tau)},\\
\gamma_1-\frac4{B_\tau}-\kappa_0
&=\frac{(N-2)\tau+N-12}{2(N-4)(2+\tau)},\\
 m_N\kappa_0-\frac4{B_\tau}-\kappa_0
&=
\begin{cases}
\displaystyle
\frac{(N-2)\tau+3N-12}{2(N-4)(2+\tau)},&N=5,6,\\[2mm]
\displaystyle
\frac{2\bigl(N+\tau(N-2)\bigr)}
{(N-2)(2+\tau)(N-4)},&N\ge7.
\end{cases}
\end{align*}
They are positive precisely by the lower bound on $\tau$. These last two
cases also give \eqref{eq:contraction-exponent-margins}: when $N=5,6$ one
has $2^*-2\ge1$, while for $N\ge7$ one has $2^*-2<1$.

The remaining zero-loss margins are
\[
\frac{N-\theta_i}{N-4}-\alpha_i=\frac1{N-4},\qquad
\alpha_i-\frac{N-2-\theta_i}{N-2}
=\frac{2(N-3)}{(N-4)(N-2)},
\]
\[
\tilde\alpha_i+\frac{\tilde\theta_i}{N-4}-1
=\frac4{N-4},\qquad
\tilde\alpha_1+\frac{\tilde\theta_1}{N-2}
-\frac{N-2}{2(N-4)}
=\frac{N-2}{2(N-4)}.
\]
The first inequality in \eqref{eq:all-barrier-margins} is automatic at zero
loss for $N=5,6$; for $N\ge7$ it reduces to
$\tau<(N+8)/(N-6)$. Also,
$(N-4)\tilde\theta+\bar\theta<2$ is strict at zero loss.

Thus every required inequality is strict at the zero-loss choice. By
continuity, choose
$\vartheta,\theta',\tilde\theta,\bar\theta>0$ sufficiently small, with
$0<\theta'<\vartheta$, so that all these inequalities remain strict. Next
choose $\delta_*>0$ sufficiently small so that
$\gamma_0<\gamma_1<2\gamma_0$ and all four entries in
\eqref{eq:self-map-exponent-margins} remain larger than $\gamma_0$. Finally,
choose $\sigma_0>0$ smaller than the least of the resulting four gaps.
\end{proof}

\section{Parameter dependence of the auxiliary operators}

The continuity statement in Proposition~\ref{prop:parameter-continuity}
requires two facts. First, the moment subspaces in the projected linear
problems vary with the parameters, so the equations and the moment conditions
must be combined into an operator on fixed Banach spaces. Second, the inner
balls in the exterior problem move with the centers, so prescribed inner data
must be transported before the corresponding minimizers can be compared. All
statements below are for each fixed sufficiently small $\epsilon>0$.

Let
\[
\mathbf t=(\lambda_1,\xi_1,\lambda_2,\xi_2)\in O_\delta.
\]
We write $L_i(\mathbf t)$ for the operators in
Propositions~\ref{prop:linear-first} and~\ref{prop:linear-second}, and
$S_{\mathbf t}$ for the outer operator in Definition~\ref{defS}. We use
$\Lambda_\epsilon(\mathbf t)$ when the parameter dependence of the
admissible class must be displayed. Functions
supported in $P_2(\mathbf t)$ or $Q_2(\mathbf t)$ are extended by zero to
$\Omega_\epsilon$ when two parameter values are compared. For
$h\in L^\infty(\Omega_\epsilon)$, let
\[
\mathbf a_1(\mathbf t;h):=(c_0,\ldots,c_N),
\qquad
\mathbf a_2(\mathbf t;h):=(d_0,\ldots,d_N)
\]
denote the multiplier vectors in the two linear problems.

\begin{Lem}\label{lem:linear-parameter-dependence}
Fix $\epsilon>0$ sufficiently small, a compact set $K\Subset O_\delta$, and
$q>N$. Let $\mathbf t_m\to\mathbf t$ in $K$. If $f_m\to f$ in
$L^q(\Omega_\epsilon)$ and $(\phi_m,\mathbf a_m)$ is the solution and
multiplier vector of the $i$th projected linear problem obtained by replacing
the localized right-hand side in \eqref{eq:linear-first} or
\eqref{eq:linear-second} by $f_m$, with parameter $\mathbf t_m$, then
\begin{equation}\label{eq:linear-sequential-parameter}
\phi_m\longrightarrow\phi
\quad\text{in }W^{2,q}(\Omega_\epsilon),
\qquad
\mathbf a_m\longrightarrow\mathbf a,
\end{equation}
where $(\phi,\mathbf a)$ is the corresponding pair for parameter $\mathbf t$
and source $f$.

In particular, for the original operators with sources cut off to the moving
balls,
\begin{equation}\label{eq:linear-parameter-operator-norm}
\sup_{\|h\|_{L^\infty}\le1}
\left(
\|L_i(\mathbf t_m)h-L_i(\mathbf t)h\|_{C^1(\overline{\Omega_\epsilon})}
+|\mathbf a_i(\mathbf t_m;h)-\mathbf a_i(\mathbf t;h)|
\right)\longrightarrow0.
\end{equation}
\end{Lem}
\begin{proof}
Set
\[
X_q:=W^{2,q}(\Omega_\epsilon)\cap W_0^{1,q}(\Omega_\epsilon).
\]
For $i=1,2$, combine the equation, multipliers, and moment conditions in the
fixed-space operator
\begin{equation}\label{eq:mixed-linear-operator}
\begin{aligned}
\mathcal L_{i,q}(\mathbf t):X_q\times\mathbb R^{N+1}
&\longrightarrow L^q(\Omega_\epsilon)\times\mathbb R^{N+1},\\
(\phi,(a_j))&\longmapsto
\left(
\begin{aligned}
&-\Delta\phi-\epsilon^{\frac{N-2}{N-4}}\phi
-(2^*-1)V_i(\mathbf t)^{2^*-2}\phi\\
&\quad-\sum_{j=0}^Na_j
U_i(\mathbf t)^{2^*-2}Y_{i,j}(\mathbf t),
\end{aligned}
\ (\ell_{i,j}(\mathbf t)(\phi))_{j=0}^N
\right).
\end{aligned}
\end{equation}
It is a Fredholm operator of index zero, since it is a compact
perturbation of the isomorphism
$(\phi,(a_j))\mapsto(-\Delta\phi,(a_j))$. Its kernel is trivial: the last $N+1$ components put $\phi$ in
the moving moment space, and uniqueness of the homogeneous projected problem
then gives $\phi=0$ and $a_j=0$. Thus
$\mathcal L_{i,q}(\mathbf t)$ is an isomorphism.

For fixed $\epsilon$, the projected bubbles, localized kernel functions, and
moment functionals depend continuously on $\mathbf t$ in the operator norms
in \eqref{eq:mixed-linear-operator}. Hence
\[
\mathcal L_{i,q}(\mathbf t_m)\longrightarrow
\mathcal L_{i,q}(\mathbf t)
\]
in operator norm. Continuity of inversion gives
\begin{equation}\label{eq:inverse-linear-parameter}
\|\mathcal L_{i,q}(\mathbf t_m)^{-1}
-\mathcal L_{i,q}(\mathbf t)^{-1}\|\longrightarrow0.
\end{equation}
Applying these inverses to $(f_m,0)$ proves
\eqref{eq:linear-sequential-parameter}.

For the second assertion, put
\[
B_1^{\rm src}(\mathbf t):=P_2(\mathbf t),
\qquad
B_2^{\rm src}(\mathbf t):=Q_2(\mathbf t).
\]
Since the radii are fixed when $\epsilon$ is fixed and the centers vary
continuously,
\[
|B_i^{\rm src}(\mathbf t_m)\triangle
B_i^{\rm src}(\mathbf t)|\longrightarrow0.
\]
Therefore
\begin{equation}\label{eq:moving-source-continuity}
\sup_{\|h\|_{L^\infty}\le1}
\|\mathbf1_{B_i^{\rm src}(\mathbf t_m)}h
-\mathbf1_{B_i^{\rm src}(\mathbf t)}h\|_{L^q}
\le |B_i^{\rm src}(\mathbf t_m)\triangle
B_i^{\rm src}(\mathbf t)|^{1/q}\longrightarrow0.
\end{equation}
Combining \eqref{eq:inverse-linear-parameter} and
\eqref{eq:moving-source-continuity}, and using
$X_q\hookrightarrow C^1(\overline{\Omega_\epsilon})$, proves
\eqref{eq:linear-parameter-operator-norm}.

The estimate \eqref{eq:moving-source-continuity} is the reason no restriction
to sources supported in the intersection of two nearby moving balls is
needed here. Characteristic functions do not converge in $L^\infty$, but they
do converge as operators from $L^\infty$ to $L^q$ for fixed $\epsilon$. This
argument is used only for parameter continuity at fixed $\epsilon$; no
uniform $W^{2,q}$ estimate as $\epsilon\to0$ is asserted. The smooth cutoffs
$\chi_i(\mathbf t)$ do not enter the inverse
$\mathcal L_{i,q}(\mathbf t)^{-1}$; they occur later in the localized source,
where their $C^2$ dependence on $\mathbf t$ is used directly.
\end{proof}

\medskip

\begin{Lem}\label{lem:outer-parameter-dependence}
Fix $\epsilon>0$ sufficiently small, a compact set $K\Subset O_\delta$, a
reference configuration $\mathbf t^0\in K$, and numbers
$0<\rho<\rho'<1$. For $\mathbf t\in K$, set
\[
\Lambda_\epsilon^\rho(\mathbf t):=
\left\{(g_1,g_2)\in\Lambda_\epsilon(\mathbf t):
\|g_1\|_{L^\infty},\|g_2\|_{L^\infty}
\le\rho\epsilon^{\gamma_0}\right\}.
\]
There are a neighborhood $\mathcal U$ of $\mathbf t^0$, $C^2$
diffeomorphisms
$\Phi_{\mathbf t}:\Omega_\epsilon\to\Omega_\epsilon$, and comparison maps
\[
J_{\mathbf t}:\Lambda_\epsilon^\rho(\mathbf t^0)
\longrightarrow\Lambda_\epsilon^{\rho'}(\mathbf t),
\qquad \mathbf t\in\mathcal U,
\]
such that $\Phi_{\mathbf t^0}=\operatorname{Id}$,
$J_{\mathbf t^0}=\operatorname{Id}$, and the following holds. If
$\mathbf t_m\to\mathbf t^0$ and $g\in
\Lambda_\epsilon^\rho(\mathbf t^0)$, then
\begin{equation}\label{eq:transported-data-convergence}
J_{\mathbf t_m}g\longrightarrow g
\quad\text{uniformly in }\Omega_\epsilon,
\end{equation}
and
\begin{equation}\label{eq:outer-parameter-uniform}
\|S_{\mathbf t_m}(J_{\mathbf t_m}g)-S_{\mathbf t^0}(g)\|_{L^\infty(\Omega_\epsilon)^2}
\longrightarrow0.
\end{equation}
Moreover, after pullback by $\Phi_{\mathbf t_m}$, the first components
converge in $C^1$ on every compact subset of
$\Omega_\epsilon\setminus\overline{P_1(\mathbf t^0)}$, and the second
components converge in $C^1$ on every compact subset of
$\Omega_\epsilon\setminus\overline{Q_1(\mathbf t^0)}$.
\end{Lem}
\begin{proof}
For $\mathbf t$ near $\mathbf t^0$, choose $\Phi_{\mathbf t}$ so that it maps
$P_1(\mathbf t^0)$ and $Q_1(\mathbf t^0)$ onto
$P_1(\mathbf t)$ and $Q_1(\mathbf t)$, agrees with the corresponding
translations on neighborhoods containing $P_2(\mathbf t^0)$ and
$Q_2(\mathbf t^0)$, is the identity outside slightly larger disjoint
neighborhoods, and depends continuously on $\mathbf t$ in $C^2$.
Transport the data by
\[
\widetilde g_{i,\mathbf t}(y)
:=g_i(\Phi_{\mathbf t}^{-1}(y)).
\]
This does not preserve the new moment conditions exactly. Let
\[
M_i(\mathbf t)_{j\ell}:=
\int_{\Omega_\epsilon}
U_i(\mathbf t)^{2^*-2}Y_{i,j}(\mathbf t)Y_{i,\ell}(\mathbf t),
\qquad 0\le j,\ell\le N.
\]
The matrices $M_i(\mathbf t)$ are positive definite and uniformly invertible
for $\mathbf t\in K$. Choose coefficients
$a_{i,\ell}(\mathbf t;g)$ so that
\[
(J_{\mathbf t}g)_i
:=\widetilde g_{i,\mathbf t}
-\sum_{\ell=0}^Na_{i,\ell}(\mathbf t;g)Y_{i,\ell}(\mathbf t)
\]
satisfies the new moment conditions. Since $g$ satisfies the reference
conditions and $\Phi_{\mathbf t}\to\operatorname{Id}$ in $C^2$,
\[
|a_{i,\ell}(\mathbf t;g)|
\le o_{\mathbf t\to\mathbf t^0}(1)\|g_i\|_{L^\infty}.
\]
Thus \eqref{eq:transported-data-convergence} holds, and the strict margin
$\rho<\rho'$ gives
$J_{\mathbf t}g\in\Lambda_\epsilon^{\rho'}(\mathbf t)$ for $\mathbf t$
close to $\mathbf t^0$.

Suppose first that $g$ is energy compatible and smooth. Let
\[
(u_{1,m},u_{2,m})=(V_{1,\mathbf t_m},V_{2,\mathbf t_m})
+S_{\mathbf t_m}(J_{\mathbf t_m}g)
\]
be the corresponding minimizers. Pullback by $\Phi_{\mathbf t_m}$ places the
affine minimization classes on fixed exterior domains. After subtracting
smooth lifts of the prescribed inner data, these classes have fixed zero
boundary conditions, and the lifts converge in $H^1$. The truncated
functionals converge on bounded sets, while the coercivity estimate in
Lemma~\ref{lem3} is uniform for parameters in $K$. Weak compactness and lower
semicontinuity therefore give a limiting minimizer. Conversely, transporting
the limiting minimizer and adding the small lift needed to match the new
inner data gives a recovery sequence, so the minimum values also converge.
Lemma~\ref{lem:outer-unique} identifies the weak limit with the minimizer at
$\mathbf t^0$; convergence of the Dirichlet energies then gives strong
$H_0^1$ convergence.

For fixed $\epsilon$, the truncated right-hand sides are uniformly bounded
for $\mathbf t\in K$. After pullback and subtraction of the same smooth
lifts, standard boundary $W^{2,q}$ estimates on the fixed exterior domains
give bounds independent of $m$ for every finite $q$. Compactness and the
preceding identification of the weak limit imply uniform convergence of the
full pairs and $C^1$ convergence on compact subsets of the exterior regions.
This proves \eqref{eq:outer-parameter-uniform} for smooth
energy-compatible data.

For general $g\in\Lambda_\epsilon^\rho(\mathbf t^0)$, choose smooth
energy-compatible $g^{(\ell)}\in\Lambda_\epsilon^\rho(\mathbf t^0)$
converging uniformly to $g$, using the density argument preceding
Definition~\ref{defS}. The construction of $J_{\mathbf t}$ is uniformly
continuous on bounded sets in the relevant inner norms. Lemma~\ref{lem9}
therefore makes the outer solutions corresponding to $g^{(\ell)}$ and $g$
uniformly close, uniformly for all sufficiently large $m$. Apply the
preceding paragraph for each fixed $\ell$, and then let $\ell\to\infty$.
The gradient estimate in Lemma~\ref{lem9} gives the same approximation on the
cutoff annuli, and local $W^{2,q}$ estimates give the stated convergence on
the remaining compact subsets.
\end{proof}

\medskip

\section*{Acknowledgments}
This research was supported by the National Natural Science Foundation of
China (Grant Nos. 12271539 and 12371107).

\end{document}